\documentclass[a4paper,12pt,reqno]{amsart}
\usepackage[T1]{fontenc}
\usepackage[utf8]{inputenc}
\usepackage[libertinus,vvarbb]{newtx}
\usepackage[margin=1in]{geometry}
\usepackage{enumitem}
\usepackage{graphicx}
\usepackage[dvipsnames]{xcolor}
\usepackage{tikz}
\usepackage{xparse}
\usepackage{xurl}
\usepackage[bookmarksdepth=2]{hyperref}

\hypersetup{colorlinks=true,urlcolor=MidnightBlue,linkcolor=MidnightBlue,citecolor=MidnightBlue}

\usepackage[normalem]{ulem}

\newcommand{\stkout}[1]{\ifmmode\text{\sout{\ensuremath{#1}}}\else\sout{#1}\fi}
\newcommand{\crsout}[1]{\ifmmode\text{\xout{\ensuremath{#1}}}\else\xout{#1}\fi}

\usetikzlibrary{decorations.markings,decorations.pathmorphing,nfold}
\makeatletter
\tikzset{
 side by side path/.style args={#1:#2}{
  postaction={path only,draw=#1,offset=+.5\pgflinewidth},
  postaction={path only,draw=#2,offset=+-.5\pgflinewidth}
 },
 side by side/.style={path only,side by side path={#1}},
 offset/.code=\tikz@addoption{%
  \pgfgetpath\tikz@temp%
  \pgfsetpath\pgfutil@empty%
  \pgfoffsetpath\tikz@temp{#1}%
 }
}
\makeatother

\numberwithin{equation}{section}

\newtheorem{theorem}{Theorem}[section]
\newtheorem{lemma}[theorem]{Lemma}

\newtheorem{proposition}[theorem]{Proposition}

\theoremstyle{definition}
\newtheorem{definition}[theorem]{Definition}

\newtheorem{remark}[theorem]{Remark}

\newtheorem*{assumption}{Assumption}

\DeclareMathOperator{\re}{Re}
\DeclareMathOperator{\im}{Im}
\let\arg\relax
\DeclareMathOperator{\arg}{Arg}

\DeclareMathOperator{\arccot}{arccot}
\DeclareMathOperator{\sign}{sign}
\DeclareMathOperator{\supp}{supp}
\DeclareMathOperator{\esssupp}{ess\,supp}

\DeclareMathOperator{\Cl}{Cl}
\DeclareMathOperator{\Int}{Int}

\newcommand{\ind}{\mathbb{1}}

\newcommand{\pr}{\mathbb{P}}
\newcommand{\ex}{\mathbb{E}}
\newcommand{\C}{\mathbb{C}}
\newcommand{\R}{\mathbb{R}}
\newcommand{\Z}{\mathbb{Z}}

\newcommand{\disk}{\mathbb{D}}
\newcommand{\torus}{\mathbb{T}}
\newcommand{\hp}{\mathbb{H}}

\newcommand{\ph}{\varphi}
\newcommand{\eps}{\varepsilon}
\newcommand{\thet}{\vartheta}
\newcommand{\cinfty}{\bar \infty}
\newcommand{\pvint}{\operatorname{pv}\int}

\newcommand{\mass}{R}
\newcommand{\limit}{L}
\newcommand{\ev}{v}
\newcommand{\evr}{w}
\newcommand{\levy}{J}

\mathchardef\mathhyphen="2D

\newcommand{\cm}{\mathscr{C}\mspace{-4mu}\mathscr{M}}
\newcommand{\am}{\mathscr{A}\mspace{-4mu}\mathscr{M}}
\newcommand{\amcm}{\am\mspace{-2mu}/\mspace{-1mu}\cm}
\newcommand{\amdcm}{\am\mspace{-2mu}\mathhyphen\mspace{-1mu}\cm}

\renewcommand{\le}{\leqslant}
\renewcommand{\ge}{\geqslant}

\newcommand{\lv}{\lvert}
\newcommand{\rv}{\rvert}

\newcommand{\ul}{\underline}
\newcommand{\ol}{\overline}

\NewDocumentCommand{\formula}{ssom}{%
 \IfBooleanTF{#1}{%
  \IfBooleanTF{#2}{%
   \IfValueTF{#3}%
    {\begin{align}\label{#3}\begin{gathered}#4\end{gathered}\end{align}}%
    {\begin{gather}#4\end{gather}}%
  }{%
   \IfValueTF{#3}%
    {\begin{align}\label{#3}\begin{aligned}#4\end{aligned}\end{align}}%
    {\begin{gather*}#4\end{gather*}}%
  }%
 }{%
  \IfValueTF{#3}%
   {\begin{align}\label{#3}#4\end{align}}%
   {\begin{align*}#4\end{align*}}%
 }%
}

\begin{document}

\title[Generalised eigenvector expansion of infinite Toeplitz matrices]{Generalised eigenvector expansion of infinite Toeplitz matrices with absolutely/completely monotone entries}
\author{Mateusz Kwaśnicki, Jacek Wszoła}
\thanks{Work supported by the National Science Centre, Poland, grant no.\@ 2023/49/B/ST1/04303}
\address{Mateusz Kwaśnicki, Jacek Wszoła \\ Department of Analysis and Stochastic Processes \\ Wrocław University of Science and Technology \\ Wybrzeże Wyspiańskiego 27 \\ 50-370 Wrocław, Poland}
\email{\textsf{\href{mailto:mateusz.kwasnicki@pwr.edu.pl}{mateusz.kwasnicki@pwr.edu.pl}, \href{mailto:jacek.wszola@pwr.edu.pl}{jacek.wszola@pwr.edu.pl}}}
\date{\today}

\begin{abstract}
We study the spectral theory of infinite Toeplitz matrices $T = (a_{k - l})$ under the assumption that $(a_k)$ and $(a_{-k})$ are completely monotone sequences. We derive expressions for generalised eigenvectors and prove a generalised eigenvector expansion of $T$. Even if the matrix $T$ is not normal, our expressions involve only eigenvalues and eigenvectors with real entries.
\end{abstract}

\maketitle

%
%

\section{Introduction}
\label{sec:intro}


\subsection{Main results}

This paper provides a detailed analysis of the Wiener--Hopf factorisation for infinite Toeplitz matrices
\formula{
 T & = (a_{k - l} : k, l \ge 0) = \begin{pmatrix}
  a_0 & a_{-1} & a_{-2} & a_{-3} & \cdots \\
  a_1 & a_0 & a_{-1} & a_{-2} & \cdots \\
  a_2 & a_1 & a_0 & a_{-1} & \cdots \\
  a_3 & a_2 & a_1 & a_0 & \cdots \\
  \vdots & \vdots & \vdots & \vdots & \ddots
 \end{pmatrix} ,
}
where the sequences $(a_k : k \ge 1)$ and $(a_{-k} : k \ge 1)$ are completely monotone. Under relatively minor further assumptions on the sequence $(a_k : k \in \Z)$, our main result, Theorem~\ref{thm:main}, provides a generalised eigenvector expansion for $T$. Surprisingly, although $T$ is not assumed to be self-adjoint or normal, our expression only involves \emph{real} generalised eigenvalues, eigenvectors and co-eigenvectors. 

Our motivation is three-fold:
\begin{itemize}
\item On the probability side, our interest in the spectral theory of $T$ stems from the fluctuation theory of random walks. We aim to develop integral formulae analogous to those obtained recently for the continuous counterparts of Toeplitz matrices and random walks, the Wiener--Hopf operators and Lévy processes; see~\cite{k11,kmr13} for self-adjoint operators and symmetric processes, and~\cite{k19,k25} for the general case. We plan to apply the generalised eigenvector expansion to provide a semi-explicit expression and asymptotic estimates for the distribution of maxima and minima of the random walk $X_n$ with increments distributed according to $(a_k : k \in \Z)$. Formula~\eqref{eq:infimum} is a preliminary result in this direction.
\item The Wiener--Hopf factorisation of infinite Toeplitz matrices $T$ is a step towards understanding spectral properties of finite Toeplitz matrices; see, for example, \cite{bbg23,bg23,bbgm17,dgk09}. Our goal is to study the asymptotic behaviour of large eigenvalues of finite matrices $(a_{k - l} : 0 \le k, l < N)$ as $N \to \infty$. We plan to approximate the corresponding eigenvectors and co-eigenvectors using the generalised eigenvectors and co-eigenvectors of $T$, following the approach successfully applied for Lévy operators in~\cite{kkm13,kkm16,k12} and a class of Toeplitz matrices in~\cite{bg23}.
\item Finally, from the point of view of functional analysis and spectral theory, infinite Toeplitz matrices are interesting and important examples of non-normal operators, and exploring their spectral properties is therefore an interesting problem on its own.
\end{itemize}
A rigorous statement of our main results requires auxiliary definitions that we now discuss.

Throughout this work, $(a_k : k \in \Z)$ denotes a two-sided summable \emph{`$\amcm$ sequence'.} Specifically, we assume that $(a_k : k \ge 1)$ and $(a_{-k} : k \ge 1)$ are completely monotone, the central term $a_0 \in \R$ is arbitrary, and the entire sequence satisfies $\sum_{k = -\infty}^\infty \lv a_k \rv < \infty$. Recall that if $\Delta^n a_k = \sum_{j = 0}^n \tbinom{n}{j} (-1)^{n - j} a_{k + j}$ is the $n$th iterated forward difference operator, then a sequence $(a_k : k \ge 1)$ is completely monotone if $(-1)^n \Delta^n a_k \ge 0$ for every $n \ge 0$ and $k \ge 1$. Furthermore, if $(a_{-k})$ is completely monotone, then $(a_k)$ is said to be absolutely monotone.

The value of $a_0$ is irrelevant to the spectral properties of $T$. In order to simplify the statements of our main results, in the introduction we often assume that
\formula[eq:a0:zero]{
 a_0 & = -\sum_{k \in \Z \setminus \{0\}} a_k .
}
The general case can be easily reduced to this setting by considering the operator $T - \alpha$, where $\alpha = \sum_{k = -\infty}^\infty a_k$.

Theorem~\ref{thm:main} requires a certain balance between the elements of $(a_k : k \in \Z)$ with positive and negative indices. This condition is expressed in terms of the generating function
\formula{
 \hat a(z) & = \sum_{k = -\infty}^\infty a_k z^k ,
}
defined initially when $\lv z \rv = 1$. Let us denote
\formula[eq:i:symbol]{
 F(z) & = \hat a(1) - \hat a(z) = \sum_{k = -\infty}^\infty a_k (1 - z^k) .
}
This definition is independent of the value of $a_0$. With the choice given in~\eqref{eq:a0:zero}, we have $\hat a(1) = \sum_{k = -\infty}^\infty a_k = 0$, and so $F(z) = -\hat a(z)$.

In Section~\ref{sec:symbol}, we observe that $F$ extends to a holomorphic function on $\C \setminus [0, \infty)$, which we denote by the same symbol $F$. Then, in Section~\ref{sec:spine}, we prove that the curve
\formula[eq:i:spine]{
 \Gamma & = \{z \in \C : \im z > 0, \, F(z) \in \R\}
}
can be parameterised as $\Gamma = \{\gamma(\lambda) : \lambda \in \Lambda\}$, where $\Lambda$ is a bounded open subset of $(0, \infty)$, and $F(\gamma(\lambda)) = \lambda$. We call $\Gamma$ the \emph{spine} of $(a_k : k \in \Z)$. While in many important cases $\Lambda$ is an interval and $\Gamma$ is a single arc, in general $\Gamma$ consists of a finite or countably infinite number of arcs that begin and end on $(0, \infty)$, and at most one arc that begins on $(0, \infty)$ and ends on $(-\infty, 0] \cup \{\infty\}$. In Theorem~\ref{thm:main}, we assume that this last arc in fact ends on $(-\infty, 0)$, so in this case we say that $\Gamma$ \emph{winds around $0$}.

In our generalised eigenfunction expansion of the infinite Toeplitz matrix $T$, the set $\Lambda$ plays the role of the generalised spectrum of $-T$. For a generalised eigenvalue $\lambda \in \Lambda$, we denote the corresponding generalised eigenvector of $-T$ by
\formula{
 \ev^+_\lambda & = (\ev^+_{\lambda, k} : k \ge 0) ,
}
and the generalised co-eigenvector of $-T$ by
\formula{
 \ev^-_\lambda & = (\ev^-_{\lambda, l} : l \ge 0) .
}
An explicit formula for $\ev^+_{\lambda, k}$ and $\ev^-_{\lambda, k}$ requires a number of further auxiliary definitions, and we postpone it to Section~\ref{sec:intro:eig}. Here we mention that
\formula[]{
\label{eq:ev:p}
 \ev^+_{\lambda, k} & = \lv \gamma(\lambda) \rv^{-k - 1} \sin\bigl((k + 1) \thet_\lambda + \thet^+_\lambda\bigr) - \evr^+_{\lambda, k} , \\
\label{eq:ev:m}
 \ev^-_{\lambda, l} & = \lv \gamma(\lambda) \rv^{l + 1} \sin\bigl((l + 1) \thet_\lambda + \thet^-_\lambda\bigr) - \evr^-_{\lambda, l} ,
}
where $(\evr^+_{\lambda, k} : k \ge 0)$, $(\evr^-_{\lambda, l} : l \ge 0)$ are completely monotone sequences and $\thet_\lambda, \thet^+_\lambda, \thet^-_\lambda \in [0, \pi)$.

The following is our main result. The proof is given in Section~\ref{sec:main}.

\begin{theorem}
\label{thm:main}
Suppose that $(a_k : k \in \Z)$ is a summable $\amcm$ sequence, $\sum_{k = -\infty}^\infty a_k = 0$, and that the spine $\Gamma$ of $(a_k : k \in \Z)$ winds around~$0$: one of the arcs of $\Gamma$ begins on $(0, \infty)$ and ends on $(-\infty, 0)$. Denote by $T = (a_{k - l} : k, l \ge 0)$ the corresponding infinite Toeplitz matrix. Then for every polynomial $P$, the entries of the matrix $P(T)$ are given by
\formula{
 (P(T))_{k, l} & = \frac{2}{\pi} \int_{\Lambda} P(-\lambda) \ev^+_{\lambda, k} \ev^-_{\lambda, l} \, \frac{\lv \gamma'(\lambda) \rv}{\lv \gamma(\lambda) \rv} \, d\lambda ,
}
where $k, l \ge 0$.
\end{theorem}

The case $P(\lambda) = 1$ is of particular interest in Theorem~\ref{thm:main}: for $k, l \ge 0$, the functions $\lambda \mapsto \ev^+_{\lambda, k}$ and $\lambda \mapsto \ev^-_{\lambda, l}$ satisfy the orthogonality relation
\formula{
 \frac{2}{\pi} \int_\Lambda \ev^+_{\lambda, k} \ev^-_{\lambda, l} \, \frac{\lv \gamma'(\lambda) \rv}{\lv \gamma(\lambda) \rv} \, d\lambda & = \delta_{k - l} ,
}
where $\delta_k = 0$ if $k \ne 0$ and $\delta_0 = 1$.

Observe that the sequences $\ev^+_\lambda$ and $\ev^-_\lambda$ may grow exponentially fast, and so the series defining $T \ev^+_\lambda$ and $T^* \ev^-_\lambda$ may fail to converge. In this sense $\ev^+_\lambda$ and $\ev^-_\lambda$ need not be true eigenvectors and co-eigenvectors of $T$. However, our next result shows that this lack of summability is essentially the only obstacle. The proof is given in Section~\ref{sec:eig}.

\begin{theorem}
\label{thm:eig}
Suppose that the assumptions of Theorem~\ref{thm:main} are satisfied, and let $\lambda \in \Lambda$. If
\formula[eq:eig:sum:p]{
 \sum_{l = 0}^\infty \frac{a_{-l}}{\lv \gamma(\lambda) \rv^l} & < \infty ,
}
then the sequence $\ev^+_\lambda$ is an eigenvector of\, $T$: for every $k \ge 0$,
\formula{
 \sum_{l = 0}^\infty a_{k - l} \ev^+_{\lambda, l} & = -\lambda \ev^+_{\lambda, k} .
}
Similarly, if
\formula[eq:eig:sum:m]{
 \sum_{k = 0}^\infty a_k \lv \gamma(\lambda) \rv^k & < \infty ,
}
then the sequence $\ev^-_\lambda$ is a co-eigenvector of\, $T$: for every $l \ge 0$,
\formula{
 \sum_{k = 0}^\infty a_{k - l} \ev^-_{\lambda, k} & = -\lambda \ev^-_{\lambda, l} .
}
\end{theorem}

If $\lv \gamma(\lambda) \rv \ge 1$, then~\eqref{eq:eig:sum:p} holds. Similarly, if $\lv \gamma(\lambda) \rv \le 1$, then~\eqref{eq:eig:sum:m} is satisfied. Hence, for each $\lambda \in \Lambda$, exactly one of the following holds:
\begin{itemize}
\item $\lv \gamma(\lambda) \rv = 1$, both $\ev^+_\lambda$ and $\ev^-_\lambda$ are bounded sequences, and both are (co)-eigenvectors;
\item $\lv \gamma(\lambda) \rv > 1$, $\ev^+_\lambda$ is summable and is an eigenvector, while $\ev^-_\lambda$ grows exponentially fast;
\item $\lv \gamma(\lambda) \rv < 1$, $\ev^-_\lambda$ is summable and is a co-eigenvector, while $\ev^+_\lambda$ grows exponentially fast.
\end{itemize}
If $a_k$ converges to zero as $k \to \infty$ or $k \to -\infty$ faster than an appropriate geometric sequence, then it is possible that both $\ev^+_\lambda$ and $\ev^-_\lambda$ are (co)-eigenvectors even if $\lv \gamma(\lambda) \rv \ne 1$. However, in many interesting cases the decay of $a_k$ is slower than geometric.


\subsection{Generalised eigenvectors and co-eigenvectors}
\label{sec:intro:eig}

Below we describe $\ev^+_\lambda$ and $\ev^-_\lambda$ in more detail. We continue to assume that $(a_k : k \in \Z)$ is a summable $\amcm$ sequence. All notions introduced below are defined in terms of $F$, and so they do not depend on the value of $a_0$.

Recall that the spine $\Gamma$ of $(a_k : k \in \Z)$ is parameterised by $\gamma(\lambda)$, where $\lambda \in \Lambda$. Below we let $\lambda \in \Lambda$ be fixed, and, for brevity, we write $\gamma$ instead of $\gamma(\lambda)$. Let
\formula[eq:i:fr]{
 F(\lambda; z) & = \frac{\lv \gamma \rv (z - 1)^2}{(\gamma - z) (\overline{\gamma} - z)} \, (F(z) - \lambda) ,
}
and define the corresponding Wiener--Hopf factors
\formula[]{
\label{eq:i:fr:wh:p}
 F^+(\lambda; z) & = c^+_\lambda \exp\biggl(\frac{1}{\pi} \int_0^1 \frac{1}{s - z^{-1}} \, \lv \arg(F(s^{-1} + 0 i) - \lambda) \rv ds \biggr) , \\
\label{eq:i:fr:wh:m}
 F^-(\lambda; z) & = c^-_\lambda \exp\biggl(\frac{1}{\pi} \int_0^1 \frac{1}{s - z^{-1}} \, \lv \arg(F(s + 0 i) - \lambda) \rv ds \biggr) .
}
Here $F(s + 0 i)$ denotes the boundary limit of $F$ approached from above: for almost every $s \in (0, \infty)$,
\formula{
 F(s + 0 i) & = \lim_{t \to 0^+} F(s + i t) ,
}
and $c^+_\lambda, c^-_\lambda > 0$ are constants such that
\formula{
 F^+(\lambda; z) F^-(\lambda; z^{-1}) & = F(\lambda; z) .
}
If $0 < x < \min\{\lv \gamma \rv, 1\}$ and $0 < y < \min\{\lv \gamma \rv^{-1}, 1\}$, then the generating functions of $\ev^+_\lambda$ and $\ev^-_\lambda$ are given by 
\formula{
 \sum_{k = 0}^\infty \ev^+_{\lambda, k} x^k & = \frac{\lv F^+(\lambda; \gamma) \rv \im \gamma}{\lv 1 - \gamma \rv} \, \frac{1 - x}{\lv \gamma - x \rv^2} \, \frac{1}{F^+(\lambda; x)} \, , \\
 \sum_{l = 0}^\infty \ev^-_{\lambda, l} y^l & = \frac{\lv F^-(\lambda; \gamma^{-1}) \rv \im \gamma}{\lv \gamma \rv \lv 1 - \gamma \rv} \, \frac{1 - y}{\lv \gamma^{-1} - y \rv^2} \, \frac{1}{F^-(\lambda; y)} \, .
}

For a more explicit description of $\ev^+_\lambda$ and $\ev^-_\lambda$, we denote
\formula[eq:i:theta]{
 \thet_\lambda & = \arg \gamma , &
 \thet^+_\lambda & = -\arg \frac{1 - \gamma}{F^+(\lambda; \gamma)} , &
 \thet^-_\lambda & = \arg \frac{1 - \gamma^{-1}}{F^-(\lambda; \gamma^{-1})} .
}
From the definition of $F^+(\lambda; z)$ and $F^-(\lambda; z)$, by a direct calculation, we find that
\formula[]{
\label{eq:i:theta:p}
 \thet^+_\lambda & = \frac{1}{\pi} \int_0^1 \im \frac{-1}{s - \gamma^{-1}} \, \bigl(\pi - \lv \arg(F(s^{-1} + 0 i) - \lambda) \rv \bigr) ds , \\
\label{eq:i:theta:m}
 \thet^-_\lambda & = \frac{1}{\pi} \int_0^1 \im \frac{1}{s - \gamma} \, \bigl(\pi - \lv \arg(F(s + 0 i) - \lambda) \rv\bigr) ds .
}
In Section~\ref{sec:main}, we prove that $\thet_\lambda \in (0, \pi)$, $\thet^+_\lambda, \thet^-_\lambda \in [0, \pi)$, and
\formula{
 \thet_\lambda + \thet^+_\lambda + \thet^-_\lambda & \le \pi .
}
More precisely,
\formula{
 \thet^+_\lambda & \in [0, \pi - \arg(\gamma - 1)] , &
 \thet^-_\lambda & \in [0, \arg(\gamma - 1) - \thet_\lambda] .
}
We also introduce completely monotone sequences $(\evr^+_{\lambda, k} : k \ge 0)$ and $(\evr^-_{\lambda, l} : l \ge 0)$, determined by their generating functions: for $x, y \in \C$ such that $\lv x \rv \le 1$, $\lv y \rv \le 1$, we have
\formula{
 \sum_{k = 0}^\infty \evr^+_{\lambda, k} x^k & = \frac{1}{2 i} \biggl(\frac{e^{i \thet^+_\lambda}}{\overline{\gamma} - x} - \frac{e^{-i \thet^+_\lambda}}{\gamma - x}\biggr) - \frac{\lv F^+(\lambda; \gamma) \rv \im \gamma}{\lv 1 - \gamma \rv} \, \frac{(1 - x) F^-(\lambda; x^{-1})}{(\gamma - x) (\overline{\gamma} - x) F(\lambda; x)} \, , \\
 \sum_{l = 0}^\infty \evr^-_{\lambda, l} y^l & = \frac{1}{2 i} \biggl(\frac{e^{i \thet^-_\lambda}}{\gamma^{-1} - y} - \frac{e^{-i \thet^-_\lambda}}{\overline{\gamma}{}^{-1} - y}\biggr) - \frac{\lv F^-(\lambda; \gamma^{-1}) \rv \im \gamma}{\lv \gamma \rv \lv 1 - \gamma \rv} \, \frac{(1 - y) F^+(\lambda; y^{-1})}{(\overline{\gamma}{}^{-1} - y) (\gamma^{-1} - y) F(\lambda; y)} \, .
}
Thus, $\evr^+_{\lambda, k}$ and $\evr^-_{\lambda, l}$ can be written in an integral form using Cauchy's integral formula. We prove in Section~\ref{sec:main} that
\formula{
 0 & \le \evr^+_{\lambda, k} < \sin(\thet_\lambda + \thet^+_\lambda), & 0 & \le \evr^-_{\lambda, l} < \sin(\thet_\lambda + \thet^-_\lambda) ,
}
and
\formula{
 \sum_{k = 0}^\infty \evr^+_{\lambda, k} & \le \im \frac{e^{i \thet^+_\lambda}}{\overline{\gamma} - 1} \, , &
 \sum_{l = 0}^\infty \evr^-_{\lambda, l} & \le \im \frac{e^{i \thet^-_\lambda}}{\gamma^{-1} - 1} \, .
}
An integral expression that is more convenient than the one involving Cauchy's integral formula is available if $F$ has a continuous boundary limit $F(s + 0 i) \ne \lambda$ for $s \in (0, 1) \cup (1, \infty)$. Then a similar boundary limit $F(\lambda; s + 0 i) \ne 0$ exists, and we prove in Section~\ref{sec:main} that
\formula[eq:i:psi:p]{
 \evr^+_{\lambda, k} & = \frac{1}{\pi} \, \frac{\lv F^+(\lambda; \gamma) \rv \im \gamma}{\lv 1 - \gamma \rv} \int_0^1 s^k \, \frac{(1 - s) F^-(\lambda; s)}{\lv \gamma s - 1 \rv^2} \, \biggl|\im \frac{1}{F(\lambda; s^{-1} + 0 i)}\biggr| \, ds ,
}
and
\formula[eq:i:psi:m]{
 \evr^-_{\lambda, l} & = \frac{1}{\pi} \, \frac{\lv \gamma \rv \lv F^-(\lambda; \gamma^{-1}) \rv \im \gamma}{\lv 1 - \gamma \rv} \int_0^1 s^l \, \frac{(1 - s) F^+(\lambda; s)}{\lv s - \gamma \rv^2} \, \biggl|\im \frac{1}{F(\lambda; s + 0 i)}\biggr| \, ds .
}
In the general case, the above expressions are still valid, as long as we understand the boundary limit in the sense of vague convergence of measures on $(0, 1)$:
\formula[]{
\label{eq:i:boundary:p}
 \biggl|\im \frac{1}{F(\lambda; s^{-1} + 0 i)}\biggr| ds & = \lim_{t \to 0^+} \biggl|\im \frac{1}{F(\lambda; s^{-1} + i t)}\biggr| ds , \\
\label{eq:i:boundary:m}
 \biggl|\im \frac{1}{F(\lambda; s + 0 i)}\biggr| ds & = \lim_{t \to 0^+} \biggl|\im \frac{1}{F(\lambda; s + i t)}\biggr| ds .
}

With the above definitions, $\ev^+_\lambda$ and $\ev^-_\lambda$ are given by~\eqref{eq:ev:p} and~\eqref{eq:ev:m}.


\subsection{Discussion}

The spectrum of the infinite Toeplitz matrix $T$ is well understood when $(a_k)$ is merely a summable sequence. By Gohberg's theorem (see Theorem~1.17 in~\cite{bs99}), $z \in \C$ is in the essential spectrum of $T$ if and only if $z = \hat a(e^{i \thet})$, where $\thet \in [0, 2 \pi)$, and the spectrum of $T$ additionally contains those $z \in \C$ for which the winding number of $\hat a(e^{i \thet})$ ($\thet \in [0, 2 \pi)$) around $z$ is nonzero. Furthermore, this winding number coincides with the Fredholm index of $T - z$.

\begin{figure}
\centering
\footnotesize
\begin{tabular}{cc}
\raisebox{-0.5\height}{\includegraphics[scale=0.7,]{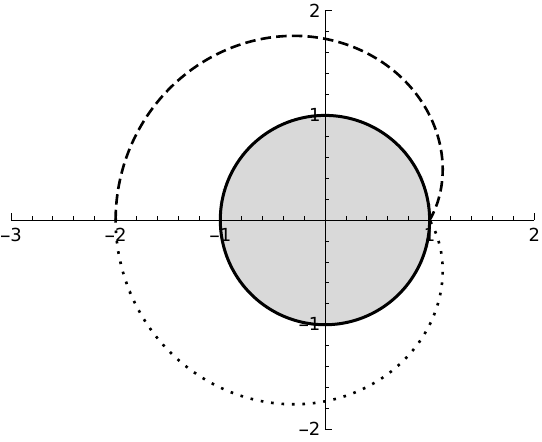}} & \raisebox{-0.5\height}{\includegraphics[scale=0.7]{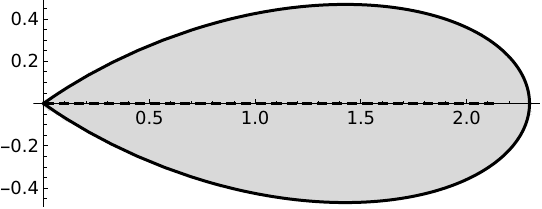}} \\
(a) & (b)
\end{tabular}
\caption{The spine $\Gamma$ and the generalised spectrum $\Lambda$ for the $\amcm$ sequence with generating function $\hat a(z) = -(1 - z)^\alpha (1 - z^{-1})^\beta$, where $\alpha = 0.4$ and $\beta = 0.8$: (a)~the spine $\Gamma$ (dashed line), its mirror image $\overline{\Gamma}$ (dotted line), the unit circle (solid line), and the unit disk (gray region);  (b)~their image under $-\hat a$: the generalised spectrum $\Lambda$ of $-T$ (dashed line on the real axis), the essential spectrum of $-T$ (solid line), and the spectrum of $-T$ (gray region).}
\label{fig:spectrum}
\end{figure}

Our Theorem~\ref{thm:eig} explicitly describes the eigenvectors and co-eigenvectors corresponding to \emph{real} elements of the spectrum of $T$, corresponding to the values of the symbol $F$ along the spine $\Gamma$. Namely, if $\lambda \in \Lambda$, then $\lambda = f(\gamma(\lambda))$ belongs to the spectrum of $-T$, and we have the following possibilities (see Figure~\ref{fig:spectrum}):
\begin{itemize}
\item if $\lv \gamma(\lambda) \rv = 1$, then $\lambda$ is in the essential spectrum of $-T$, and $\ev^+_\lambda$ and $\ev^-_\lambda$ are the corresponding `resonances' of $-T$ and the adjoint $-T^*$;
\item if $\lv \gamma(\lambda) \rv > 1$, then $\lambda$ is in the point spectrum of $-T$, with eigenvector $\ev^+_\lambda$;
\item if $\lv \gamma(\lambda) \rv < 1$, then $-\lambda$ is in the residual spectrum of $-T$, with co-eigenvector $\ev^-_\lambda$.
\end{itemize}
The Wiener--Hopf techniques used in the proof of Theorem~\ref{thm:eig} likely lead to a similar explicit description of eigenvectors and co-eigenvectors for non-real elements of the spectrum of $T$. However, we did not pursue this direction, as the present work is focused on the generalised eigenfunction expansion.

Variants of Theorem~\ref{thm:main} are well-known for self-adjoint infinite Toeplitz matrices, that is, when $a_{-k} = \overline{a}_k$. In this case the explicit generalised eigenvector expansion goes back to Rosenblum's work~\cite{r65}; see also~\cite{sy20,sy22} for a modern exposition and extensions. When $(a_k)$ is an $\amcm$ sequence with $a_0 = -\sum_{k \in \Z \setminus \{0\}} a_k$ and $a_{-k} = a_k$, then our Theorem~\ref{thm:main} recovers this general result in the following form. We have $\lv \gamma(\lambda) \rv = 1$ and $\ev^+_\lambda = \ev^-_\lambda$ for every $\lambda \in \Lambda = (0, -\hat a(-1))$. Since $\Gamma$ is a semi-circle, it winds around $0$, and Theorem~\ref{thm:main} simplifies to
\formula{
 (P(T))_{k, l} & = \frac{2}{\pi} \int_{\Lambda} P(-\lambda) \ev_{\lambda, k} \ev_{\lambda, l} d\lambda ,
}
where
\formula{
 \ev_{\lambda, k} & = \sin\bigl((k + 1) \thet_\lambda + \thet_\lambda^\pm\bigr) - \evr_{\lambda, k} ,
}
$(\evr_{\lambda, k} : k \ge 0)$ is a completely monotone sequence, $\thet_\lambda \in (0, \pi)$, and $\thet_\lambda^\pm \in [0, \tfrac{1}{2} (\pi - \thet_\lambda)]$.

On the other hand, very little seems to be known about the explicit spectral representations of non-normal infinite Toeplitz matrices $T$, and in particular no results similar to the generalised eigenvector expansion given in Theorem~\ref{thm:main} seem to be available in the literature. We refer to~\cite{bs99,bs06} for a general account on the spectral theory of finite and infinite Toeplitz matrices. Spectral properties of non-normal operators, including finite and infinite Toeplitz matrices, are discussed in~\cite{te05}. The classical text~\cite{g66} provides a detailed exposition of the theory of the Riemann--Hilbert problem, intimately related to the Wiener--Hopf factorisation; see also Chapter~IX in~\cite{mw73}.


\subsection{Applications in probability}

If $(a_k : k \in \Z)$ is a probability sequence (that is, $a_k \ge 0$ for $k \in \Z$ and $\sum_{k = -\infty}^\infty a_k = 1$), then there is a random walk $(X_n : n \ge 0)$ with values in $\Z$ such that $a_k = \pr[X_{n + 1} - X_n = k]$ for every $k \in \Z$ and $n \ge 0$. We assume that the increments $X_{n + 1} - X_n$ are independent random variables, and $X_0 = 0$. The running maxima and minima
\formula{
 \ol{X}_n & = \max\{X_0, X_1, \ldots, X_n\} , \\
 \ul{X}_n & = \min\{X_0, X_1, \ldots, X_n\}
}
are central objects in the fluctuation theory of random walks. The entries of the $n$th power of the Toeplitz matrix $T = (a_{k - l} : k, l \ge 0)$ are easily seen to coincide with the transition probabilities of $(X_n : n \ge 0)$ with a barrier:
\formula{
 (T^n)_{k, l} & = \pr \bigl[ \ul{X}_n \ge -l , \, X_n = -l + k \bigr] ;
}
see Sections~\ref{sec:rw} and~\ref{sec:rw:wh} for a more detailed discussion.

Suppose that $(a_k : k \in \Z)$ is a probability $\amcm$ sequence such that the spine $\Gamma$ winds around $0$. Then $\hat a(1) = 1$, and Theorem~\ref{thm:main} with $P(\lambda) = (1 - \lambda)^n$ gives a semi-explicit description of the joint distribution of $(\ul{X}_n, X_n)$:
\formula[eq:infimum]{
 \pr \bigl[ \ul{X}_n \ge -l, \, X_n = k \bigr] & = \frac{2}{\pi} \int_{\Lambda} (1 - \lambda)^n \ev^+_{\lambda, k + l} \ev^-_{\lambda, l} \, \frac{\lv \gamma'(\lambda) \rv}{\lv \gamma(\lambda) \rv} \, d\lambda .
}
In some special cases, the above expression simplifies to an explicit integral formula, suitable for numerical evaluation; some examples of this kind are discussed in Section~\ref{sec:ex}.

While in general the formula presented above may not be tractable numerically, we plan to use it in the theoretical study of asymptotic properties of the distribution of $\ul{X}_n$ and $\ol{X}_n$. This is motivated by the analogous results that have been obtained in the context of, mostly symmetric, Lévy processes (and the corresponding Wiener--Hopf operators), see~\cite{k25,kmr13}.


\subsection{Examples}

Our expressions~\eqref{eq:i:theta}, \eqref{eq:i:theta:p} and~\eqref{eq:i:theta:m} for $\thet_\lambda$, $\thet_\lambda^+$, $\thet_\lambda^-$, and formulae~\eqref{eq:i:psi:p} and~\eqref{eq:i:psi:m} for $\evr_{\lambda, k}^+$ and $\evr_{\lambda, l}^-$ are complicated, but explicit. They involve the Wiener--Hopf factors $F^+(\lambda; z)$ and $F^-(\lambda; z)$, which are given explicitly in~\eqref{eq:i:fr:wh:p} and~\eqref{eq:i:fr:wh:m}, and the only implicitly defined notion: the parameterisation $\gamma(\lambda)$ of the spine $\Gamma$.

For a given summable $\amcm$ sequence $(a_k)$ with explicitly known $\gamma(\lambda)$, we obtain explicit, numerically tractable integral expressions for the generalised eigenvalues, eigenvectors and co-eigenvectors. In Section~\ref{sec:ex}, we discuss the following classes of examples:
\begin{enumerate}[label=(\alph*)]
\item\label{it:ex:1} tridiagonal Toeplitz matrices;
\item\label{it:ex:2} Toeplitz matrices corresponding to sequences $(a_k)$ such that $(a_k : k \ge 1)$ and $(a_{-k} : k \ge 1)$ are geometric sequences;
\item\label{it:ex:3} Toeplitz matrices with symbol $\hat a(z) = -(1 - z)^\alpha (1 - z^{-1})^\beta$ of Fisher--Hartwig type.
\end{enumerate}
An explicit description of the parameterisation of the spine is available in cases~\ref{it:ex:1} and~\ref{it:ex:2}, and also in case~\ref{it:ex:3} with $\alpha = \beta$. In case~\ref{it:ex:3} with $\alpha \ne \beta$, $\gamma(\lambda)$ can still be evaluated numerically using, for example, Newton's method. A more detailed discussion is given in Section~\ref{sec:ex}.


\subsection{Structure of the paper}

In Section~\ref{sec:pre}, we recall standard definitions and set up the notation. Then, in Section~\ref{sec:amcm}, we study the integral representation of symbols of $\amcm$ sequences. It reveals a link between them and the class of Rogers functions, introduced and thoroughly studied in~\cite{k19,k25}. This connection is established and exploited in Section~\ref{sec:rog}, where we prove the key technical results on the Wiener--Hopf factorisation of symbols of $\amcm$ sequences.

We return to the study of Toeplitz matrices in Section~\ref{sec:wh}, where we express the generating function of the sequence of powers of a Toeplitz matrix in terms of the Wiener--Hopf factors of the corresponding symbol. We give two independent arguments. The shorter one is purely analytical and involves Krein's factorisation theorem. The other proof is of probabilistic nature, and while it is more laborious, we find it more intuitive and insightful. It also gives us an opportunity to further discuss our probabilistic motivations.

After all this preparatory work, we prove our main results, Theorems~\ref{thm:main} and~\ref{thm:eig}, in Section~\ref{sec:inv}. A number of examples are then discussed in Section~\ref{sec:ex}.

%
%

\section{Preliminaries}
\label{sec:pre}

In this section we fix the notation, recall standard definitions, and prove some auxiliary results.


\subsection{Notation}

By a sequence we mean either a one-sided sequence $(a_k : k \ge 0)$ or $(a_k : k \ge 1)$, or a two-sided sequence $(a_k : k \in \Z)$. When the meaning is clear from the context, we write $(a_k)$ for a one-sided or a two-sided sequence.

For a real function $f$, we denote by $f(x^+) = \lim_{y \to x^+} f(y)$ and $f(x^-) = \lim_{y \to x^-} f(y)$ the one-sided limits. In a similar way, by $f(\infty^-)$ and $f(-\infty^+)$ we denote the limits at $\infty$ and $-\infty$.

We use $\Z$, $\R$, and $\C$ for the sets of integers, real numbers, and complex numbers, respectively. We denote the open and closed unit disks in the complex plane by $\disk = \{z \in \C : \lv z \rv < 1\}$ and $\overline \disk = \{z \in \C : \lv z \rv \le 1\}$, and we use $\torus = \{z \in \C : \lv z \rv = 1\}$ for the unit circle. The open and closed right complex half-planes are denoted by $\hp = \{z \in \C : \re z > 0\}$, $\overline \hp = \{z \in \C : \re z \ge 0\}$.

When $a, b \in \C$, we denote by $[a, b]$ the line segment in the complex plane, and we identify $[a, b]$ with $[b, a]$. By $\cinfty$ we denote the complex infinity, so that $\C \cup \{\cinfty\}$ becomes the Riemann sphere. In this context we agree that $0^{-1} = \cinfty$ and $\cinfty^{-1} = 0$.

We use $\re z$, $\im z$, $\lv z \rv$ and $\arg z$ to denote the real part, the imaginary part, the modulus, and the argument of a complex number $z$. We always assume that $\arg z$ takes values in $(-\pi, \pi]$, and we agree that $\arg 0 = 0$.

We record two elementary inequalities. If $a, b \in \C \setminus \{0\}$ and $\thet = \lv \arg (a b^{-1}) \rv$, then
\formula[eq:triangle]{
 \lv a - b \rv & \ge (\lv a \rv + \lv b \rv) \sin \tfrac{\thet}{2} ;
}
see equation~(2.8) in~\cite{k25}. If $r \in [0, 1]$ and $z \in \torus$, then
\formula{
 4 \lv r - z \rv^2 - \lv 1 - z \rv^2 & = (4 + 4 r^2 - 8 r \re z) - (2 - 2 \re z) \\
 & = 2 (1 + \re z) (1 - r)^2 + 2 (1 - \re z) (r^2 + 2 r) \ge 0 ,
}
and hence
\formula[eq:disk]{
 \lv r - z \rv \ge \tfrac{1}{2} \lv 1 - z \rv .
}


\subsection{Stieltjes transform}

Integral representations of holomorphic functions of certain classes play an essential role in our development. They are closely related to the Stieltjes transform of a measure $\mu$ on $\R$, defined by
\formula{
 f(z) & = \frac{1}{\pi} \int_{\R} \frac{1}{z - s} \, \mu(ds)
}
whenever $\min\{1, \lv s \rv^{-1}\}$ is integrable with respect to $\mu$. In this case $\mu$ can be recovered as the vague boundary limit of $\im f$:
\formula{
 \mu(ds) & = \lim_{t \to 0^+} \im f(s + i t) ds .
}
Additionally, if $\mu$ is absolutely continuous, then the density function of $\mu$, denoted again by $\mu(s)$, is equal to 
\formula{
 \mu(s) & = \lim_{t \to 0^+} \im f(s + i t)
}
for almost every $s \in \R$, and the above limit can be taken nontangentially, in the sense that for every $\eps > 0$, we have
\formula{
 \mu(s) & = \lim_{\substack{z \to s \\ \arg(z - s) \in (\eps, \pi - \eps)}} \im f(z)
}
for almost every $s \in \R$. Furthermore,
\formula{
 \mu(\{s\}) & = \lim_{t \to 0^+} \pi t f(s + i t)
}
for every $s \in \R$, and again a nontangential version of this identity holds true. These and similar properties are frequently used below without further comment.

Stieltjes transforms of nonnegative measures map the upper complex half-plane into the (closed) lower complex half-plane, and so they are closely related to Nevanlinna--Pick functions, which preserve a given complex half-plane. A general Nevanlinna--Pick function in the right complex half-plane $\hp$ has the form
\formula{
 c_1 z + i c_0 - \frac{c_{-1}}{z} + \int_{\R} \biggl( \frac{1}{z + i s} + \frac{i \sign s}{1 + \lv s \rv} \biggr) \mu(ds) ,
}
where $c_1 \ge 0$, $c_0 \in \R$, $c_{-1} \ge 0$, and $\mu$ is a nonnegative measure on $\R$ such that $\min\{1, s^{-2}\}$ is integrable with respect to $\mu$. We will see an equivalent representation of Rogers functions in~\eqref{eq:rogers:int}.


\subsection{Stieltjes and complete Bernstein functions}

The following two classes of holomorphic functions $f$ on $\C \setminus (-\infty, 0]$ and nonnegative on $(0, \infty)$ play an important role in our development. We say that $f$ is a \emph{Stieltjes function} if additionally $f$ maps the upper complex half-plane into the (closed) lower complex half-plane. In this case
\formula[eq:stieltjes:int]{
 f(z) & = c_0 + \frac{c_{-1}}{z} + \frac{1}{\pi} \int_{(0, \infty)} \frac{1}{z + s} \, \nu(ds) , 
}
where $c_0, c_{-1} \ge 0$ and $\nu$ is a nonnegative measure on $(0, \infty)$ such that $\min\{1, s^{-1}\}$ is integrable with respect to $\nu$. Furthermore,
\formula*[eq:stieltjes:constants]{
 c_0 & = \lim_{z \to \infty} f(z) , \\
 c_{-1} & = \lim_{z \to 0^+} z f(z) ,
}
and, in the sense of vague convergence of measures on $(0, \infty)$,
\formula[eq:stieltjes:measure]{
 \nu(ds) & = \lim_{t \to 0^+} (-\im f(-s + i t)) ds .
}
A holomorphic function $f$ in $\C \setminus (-\infty, 0]$ is a Stieltjes function if and only if
\formula{
 0 & \ge \arg f(z) \ge -\arg z && \text{if $\im z > 0$,} \\
 0 & \le \arg f(z) \le -\arg z && \text{if $\im z < 0$.}
}
For these and further properties of Stieltjes functions, we refer to Chapter~2 in~\cite{ssv26}.

If $f$ is holomorphic on $\C \setminus (-\infty, 0]$, nonnegative on $(0, \infty)$, and maps the upper complex half-plane into its closure, then $f$ is said to be a \emph{complete Bernstein function}. Such functions have the form
\formula{
 f(z) & = c_1 z + c_0 + \frac{1}{\pi} \int_{(0, \infty)} \frac{z}{z + s} \, \nu(ds) , 
}
where $c_1, c_0 \ge 0$ and $\nu$ is a nonnegative measure on $(0, \infty)$ such that $\min\{1, s^{-1}\}$ is integrable with respect to $\nu$, and these parameters are given by
\formula{
 c_1 & = \lim_{z \to \infty} \frac{f(z)}{z} , \\
 c_0 & = \lim_{z \to 0^+} f(z) ,
}
and
\formula{
 \nu(ds) & = \lim_{t \to 0^+} \im f(-s + i t) ds .
}
A holomorphic function $f$ in $\C \setminus (-\infty, 0]$ is a Stieltjes function if and only if
\formula{
 0 & \ge \arg f(z) \ge -\arg z && \text{if $\im z > 0$,} \\
 0 & \le \arg f(z) \le -\arg z && \text{if $\im z < 0$.}
}
A holomorphic function $f$ in $\C \setminus (-\infty, 0]$ is a complete Bernstein function if and only if
\formula*[eq:cbf:arg]{
 0 & \le \arg f(z) \le \arg z && \text{if $\im z > 0$,} \\
 0 & \ge \arg f(z) \ge \arg z && \text{if $\im z < 0$.}
}
A detailed study of complete Bernstein functions can be found in Chapters~7--8 in~\cite{ssv26}.

Clearly, $f(z)$ is a Stieltjes function if and only if $z f(z)$ is a complete Bernstein function. Similarly, unless $f$ is identically $0$, $f(z)$ is a Stieltjes function if and only if $1 / f(z)$ is a complete Bernstein function.

We sometimes say that the restriction of $f$ to $(0, \infty)$ is a Stieltjes of complete Bernstein function. By the uniqueness of the holomorphic extension, this causes no confusion.

We will need the following auxiliary result.

\begin{proposition}[Proposition~3.18 in~\cite{k25}]
\label{prop:cbf:s}
If $h$ is a complete Bernstein function and $\im \zeta > 0$, then there is a Stieltjes function $g$ such that
\formula{
 \frac{h(\xi)}{(\xi - \zeta) (\xi - \overline \zeta)} = \frac{1}{2 i \im \zeta} \biggl( \frac{h(\zeta)}{\xi - \zeta} - \frac{h(\overline \zeta)}{\xi - \overline \zeta} \biggr) - g(\xi)
}
for $\xi \in \C \setminus (-\infty, 0] \setminus \{\zeta, \overline \zeta\}$. Furthermore, the constants $c_0$ and $c_{-1}$ in the Stieltjes representation~\eqref{eq:stieltjes:int} of $g$ are equal to zero.
\end{proposition}


\subsection{Completely monotone sequences}

Recall that a one-sided sequence $(a_k : k \ge 0)$ is completely monotone if and only if for every $n \ge 0$, the $n$th iterated differences $(\Delta^n a_k)$ are of alternating signs: $(-1)^n \Delta^n a_k \ge 0$ for every $n, k \ge 0$. Equivalently, $(a_k)$ is the mixture of geometric sequences:
\formula{
 a_k & = \int_{[0, 1]} s^k \mu(ds)
}
for some nonnegative finite measure $\mu$ on $[0, 1]$. The generating function of such a sequence is given by
\formula{
 \hat a(z) & = \sum_{k = 0}^\infty a_k z^k
}
whenever the series converges, so at least when $z \in \disk$. By Fubini's theorem,
\formula[eq:cm:int]{
 \hat a(z) & = \int_{[0, 1]} \frac{1}{1 - s z} \, \mu(ds) .
}
It is easy to see that if $c_0 = 0$, then the Stieltjes function~\eqref{eq:stieltjes:int} is the Laplace transform of a completely monotone function. It turns out that Stieltjes functions also arise as generating functions of completely monotone sequences. The following simple result is well-known, but we could not find this exact statement in the literature.

\begin{proposition}
\label{prop:cm:gf}
A function $f$ is a Stieltjes function if and only if there is a completely monotone sequence $(a_k)$ with generating function $\hat a(z) = f(1 - z)$.
\end{proposition}

\begin{proof}
The proof boils down to a simple comparison of~\eqref{eq:stieltjes:int} and~\eqref{eq:cm:int}. If $(a_k)$ has the generating function given by~\eqref{eq:cm:int}, then
\formula{
 \hat a(1 - z) & = \int_{[0, 1]} \frac{1}{1 - s + s z} \, \mu(ds) = \mu(\{0\}) + \frac{\mu(\{1\})}{z} + \int_{(0, 1)} \frac{s^{-1}}{s^{-1} - 1 + z} \, \mu(ds) .
}
If we set $\ph(s) = s^{-1} - 1$ for $s \in (0, 1)$, then, by the substitution $r = \ph(s)$, or $s = (1 + r)^{-1}$, we obtain
\formula{
 \hat a(1 - z) & = \mu(\{0\}) + \frac{\mu(\{1\})}{z} + \int_{(0, \infty)} \frac{1 + r}{r + z} \, \mu \circ \ph^{-1}(dr) .
}
This is the same as~\eqref{eq:stieltjes:int}, with $c_0 = \mu(\{0\})$, $c_{-1} = \mu(\{1\})$, and $\nu(dr) = (1 + r) \mu \circ \ph^{-1}(dr)$.

The converse is proved exactly in the same way.
\end{proof}

%
%

\section{\texorpdfstring{$\amcm$}{AM/CM} sequences}
\label{sec:amcm}

In this relatively short section, we recall the definition of $\amcm$ sequences $(a_k)$ and their generating functions $\hat a$. We also introduce the notion of the symbol $F(z) = \hat a(1) - \hat a(z)$ and discuss its integral representation and holomorphic extension. Finally, we evaluate certain nontangential boundary limits of this extension.


\subsection{\texorpdfstring{$\amcm$}{AM/CM} sequences}

The following is our standing assumption.

\begin{assumption}
The two-sided real sequence $(a_k : k \in \Z)$ is summable and it is an $\amcm$ sequence in the following sense:
\formula{
 & \text{$(a_k : k \ge 1)$ and $(a_{-k} : k \ge 1)$ are completely monotone.}
}
\end{assumption}

Note that the meaning of an `$\amcm$ sequence' here is slightly different from the definition of an `$\amdcm$ sequence' in the context of bell-shaped sequences~\cite{kw23,kw}, where $(a_k : k \ge 0)$ and $(a_{-k} : k \ge 0)$ were assumed to be completely monotone. Clearly, every $\amdcm$ sequence is an $\amcm$ sequence. The converse is not true. However, if $(a_k)$ is an $\amcm$ sequence and we replace $a_0$ by $\infty$, then $(a_k)$ is formally an $\amdcm$ sequence.

In the introduction, we denoted by $\hat a$ the generating function of $(a_k)$:
\formula{
 \hat a(z) & = \sum_{k = -\infty}^\infty a_k z^k
}
for $z \in \torus$. However, it will be more convenient for us to use the function
\formula{
 F(z) & = \hat a(1) - \hat a(z) = \sum_{k = -\infty}^\infty a_k (1 - z^k) ,
}
defined initially for $z \in \torus$. We say that $F$ is the \emph{symbol} of $(a_k)$, and in the remaining part of the article we only use $F$ in definitions and statements of our results. Clearly, $F$ is the generating function of the sequence $(\hat a(1) \delta_k - a_k)$, where $\delta_0 = 1$ and $\delta_k = 0$ for $k \ne 0$.

Note that the definition of $F$ does not depend on the value of $a_0$. In the introduction, we often assume that $a_0$ is chosen in such a way that $\hat a(1) = 0$. In this case, $F(z) = -\hat a(z)$. Throughout the rest of the paper, we generally consider an arbitrary value of $a_0$, although sometimes it will be convenient to assume that $a_0 = 0$.

If $(a_k)$ tends to zero as $k \to \pm \infty$, then $(a_k)$ is an $\amcm$ sequence in the above sense if and only if there are finite nonnegative measures $\mu_+$ and $\mu_-$ on $[0, 1)$ such that
\formula[]{
\label{eq:amcm:p}
 a_k & = \int_{[0, 1)} s^{k - 1} \mu_+(ds) && \text{for $k \ge 1$,} \\
\label{eq:amcm:m}
 a_k & = \int_{[0, 1)} s^{1 - k} \mu_-(ds) && \text{for $k \le -1$.}
}
Furthermore, by Fubini's theorem, $(a_k)$ defined by the above formulae is summable if and only if
\formula{
 \int_{[0, 1)} \frac{1}{1 - s} \, \mu_+(ds) & = \sum_{k = 1}^\infty a_k < \infty , \\
 \int_{[0, 1)} \frac{1}{1 - s} \, \mu_-(ds) & = \sum_{k = -\infty}^{-1} a_k < \infty .
}
Note that $a_k = 0$ for all $k < 0$ if and only if $\mu_- = 0$. Similarly, $a_k = 0$ for all $k > 0$ if and only if $\mu_+ = 0$.

We conclude with an auxiliary definition.

\begin{definition}
An $\amcm$ sequence $(a_k)$ is said to be:
\begin{itemize}
\item \emph{nontrivial} if $a_k \ne 0$ for some $k \ne 0$;
\item \emph{one-sided} if $a_k = 0$ for all $k < 0$ or $a_k = 0$ for all $k > 0$;
\item \emph{two-sided} if it is not one-sided.
\end{itemize}
\end{definition}


\subsection{Symbols of \texorpdfstring{$\amcm$}{AM/CM} sequences}
\label{sec:symbol}

For a summable $\amcm$ sequence $(a_k)$, we define
\formula{
 \mass & = \sum_{k \in \Z \setminus \{0\}} a_k = \hat a(1) - a_0 .
}
Clearly, $\lv \hat a(z) - a_0 \rv \le \mass$, and so $\lv F(z) \rv \le \lv \hat a(1) - a_0 \rv + \lv \hat a(z) - a_0 \rv \le 2 \mass$ for $z \in \torus$. We have
\formula{
 F(z) & = \hat a(1) - \hat a(z) = \sum_{k = 1}^\infty a_{-k} (1 - z^{-k}) + \sum_{k = 1}^\infty a_k (1 - z^k) .
}
Using~\eqref{eq:amcm:p} and~\eqref{eq:amcm:m}, we obtain
\formula{
 F(z) & = \sum_{k = 1}^\infty \int_{[0, 1)} s^{k - 1} (1 - z^{-k}) \mu_-(ds) + \sum_{k = 1}^\infty \int_{[0, 1)} s^{k - 1} (1 - z^k) \mu_+(ds) .
}
By Fubini's theorem,
\formula{
 F(z) & = \int_{[0, 1)} \sum_{k = 1}^\infty s^{k - 1} (1 - z^{-k}) \mu_-(ds) + \int_{[0, 1)} \sum_{k = 1}^\infty s^{k - 1} (1 - z^k) \mu_+(ds) .
}
Adding up the geometric series, we arrive at
\formula[eq:symbol]{
 F(z) & = \int_{[0, 1)} \biggl(\frac{1}{1 - s} - \frac{z^{-1}}{1 - s z^{-1}}\biggr) \mu_-(ds) + \int_{[0, 1)} \biggl(\frac{1}{1 - s} - \frac{z}{1 - s z}\biggr) \mu_+(ds) .
}
Note that the right-hand side of~\eqref{eq:symbol} defines a holomorphic function of $z \in \C$ such that $z \notin \supp \mu_-$ and $z^{-1} \notin \supp \mu_+$. From now on, $F$ denotes this extension. Clearly, $F(\overline z) = \overline{F(z)}$ for every admissible $z$.

By formula~\eqref{eq:symbol}, $F$ is closely related to the Stieltjes transforms of $\mu_+$ and $\mu_-$. More precisely, when $\im z \ne 0$, we have
\formula[eq:poisson]{
 \im F(z) & = \int_{[0, 1)} \frac{\im z}{\lv z - s \rv^2} \, \mu_-(ds) - \int_{[0, 1)} \frac{\im z}{\lv s z - 1 \rv^2} \, \mu_+(ds) .
}
It follows that, in the sense of vague convergence of measures on $(0, 1)$,
\formula[]{
\label{eq:measure:p}
 \mu_+(ds) & = \lim_{t \to 0^+} \frac{\im F((s + i t)^{-1})}{\pi} \, \ind_{(0, 1)}(s) ds , \\
\label{eq:measure:m}
 \mu_-(ds) & = \lim_{t \to 0^+} \frac{\im F(s + i t)}{\pi} \, \ind_{(0, 1)}(s) ds ,
}
and additionally
\formula[]{
\label{eq:measure:atom:p}
 \mu_+(\{0\}) & = \lim_{t \to 0^+} \frac{t \im F(-i t^{-1})}{\pi} \, , \\
\label{eq:measure:atom:m}
 \mu_-(\{0\}) & = \lim_{t \to 0^+} \frac{t \im F(i t)}{\pi} \, .
}

Observe that all expressions given above are \emph{symmetric}, in the following sense. If we consider the reversed sequence $(\check a_k) = (a_{-k})$, then $(\check a_k)$ is $\amcm$, the corresponding function $F$ becomes $\check F(x) = F(x^{-1})$, and the roles of the measures $\mu_+$ and $\mu_-$ are exchanged: $\check \mu_+ = \mu_-$ and $\check \mu_- = \mu_+$.


\subsection{Boundary limits}

We collect two simple properties of symbols of summable $\amcm$ sequences.

\begin{lemma}
\label{lem:one}
If $F$ is the symbol of a summable $\amcm$ sequence, then
\formula{
 \lim_{\xi \to 0^\pm} F(1 + i \xi) & = 0 ,
}
and the limit can be taken nontangentially: $\lv \xi \rv \to 0$ with $\lv\arg \xi \rv \le \tfrac{\pi}{2} - \eps$ or $\lv \arg(-\xi) \rv \le \tfrac{\pi}{2} - \eps$, where $\eps > 0$ is arbitrary.
\end{lemma}

\begin{proof}
The result follows from~\eqref{eq:symbol} and the dominated convergence theorem; we only need to verify that the application of the latter is legitimate. Recall that $(1 - s)^{-1}$ is integrable over $(0, 1)$ with respect to $\mu_+(ds)$ and $\mu_-(ds)$. Furthermore, if $s \in (0, 1)$, $\lv x^{-1} \rv \le 2$ and $\lv \arg(1 - x^{-1}) \rv \le \pi - \eps$, by~\eqref{eq:triangle} we have
\formula{
 \biggl|\frac{x^{-1}}{1 - s x^{-1}}\biggr| & = \frac{s^{-1} \lv x \rv^{-1}}{\lv s^{-1} - x^{-1} \rv} \le \frac{s^{-1} \lv x \rv^{-1}}{(\lv s^{-1} - 1 \rv + \lv x^{-1} - 1 \rv) \sin \tfrac{\eps}{2}} \le \frac{\lv x\rv^{-1}}{(1 - s) \sin \tfrac{\eps}{2}} \le \frac{2}{(1 - s) \sin \tfrac{\eps}{2}} \, .
}
Similarly, if $s \in (0, 1)$, $\lv x \rv \le 2$ and $\lv \arg(1 - x) \rv \le \pi - \eps$, we have
\formula{
 \biggl|\frac{x}{1 - s x}\biggr| & = \frac{s^{-1} \lv x \rv}{\lv x - s^{-1} \rv} \le \frac{s^{-1} \lv x \rv}{(\lv x - 1 \rv + \lv s^{-1} - 1 \rv) \sin \tfrac{\eps}{2}} \le \frac{\lv x \rv}{(1 - s) \sin \tfrac{\eps}{2}} \le \frac{2}{(1 - s) \sin \tfrac{\eps}{2}} \, .
}
By the dominated convergence theorem, $F(x)$ tends to $0$ as $\lv x - 1 \rv \to 0$ with $\lv \arg(1 - x) \rv \le \pi - \eps$ and $\lv \arg(1 - x^{-1}) \rv \le \pi - \eps$. To complete the proof, it remains to substitute $x = 1 + i \xi$.
\end{proof}

\begin{lemma}
\label{lem:range}
If $F$ is the symbol of a summable $\amcm$ sequence, then
\formula{
 \lim_{x \to 0^-} F(x) & = \int_{[0, 1)} \frac{1}{s (1 - s)} \, \mu_-(ds) + \int_{[0, 1)} \frac{1}{1 - s} \, \mu_+(ds) , \\
 \lim_{x \to -\infty} F(x) & = \int_{[0, 1)} \frac{1}{1 - s} \, \mu_-(ds) + \int_{[0, 1)} \frac{1}{s (1 - s)} \, \mu_+(ds) ,
}
where we agree that $1 / 0 = \infty$. Furthermore, the limits can be taken nontangentially: $\lv x \rv \to 0$ or $\lv x \rv \to \infty$, with $\lv \arg(-x) \rv \le \pi - \eps$, where $\eps > 0$ is arbitrary.
\end{lemma}

\begin{proof}
As $x \to 0^-$, the function $-x^{-1} / (1 - s x^{-1})$ increases to $s^{-1}$, while the function $-x / (1 - s x)$ decreases to $0$. Hence, by~\eqref{eq:symbol}, the monotone convergence theorem, and the dominated convergence theorem, we have
\formula{
 \lim_{x \to 0^-} F(x) & = \int_{[0, 1)} \biggl(\frac{1}{1 - s} + \frac{1}{s}\biggr) \mu_-(ds) + \int_{[0, 1)} \frac{1}{1 - s} \, \mu_+(ds) \\
 & = \int_{[0, 1)} \frac{1}{s (1 - s)} \, \mu_-(ds) + \int_{[0, 1)} \frac{1}{1 - s} \, \mu_+(ds) .
}
Similarly,
\formula{
 \lim_{x \to -\infty} F(x) & = \int_{[0, 1)} \frac{1}{1 - s} \, \mu_-(ds) + \int_{[0, 1)} \biggl(\frac{1}{1 - s} + \frac{1}{s}\biggr) \mu_+(ds) \\
 & = \int_{[0, 1)} \frac{1}{1 - s} \, \mu_-(ds) + \int_{[0, 1)} \frac{1}{s (1 - s)} \, \mu_+(ds) .
}
The argument is slightly more involved for the nontangential convergence: if $\lv \arg(-x) \rv \le \pi - \eps$, $\lv x \rv \le 1$ and $s \in (0, 1)$, then
\formula{
 \lv 1 - s x^{-1} \rv & \ge s \lv \im x^{-1} \rv \ge s \lv x^{-1} \rv \sin \eps , & \lv 1 - s x \rv & \ge \sin \eps ,
}
and hence
\formula*{
 \biggl|\frac{1}{1 - s} - \frac{x^{-1}}{1 - s x^{-1}}\biggr| \le \frac{1}{1 - s} + \frac{1}{s \sin \eps} \le \frac{(\sin \eps)^{-1}}{s (1 - s)} \, , \\
 \biggl|\frac{1}{1 - s} - \frac{x}{1 - s x}\biggr| \le \frac{1}{1 - s} + \frac{\lv x \rv}{\sin \eps} \le \frac{1 + (\sin \eps)^{-1}}{1 - s} \, .
}
Using the dominated convergence theorem for the latter integral in~\eqref{eq:symbol}, and either the dominated convergence theorem or Fatou's lemma for the former one (depending on whether the limiting integral is finite or not), we find that the limit as $x \to 0^-$ can be taken nontangentially. A very similar argument applies to the limit as $x \to -\infty$.
\end{proof}

%
%

\section{Symbols as Rogers functions}
\label{sec:rog}

In this lengthy and technical section, we study more refined properties of symbols of $\amcm$ sequences. Our main tool is the theory of Rogers functions, developed in~\cite{k19,k25} in the context of Wiener--Hopf operators related to Lévy processes with completely monotone jumps. We prove that the symbol of an $\amcm$ sequence is, up to a simple change of variables, a Rogers function, and this allows us to reuse various results from~\cite{k19,k25}. In fact, we identify the symbol with two Rogers functions, which leads to even stronger results.

While the difficult part of the proofs was mostly carried out in~\cite{k19,k25}, some technical difficulties remain.


\subsection{Rogers functions}

A holomorphic function $f$ on the right complex half-plane $\hp = \{\xi \in \C : \re \xi > 0\}$ is said to be a \emph{Rogers function} if $\re(\xi^{-1} f(\xi)) \ge 0$ for every $\xi \in \hp$ (Definition~3.2 in~\cite{k19}).

Rogers functions are closely related to Nevanlinna--Pick functions, and they have the following Stieltjes representation:
\formula[eq:rogers:int]{
 f(\xi) & = c_2 \xi^2 - i c_1 \xi + c_0 + \frac{1}{\pi} \int_{\R \setminus \{0\}} \biggl(\frac{\xi}{\xi + i r} + \frac{i \xi \sign r}{1 + \lv r \rv}\biggr) \frac{\mu(dr)}{\lv r \rv}
}
where $c_2 \ge 0$, $c_1 \in \R$, $c_0 \ge 0$, and $\mu$ is a nonnegative measure on $\R \setminus \{0\}$ such that $\lv r \rv^{-3} \min\{1, r^2\}$ is integrable with respect to $\mu$; see Theorem~3.3 in~\cite{k19}. Furthermore, we have
\formula[eq:rogers:constants]{
 c_2 & = \lim_{\xi \to \infty} \frac{f(\xi)}{\xi^2} \, , &
 c_1 & = \lim_{\xi \to \infty} \frac{-\im f(\xi)}{\xi} \, , &
 c_0 & = \lim_{\xi \to 0^+} f(\xi) , 
}
and in the sense of vague convergence of measures on $\R \setminus \{0\}$,
\formula[eq:rogers:mu]{
 \lv r \rv^{-1} \mu(dr) = \lim_{t \to 0^+} \re \, \frac{f(t - i r)}{t - i r} \, dr ;
}
see Remark~3.4(f) in~\cite{k19}.

A real function $\levy$ is said to be $\amdcm$ if $\levy(r)$ and $\levy(-r)$ are completely monotone functions of $r > 0$. By Theorem~3.3 in~\cite{k19}, Rogers functions admit the following integral representation:
\formula{
 f(\xi) & = c_2 \xi^2 - i c_1 \xi + c_0 + \frac{1}{\pi} \int_{-\infty}^\infty \bigl(1 - e^{i \xi r} + i \xi (1 - e^{-\lv r \rv}) \sign r\bigr) \levy(r) dr ,
}
where $c_0, c_1, c_2$ are the same constants as in~\eqref{eq:rogers:int}, and $\levy$ is an $\amdcm$ function such that $\levy(r)$ tends to zero as $r \to \pm \infty$ and the integral $\int_{-\infty}^\infty \min\{1, r^2\} \levy(r) dr$ is finite.

Recall that the symbol $F$ of an $\amcm$ sequence $(a_k)$ is equal to the generating function of $(-a_k)$ up to addition by a constant. Similarly, at least for integrable $\amdcm$ functions $\levy$, the corresponding Rogers function $f$ is the Fourier transform of $-\levy$, up to addition by a polynomial term. This suggests that there is some analogy between $(a_k)$ and $F$ on the one hand, and $\levy$ and $f$ on the other hand.

It turns out, however, that the link between symbols of $\amcm$ sequences and Rogers functions is much deeper than the above superficial remark suggests. We begin this section by proving that we can associate two Rogers functions $f$ and $g$ to a symbol $F$ of an $\amcm$ sequence. Next, we recall various results about Rogers functions from~\cite{k19,k25}, and we study their analogues for symbols of $\amcm$ sequences.


\subsection{The symbol as a Rogers function}

We have the following fundamental observation.

\begin{lemma}
\label{lem:symbol:rogers}
If $F$ is the symbol of a summable $\amcm$ sequence, then
\formula{
 f(\xi) & = F(1 + i \xi) \qquad \text{and} \qquad g(\xi) = F((1 + i \xi)^{-1})
}
are Rogers functions. Furthermore, $f(0^+) = g(0^+) = 0$, while $f(\infty^-) = F(-\infty^+)$ and $g(\infty^-) = F(0^-)$ are given by Lemma~\ref{lem:range}.
\end{lemma}

\begin{proof}
We recall formula~\eqref{eq:symbol}: for $z \in \C \setminus [0, \infty)$ we have
\formula{
 F(z) & = \int_{[0, 1)} \biggl(\frac{1}{1 - s} - \frac{z^{-1}}{1 - s z^{-1}}\biggr) \mu_-(ds) + \int_{[0, 1)} \biggl(\frac{1}{1 - s} - \frac{z}{1 - s z}\biggr) \mu_+(ds) .
}
This can be rewritten as
\formula{
 F(z) & = \int_{[0, 1)} \frac{z - 1}{(1 - s)(z - s)} \, \mu_-(ds) + \int_{[0, 1)} \frac{1 - z}{(1 - s)(1 - s z)} \, \mu_+(ds) ,
}
and therefore
\formula{
 \frac{F(1 + i \xi)}{\xi} & = \int_{[0, 1)} \frac{i}{(1 - s)(1 + i \xi - s)} \, \mu_-(ds) + \int_{[0, 1)} \frac{-i}{(1 - s)(1 - s - i s \xi)} \, \mu_+(ds) .
}
When $\re \xi > 0$ and $s \in [0, 1)$, we have
\formula{
 \re \frac{i}{(1 - s)(1 + i \xi - s)} & = \re \frac{i (1 - i \overline{\xi} - s)}{(1 - s) \lv 1 + i \xi - s \rv^2} \\
 & = \frac{\re \xi}{(1 - s) \lv 1 + i \xi - s \rv^2} > 0
}
and
\formula{
 \re \frac{-i}{(1 - s)(1 - s - i s \xi)} & = \re \frac{-i (1 - s + i s \overline{\xi})}{(1 - s) \lv 1 - s - i s \xi \rv^2} \\
 & = \frac{s \re \xi}{(1 - s) \lv 1 - s - i s \xi \rv^2} > 0 .
}
It follows that $\re (\xi^{-1} F(1 + i \xi)) \ge 0$, which proves that $f$ is a Rogers function. By Lemma~\ref{lem:one},
\formula{
 f(0^+) & = \lim_{\xi \to 0^+} F(1 + i \xi) = 0 ,
}
while by Lemma~\ref{lem:range},
\formula{
 f(\infty^-) & = \lim_{\xi \to \infty} f(\xi) = \lim_{\xi \to \infty} F(1 + i \xi) = \lim_{x \to -\infty} F(x) = F(-\infty^+) .
}
The same argument applied to the dual symbol $\check F(z) = F(z^{-1})$ implies that also $\check f(\xi) = \check F(1 + i \xi) = F((1 + i \xi)^{-1}) = g(\xi)$ has the desired properties.
\end{proof}


\subsection{Exponential representation of Rogers functions}

The following exponential representation formula for Rogers functions played a key role in~\cite{k19,k25}.

\begin{proposition}[Theorem~3.3 and Remark 3.4(g) in~\cite{k19}]
\label{prop:rogers:exp}
Let $f$ be a nonzero Rogers function. Then there is a constant $b > 0$ and a Borel function $\ph : \R \to [0, \pi]$ such that
\formula[eq:rogers:exp]{
 f(\xi) & = b \exp \biggl(\frac{1}{\pi} \int_{-\infty}^\infty \biggl(\frac{\xi}{\xi + i r} - \frac{1}{1 + \lv r \rv}\biggr) \frac{\ph(r)}{\lv r \rv} \, dr\biggr)
}
for $\xi \in \hp$. Conversely, the above expression defines a Rogers function for every $c$ and $\ph$ satisfying the conditions listed above. Finally, for almost every $r \in \R$ we have
\formula{
 -\ph(r) \sign r & = \lim_{t \to 0^+} \arg f(-i r + t) ,
}
and the limit can be taken nontangentially: $\lv t \rv \to 0$ with $\lv \arg t \rv \le \tfrac{\pi}{2} - \eps$, where $\eps > 0$ is arbitrary.
\end{proposition}

Formula~\eqref{eq:rogers:exp} allows one to extend a Rogers function $f$ to a holomorphic function on $D_f = \C \setminus (-i \esssupp \ph)$ (Remark~3.4(a) in~\cite{k19}).

By Proposition~3.14 in~\cite{k19}, if $f$ is as in~\eqref{eq:rogers:exp}, then
\formula[eq:rogers:zero]{
 f(0^+) & = b \exp \biggl(-\frac{1}{\pi} \int_{-1}^\infty \frac{1}{1 + \lv r \rv} \, \frac{\ph(r)}{\lv r \rv} \, dr\biggr) .
}


\subsection{Exponential representation of symbols of \texorpdfstring{$\amcm$}{AM/CM} sequences}

The following is an analogue of Proposition~\ref{prop:rogers:exp} for the symbol $F$.

\begin{proposition}
\label{prop:exp}
If $F$ is the symbol of a nontrivial summable $\amcm$ sequence and $\sigma \ge 0$, then there is a constant $c > 0$ and Borel functions $\ph_+, \ph_- : (0, 1) \to [0, \pi]$ such that
\formula[eq:exp]{
 \sigma + F(z) & = c \exp \biggl(\frac{1}{\pi} \int_0^1 \frac{1}{s - z^{-1}} \, \ph_+(s) ds + \frac{1}{\pi} \int_0^1 \frac{1}{s - z} \, \ph_-(s) ds\biggr)
}
for every $z \in \C \setminus [0, \infty)$. Furthermore, for almost every $s \in (0, 1)$ we have
\formula[]{
\label{eq:phi:p}
 \ph_+(s) & = -\lim_{t \to 0^+} \arg(\sigma + F(s^{-1} + i t)) , \\
\label{eq:phi:m}
 \ph_-(s) & = \lim_{t \to 0^+} \arg(\sigma + F(s + i t)) ,
}
and the limit can be taken nontangentially: $\lv t \rv \to 0$ with $\lv \arg t \rv \le \tfrac{\pi}{2} - \eps$, where $\eps > 0$ is arbitrary. Additionally,
\formula[eq:exp:sigma]{
 \sigma & = c \exp\biggl(-\frac{1}{\pi} \int_0^1 \frac{1}{1 - s} \, (\ph_-(s) + \ph_+(s)) ds\biggr) ,
}
with $e^{-\infty}$ understood to be $0$.
\end{proposition}

\begin{proof}
By Lemma~\ref{lem:symbol:rogers}, $f(\xi) = F(1 + i \xi)$ is a (nonzero) Rogers function of $\xi$. Hence, by Proposition~\ref{prop:rogers:exp} applied to the Rogers function $\sigma + f$,
\formula{
 \sigma + F(1 + i \xi) & = b \exp \biggl(\frac{1}{\pi} \int_{-\infty}^\infty \biggl(\frac{\xi}{\xi + i r} - \frac{1}{1 + \lv r \rv}\biggr) \frac{\ph(r)}{\lv r \rv} \, dr\biggr) \\
 & = b \exp \biggl(\frac{1}{\pi} \int_{-\infty}^\infty \biggl(\frac{1}{1 + \lv r \rv} - \frac{i \sign r}{\xi + i r}\biggr) \ph(r) dr\biggr)
}
for $\xi \in \hp$, where $b > 0$ and $\ph : \R \to [0, \pi]$. Furthermore,
\formula{
 -\ph(r) \sign r & = \lim_{t \to 0^+} \arg(\sigma + F(1 + r + i t))
}
for almost every $r \in \R$. We substitute $1 + i \xi = z$, or $\xi = i - i z$, in the above expression for $\sigma + F(1 + i \xi)$:
\formula{
 \sigma + F(z) & = b \exp \biggl(\frac{1}{\pi} \int_{-\infty}^\infty \biggl(\frac{1}{1 + \lv r \rv} - \frac{\sign r}{1 - z + r}\biggr) \ph(r) dr\biggr) .
}
Here $\im z > 0$, but since $F(\overline z) = \overline{F(z)}$, the above equality in fact holds for $z \in \C \setminus \R$.

Observe that $F$ is holomorphic in a neighbourhood of $(-\infty, 0)$ and $F(z) > 0$ for $z < 0$ by~\eqref{eq:symbol}. Hence, if $r < -1$, then $F(1 + r) > 0$ and $\ph(r) = 0$. For $s \in (0, 1)$ let us denote
\formula[]{
\label{eq:ph:m:1}
 \ph_-(s) & = \ph(s - 1) = \lim_{t \to 0^+} \arg(\sigma + F(s + i t)) , \\
\label{eq:ph:p:1}
 \ph_+(s) & = \ph(s^{-1} - 1) = -\lim_{t \to 0^+} \arg(\sigma + F(s^{-1} + i t)) .
}
Then,
\formula{
 \sigma + F(z) = b \exp \biggl( & \frac{1}{\pi} \int_{-1}^0 \biggl(\frac{1}{1 - r} + \frac{1}{1 - z + r}\biggr) \ph_-(1 + r) dr \\
 & \qquad + \frac{1}{\pi} \int_0^\infty \biggl(\frac{1}{1 + r} - \frac{1}{1 - z + r}\biggr) \ph_+((1 + r)^{-1}) dr\biggr) .
}
We substitute
\formula[eq:sub:1]{
 1 + r & = s, & r & = s - 1
}
in the former integral, and
\formula[eq:sub:2]{
 \frac{1}{1 + r} & = s , & r & = \frac{1}{s} - 1
}
in the latter one (these substitutions are frequently used later):
\formula{
 \sigma + F(z) & = b \exp \biggl(\frac{1}{\pi} \int_0^1 \biggl(\frac{1}{2 - s} + \frac{1}{s - z}\biggr) \ph_-(s) ds + \frac{1}{\pi} \int_0^1 \frac{1}{s^2} \biggl(s - \frac{1}{s^{-1} - z}\biggr) \ph_+(s) ds\biggr) .
}
But
\formula{
 \frac{1}{s^2} \biggl(s - \frac{1}{s^{-1} - z}\biggr) & = \frac{1}{s^2} \, \frac{-s z}{s^{-1} - z} = \frac{-z}{1 - s z} = \frac{1}{s - z^{-1}} \, ,
}
and therefore
\formula{
 \sigma + F(z) & = b \exp \biggl(\frac{1}{\pi} \int_0^1 \frac{1}{2 - s} \, \ph_-(s) ds + \frac{1}{\pi} \int_0^1 \frac{1}{s - z} \, \ph_-(s) ds + \frac{1}{\pi} \int_0^1 \frac{1}{s - z^{-1}} \, \ph_+(s) ds\biggr) .
}
This is formula~\eqref{eq:exp}, with
\formula[eq:exp:c]{
 c & = b \exp \biggl(\frac{1}{\pi} \int_0^1 \frac{1}{2 - s} \, \ph_-(s) ds\biggr) .
}

In order to prove~\eqref{eq:exp:sigma}, we apply Lemma~\ref{lem:one}, the definition of $f$, and equation~\eqref{eq:rogers:zero} (applied to the Rogers function $\sigma + f$): we have
\formula{
 \sigma & = \lim_{\xi \to 0^+} (\sigma + F(1 + i \xi)) = \lim_{\xi \to 0^+} (\sigma + f(\xi)) = b \exp\biggl(-\frac{1}{\pi} \int_{-1}^\infty \frac{1}{1 + \lv r \rv} \, \frac{\ph(r)}{\lv r \rv} \, dr\biggr).
}
Using the same substitutions~\eqref{eq:sub:1} and~\eqref{eq:sub:2}, we get
\formula{
 \sigma & = b \exp\biggl(\frac{1}{\pi} \int_0^1 \frac{1}{2 - s} \, \ph_-(s) ds - \frac{1}{\pi} \int_0^1 \frac{\ph_-(s)}{1 - s} \, ds - \frac{1}{\pi} \int_0^1 \frac{\ph_+(s)}{1 - s} \, ds\biggl) \\
 & = c \exp\biggl(-\frac{1}{\pi} \int_0^1 \frac{\ph_-(s)}{1 - s} \, ds - \frac{1}{\pi} \int_0^1 \frac{\ph_+(s)}{1 - s} \, ds\biggl) ,
}
as desired.
\end{proof}

Note that the statement of Proposition~\ref{prop:exp} is symmetric: when $(a_k)$ is changed to the reversed sequence $(\check a_k) = (a_{-k})$, then the corresponding dual symbol $\check F(z) = F(z^{-1})$ is given by~\eqref{eq:exp}, with the roles of $\ph_{\pm}$ exchanged: $\check \ph_+(s) = \ph_-(s)$ and $\check \ph_-(s) = \ph_+(s)$.

The right-hand side of formula~\eqref{eq:exp} defines a holomorphic function of $z$ on the set
\formula[eq:dom]{
 D_F & = \{z \in \C : z^{-1} \notin \esssupp \ph_+ , \, z \notin \esssupp \ph_-\} .
}
Observe that if $f$ is the Rogers function defined in Lemma~\ref{lem:symbol:rogers}, then the proof of Proposition~\ref{prop:exp} shows that $\xi \in D_f$ if and only if $1 + i \xi \in D_F$.

It is easy to see that the right-hand side of~\eqref{eq:symbol} is well-defined on $D_F$, and so $F$ is given on $D_F$ by both~\eqref{eq:symbol} and~\eqref{eq:exp} (with $\sigma = 0$). We say that $D_F$ is the \emph{domain} of the symbol $F$.

\begin{remark}
\label{rem:onesided}
If the sequence $(a_k)$ is one-sided, one of the measures $\mu_+$ and $\mu_-$ in~\eqref{eq:symbol} is equal to zero, and one of the functions $\ph_+$ and $\ph_-$ is equal to zero almost everywhere. Indeed: if $a_k = 0$ for $k \le -1$, then $\mu_-([0, 1)) = 0$, and, by~\eqref{eq:symbol}, $F$ extends to a holomorphic function on $\C \setminus [1, \infty]$ satisfying $F(z) \ge 0$ for $z \in (0, 1)$. If $a_k > 0$ for some $k \ge 1$, then $\mu_+([0, 1)) > 0$, and we have a strict inequality $F(z) > 0$ for $z \in (0, 1)$. By~\eqref{eq:phi:m}, $\ph_-(s) = 0$ for almost every $s \in (0, 1)$. Similarly, if $a_k = 0$ for $k \ge 1$, but $a_k > 0$ for some $k \le -1$, then, by~\eqref{eq:phi:p}, $\ph_+(s) = 0$ for almost every $s \in (0, 1)$.

Conversely, if $(a_k)$ is two-sided, then neither $\mu_+$ nor $\mu_-$ is zero, and neither of the functions $\ph_+$ nor $\ph_-$ is equal to zero almost everywhere. Indeed: if $\ph_-(s) = 0$ for almost all $s \in (0, 1)$, then $\C \setminus [1, \infty) \subseteq D_F$ and by~\eqref{eq:exp}, we have $\im F(z) = 0$ for $z \in [0, 1)$. From~\eqref{eq:measure:m} and~\eqref{eq:measure:atom:m} it follows that $\mu_-([0, 1)) = 0$, that is, $a_k = 0$ for $k \le -1$. In a similar way, by~\eqref{eq:measure:p} and~\eqref{eq:measure:atom:p}, if $\ph_+(s) = 0$ for almost all $s \in (0, 1)$, then $a_k = 0$ for $k \ge 1$.
\end{remark}

We have the following converse of Proposition~\ref{prop:exp}.

\begin{proposition}
\label{prop:exp:inv}
If $c > 0$, $\ph_+$ and $\ph_-$ are Borel functions with values in $[0, \pi]$, at least one of $\ph_+$, $\ph_-$ is not equal to zero almost everywhere, and $\sigma$ is given by formula~\eqref{eq:exp:sigma}, then equation~\eqref{eq:exp} defines the symbol $F$ of a summable $\amcm$ sequence, unique up to the choice of $a_0$.
\end{proposition}

\begin{proof}
Suppose that $F$ is given by~\eqref{eq:exp}. We divide the argument into six steps.

\emph{Step 1.}
We can easily reverse the proof of Proposition~\ref{prop:exp} to find that if we define $f(\xi) = F(1 + i \xi)$ and
\formula[eq:ph:pm:2]{
 \ph(r) & = \begin{cases}
   \ph_-(1 + r) & \text{if $r \in (-1, 0)$,} \\
   \ph_+((1 + r)^{-1}) & \text{if $r \in (0, \infty)$,} \\
   0 & \text{otherwise}
 \end{cases}
}
(as in~\eqref{eq:ph:m:1} and~\eqref{eq:ph:p:1}), then
\formula[eq:exp:inf:rogers:0]{
 \sigma + f(\xi) & = b \exp \biggl(\frac{1}{\pi} \int_{-1}^\infty \biggl(\frac{\xi}{\xi + i r} - \frac{1}{1 + \lv r \rv}\biggr) \frac{\ph(r)}{\lv r \rv} \, dr\biggr) ,
}
with constant
\formula{
 b & = c \exp \biggl(-\frac{1}{\pi} \int_0^1 \frac{1}{2 - s} \, \ph_-(s) ds\biggr)
}
(as in~\eqref{eq:exp:c}). It follows that $\sigma + f(\xi)$ is a Rogers function, holomorphic in $\C \setminus (-i \infty, i]$. We consider its Stieltjes representation~\eqref{eq:rogers:int}:
\formula[eq:exp:inv:rogers:1]{
 \sigma + f(\xi) & = c_2 \xi^2 - i c_1 \xi + c_0 + \frac{1}{\pi} \int_{\R \setminus \{0\}} \biggl(\frac{\xi}{\xi + i r} + \frac{i \xi \sign r}{1 + \lv r \rv}\biggr) \frac{\mu(dr)}{\lv r \rv} \, .
}

\emph{Step 2.}
By~\eqref{eq:rogers:mu}, on $\R \setminus \{0\}$ we have, in the sense of vague convergence of measures,
\formula{
 \lv r \rv^{-1} \mu(dr) = \lim_{t \to 0^+} \re \, \frac{f(t - i r)}{t - i r} \, dr .
}
Since $\sigma + f$ is holomorphic in $\C \setminus (-i \infty, i]$ and real-valued on $(i, i \infty)$, we have $\mu((-\infty, -1)) = 0$. In other words, the integral in~\eqref{eq:exp:inv:rogers:1} is effectively taken over $[-1, 0) \cup (0, \infty)$.

\emph{Step 3.}
We now evaluate the constant $c_0$. On the one hand, by~\eqref{eq:rogers:constants}, $c_0 = \sigma + f(0^+)$. On the other hand, by~\eqref{eq:exp:inf:rogers:0} and~\eqref{eq:rogers:zero}, we have
\formula{
 \sigma + f(0^+) & = b \exp \biggl(-\frac{1}{\pi} \int_{-1}^\infty \frac{1}{1 + \lv r \rv} \, \frac{\ph(r)}{\lv r \rv} \, dr\biggr) .
}
As in~\eqref{eq:sub:1} and~\eqref{eq:sub:2} in the proof of Proposition~\ref{prop:exp}, we split the integral into two integrals, over $(-1, 0)$ and $(0, \infty)$, and we substitute $s = 1 + r$ in the former one, and $s = (1 + r)^{-1}$ in the latter one. This leads to
\formula{
 \sigma + f(0^+) & = b \exp \biggl(-\frac{1}{\pi} \int_{-1}^0 \frac{1}{1 - r} \, \frac{\ph_-(1 + r)}{-r} \, dr - \frac{1}{\pi} \int_0^\infty \frac{1}{1 + r} \, \frac{\ph_+((1 + r)^{-1})}{r} \, dr\biggr) \\
 & = c \exp \biggl(-\frac{1}{\pi} \int_0^1 \frac{1}{2 - s} \, \ph_-(s) ds - \frac{1}{\pi} \int_0^1 \frac{1}{2 - s} \, \frac{\ph_-(s)}{1 - s} \, ds - \frac{1}{\pi} \int_0^1 \frac{\ph_+(s)}{1 - s} \, ds\biggr) \\
 & = c \exp \biggl(-\frac{1}{\pi} \int_0^1 \frac{\ph_-(s)}{1 - s} \, ds - \frac{1}{\pi} \int_0^1 \frac{\ph_+(s)}{1 - s} \, ds\biggr) .
}
Comparing this with the definition~\eqref{eq:exp:sigma} of $\sigma$, we find that $f(0^+) = 0$ and $c_0 = \sigma$. Therefore, equation~\eqref{eq:exp:inv:rogers:1} takes the form
\formula[eq:exp:inv:rogers:2]{
 f(\xi) & = c_2 \xi^2 - i c_1 \xi + \frac{1}{\pi} \int_{[-1, 0) \cup (0, \infty)} \biggl(\frac{\xi}{\xi + i r} + \frac{i \xi \sign r}{1 + \lv r \rv}\biggr) \frac{\mu(dr)}{\lv r \rv} \, .
}

\emph{Step 4.}
We turn to the evaluation of $c_2$. Observe that for $\xi > 0$,
\formula{
 \lv \sigma + f(\xi) \rv & = b \exp \biggl(\frac{1}{\pi} \int_{-1}^\infty \biggl(\frac{\xi^2}{\xi^2 + r^2} - \frac{1}{1 + \lv r \rv}\biggr) \frac{\ph(r)}{\lv r \rv} \, dr\biggr) \\
 & \le b \exp \biggl(\int_{-1}^\infty \biggl(\frac{\xi^2}{\xi^2 + r^2} - \frac{1}{1 + \lv r \rv}\biggr) \frac{1}{\lv r \rv} \, dr\biggr)
}
By a direct calculation, the exponent on the right-hand side is equal to $2 \log \xi - \tfrac{1}{2} \log(1 + \xi^2) + \log 2$, and hence
\formula{
 \lv \sigma + f(\xi) \rv & \le 2 b \, \frac{\xi^2}{\sqrt{1 + \xi^2}} \, .
}
Combining this with~\eqref{eq:rogers:constants}, we find that
\formula{
 c_2 & = \lim_{\xi \to \infty} \frac{\sigma + f(\xi)}{\xi^2} = 0 .
}
Taking this into account, and substituting $z = 1 + i \xi$, we rewrite~\eqref{eq:exp:inv:rogers:2} as
\formula[]{
\notag
 F(z) & = c_1 (1 - z) + \frac{1}{\pi} \int_{[-1, 0) \cup (0, \infty)} \biggl(\frac{1 - z}{1 - z + r} - \frac{(1 - z) \sign r}{1 + \lv r \rv}\biggr) \frac{\mu(dr)}{\lv r \rv} \\
\label{eq:exp:inv:rogers:3}
 & = c_1 (1 - z) + \frac{1}{\pi} \int_{[-1, 0)} \biggl(\frac{1 - z}{1 - z + r} + \frac{1 - z}{1 - r}\biggr) \frac{\mu(dr)}{-r} \\
\notag
 & \hspace*{12em} + \frac{1}{\pi} \int_{(0, \infty)} \biggl(\frac{1 - z}{1 - z + r} - \frac{1 - z}{1 + r}\biggr) \frac{\mu(dr)}{r} \, .
}

\emph{Step 5.}
We now study the value of $c_1$. By the dominated convergence theorem for the former integral in~\eqref{eq:exp:inv:rogers:3}, and the monotone convergence theorem for the latter one, we have
\formula{
 \lim_{z \to -\infty} \frac{F(z)}{1 - z} & = c_1 + \frac{1}{\pi} \int_{[-1, 0)} \frac{1}{-r (1 - r)} \, \mu(dr) - \frac{1}{\pi} \int_{(0, \infty)} \frac{1}{r (1 + r)} \, \mu(dr) .
}
On the other hand, by definition, $\sigma + F(z) > 0$ for $z < 0$, and hence the above limit is nonnegative. It follows that $\min\{\lv r \rv^{-1}, r^{-2}\}$ (rather than $\min\{\lv r \rv^{-1}, \lv r \rv^{-3}\}$) is integrable with respect to $\mu$, and
\formula[eq:exp:inv:rogers:c1]{
 c_1 + \frac{1}{\pi} \int_{[-1, 0)} \frac{1}{-r (1 - r)} \, \mu(dr) - \frac{1}{\pi} \int_{(0, \infty)} \frac{1}{r (1 + r)} \, \mu(dr) & \ge 0.
}

\emph{Step 6.}
Once again, we use the substitutions~\eqref{eq:sub:1} and~\eqref{eq:sub:2} of Proposition~\ref{prop:exp}: we substitute $s = 1 + r$ in the former integral in~\eqref{eq:exp:inv:rogers:3}, and $s = (1 + r)^{-1}$ in the latter one. In order to do so, we define
\formula{
 \mu_+(A) & = \int_{(0, \infty)} \ind_A((1 + r)^{-1}) \, \frac{\mu(dr)}{(1 + r)^2} \, , \\
 \mu_-(A) & = \int_{[-1, 0)} \ind_A(1 + r) \, \mu(dr) .
}
Since $(-r)^{-1}$ is integrable with respect to $\mu$ over $[-1, 0)$, we find that $(1 - s)^{-1}$ is integrable with respect to $\mu_-$ over $[0, 1)$. Similarly, $r^{-1} (1 + r)^{-1}$ is integrable with respect to $\mu$ over $(0, \infty)$, and so $(1 - s)^{-1}$ is integrable with respect to $\mu_+$ over $(0, 1)$. Formula~\eqref{eq:exp:inv:rogers:3} now takes the following form:
\formula{
 F(z) & = c_1 (1 - z) + \frac{1}{\pi} \int_{[0, 1)} \biggl(\frac{1 - z}{s - z} + \frac{1 - z}{2 - s}\biggr) \frac{\mu_-(ds)}{1 - s} + \frac{1}{\pi} \int_{(0, 1)} \biggl(\frac{1 - z}{s^{-1} - z} - s (1 - z)\biggr) \frac{\mu_+(ds)}{s (1 - s)} \\
 & = \tilde c_1 (1 - z) + \frac{1}{\pi} \int_{[0, 1)} \biggl(\frac{1}{1 - s} - \frac{z^{-1}}{1 - s z^{-1}}\biggr) \mu_-(ds) + \frac{1}{\pi} \int_{(0, 1)} \biggl(\frac{1}{1 - s} - \frac{z}{1 - s z}\biggr) \mu_+(ds)
}
when $z \in \C \setminus [0, \infty)$, with
\formula{
 \tilde c_1 & = c_1 + \frac{1}{\pi} \int_{[0, 1)} \frac{1}{(2 - s) (1 - s)} \, \mu_-(ds) - \frac{1}{\pi} \int_{(0, 1)} \frac{1}{1 - s} \, \mu_+(ds) \\
 & = c_1 + \frac{1}{\pi} \int_{[-1, 0)} \frac{1}{-r (1 - r)} \, \mu(dr) - \frac{1}{\pi} \int_{(0, \infty)} \frac{1}{r (1 + r)} \, \mu(dr) .
}
By~\eqref{eq:exp:inv:rogers:c1}, we have $\tilde c_1 \ge 0$. Thus, if we extend $\mu_+$ to a measure on $[0, 1)$ so that $\mu_+(\{0\}) = \pi \tilde c_1$, then we arrive at
\formula{
 F(z) & = \frac{1}{\pi} \int_{[0, 1)} \biggl(\frac{1}{1 - s} - \frac{z^{-1}}{1 - s z^{-1}}\biggr) \mu_-(ds) + \frac{1}{\pi} \int_{[0, 1)} \biggl(\frac{1}{1 - s} - \frac{z}{1 - s z}\biggr) \mu_+(ds) ,
}
which is precisely the Stieltjes representation~\eqref{eq:symbol} of the symbol of an $\amcm$ sequence.
\end{proof}

Inspecting the above proof, we find that it also yields the following result.

\begin{proposition}
\label{prop:characterisation}
Suppose that $F$ is a holomorphic function in $\C \setminus \R$, and $f(\xi) = F(1 + i \xi)$. Then $F$ is the symbol of a summable $\amcm$ sequence if and only if:
\begin{itemize}
\item $f$ is a Rogers function;
\item $f(0^+) = 0$;
\item $f$ extends to a holomorphic function on $\C \setminus (-i \infty, i]$;
\item this extension satisfies $f(i r) > 0$ for $r > 1$.
\end{itemize}
\end{proposition}

We conclude this part with the following observation.

\begin{remark}
\label{rem:compatible}
Let $F$ be the symbol of a nonzero summable $\amcm$ sequence $(a_k)$, let $\sigma \ge 0$, and let $f$ be a Rogers function. Consider the functions $\ph_+$ and $\ph_-$ in the exponential representation~\eqref{eq:exp} of $\sigma + F$, and the function $\ph$ in the exponential representation~\eqref{eq:rogers:exp} of the Rogers function $\sigma + f$. Let us say that the triple $\ph, \ph_+, \ph_-$ is compatible if $\ph(r) = \ph_-(1 + r)$ for $r \in (-1, 0)$, $\ph(r) = \ph_+((1 + r)^{-1})$ for $r \in (0, \infty)$, and $\ph(r) = 0$ for $r \in (-\infty, -1)$ (see~\eqref{eq:ph:m:1}, \eqref{eq:ph:p:1} and~\eqref{eq:ph:pm:2}). In the proofs of Propositions~\ref{prop:exp} and~\ref{prop:exp:inv}, we have observed that the triple $\ph, \ph_+, \ph_-$ is compatible if and only if $\sigma + f(\xi)$ is equal to $\sigma + F(1 + i \xi)$, up to multiplication by a constant. This constant disappears if additionally the constants $b$ and $c$ satisfy~\eqref{eq:exp:c}.
\end{remark}


\subsection{Spine of a Rogers function}

\begin{figure}
\centering
\setlength{\tabcolsep}{0pt}
\begin{tabular}{cc}
\begin{tikzpicture}
\footnotesize
\coordinate (X) at (3,0);
\coordinate (Y) at (0,3.5);
\coordinate (Xn) at (-2,0);
\coordinate (Yn) at (0,-2);
\coordinate (zeta) at (1,1);
\coordinate (modzetax) at (1.4142,0);
\coordinate (modzetay) at (0,1.4142);
\draw[white, very thick, fill=magenta!20!white] (0,0) .. controls (0.4,-0.4) and (1.15,0) .. (1,1) .. controls (0.85,2) and (0.6,2.5) .. (0,2.5);
\draw[white, very thick, fill=cyan!20!white] (0,0) .. controls (0.4,-0.4) and (1.15,0) .. (1,1) .. controls (0.85,2) and (0.6,2.5) .. (0,2.5) -- (0,3.5) -- (3,3.5) -- (3,-2) -- (0,-2);
\draw[white, very thick, fill=magenta!20!white] (0,0) .. controls (-0.4,-0.4) and (-1.15,0) .. (-1,1) .. controls (-0.85,2) and (-0.6,2.5) .. (0,2.5);
\draw[white, very thick, fill=cyan!20!white] (0,0) .. controls (-0.4,-0.4) and (-1.15,0) .. (-1,1) .. controls (-0.85,2) and (-0.6,2.5) .. (0,2.5) -- (0,3.5) -- (-2,3.5) -- (-2,-2) -- (0,-2);
\draw (modzetay) arc (90:-90:1.4142);
\node[above right] at (modzetax) {$r$};
\node[violet, below right] at (0.7,2.5) {$\Gamma_f$};
\draw[->] (Xn) -- (X) node[above] {$\re \xi$};
\draw[->] (Yn) -- (Y) node[left] {$\im \xi$};
\draw[very thick, violet, decoration={markings,
  mark=at position 0.15 with {\arrow{>}},
  mark=at position 0.4 with {\arrow{>}},
  mark=at position 0.8 with {\arrow{>}}
 }, postaction={decorate}] (0,0) .. controls (0.4,-0.4) and (1.15,0) .. (1,1) .. controls (0.85,2) and (0.6,2.5) .. (0,2.5);
\filldraw (zeta) circle[radius=1pt] node[above right] {$\zeta_f(r)$};
\filldraw (0,0) circle[radius=1pt];
\filldraw (0,2.5) circle[radius=1pt];
\end{tikzpicture}
&
\begin{tikzpicture}
\footnotesize
\coordinate (X) at (2.7,0);
\coordinate (Y) at (0,3);
\coordinate (Xn) at (-2.8,0);
\coordinate (Yn) at (0,-2.5);
\coordinate (zeta) at (-0.3,1);
\draw[white, very thick, fill=magenta!20!white] (0.7,0) .. controls (1.1,0.4) and (0.7,1.15) .. (-0.3,1) .. controls (-1.3,0.85) and (-1.8,0.6) .. (-1.8,0);
\draw[white, very thick, fill=cyan!20!white] (0.7,0) .. controls (1.1,0.4) and (0.7,1.15) .. (-0.3,1) .. controls (-1.3,0.85) and (-1.8,0.6) .. (-1.8,0) -- (-2.8,0) -- (-2.8,3) -- (2.7,3) -- (2.7,0);
\draw[white, very thick, fill=magenta!20!white] (0.7,0) .. controls (1.1,-0.4) and (0.7,-1.15) .. (-0.3,-1) .. controls (-1.3,-0.85) and (-1.8,-0.6) .. (-1.8,0);
\draw[white, very thick, fill=cyan!20!white] (0.7,0) .. controls (1.1,-0.4) and (0.7,-1.15) .. (-0.3,-1) .. controls (-1.3,-0.85) and (-1.8,-0.6) .. (-1.8,0) -- (-2.8,0) -- (-2.8,-2.5) -- (2.7,-2.5) -- (2.7,0);
\node[violet, above left] at (-1.3,0.7) {$\Gamma_F$};
\node at (-0.4,-0.4) {$D_F^+$};
\node at (-0.8,-1.8) {$D_F^-$};
\filldraw (0.7,0) circle[radius=1pt] node[below] {$1$};
\draw[->] (Xn) -- (X) node[above] {$\re x$};
\draw[->] (Yn) -- (Y) node[left] {$\im x$};
\draw[very thick, violet, decoration={markings,
  mark=at position 0.15 with {\arrow{>}},
  mark=at position 0.4 with {\arrow{>}},
  mark=at position 0.8 with {\arrow{>}}
 }, postaction={decorate}] (0.7,0) .. controls (1.1,0.4) and (0.7,1.15) .. (-0.3,1) .. controls (-1.3,0.85) and (-1.8,0.6) .. (-1.8,0);
\draw[very thick, violet, densely dashed, decoration={markings,
  mark=at position 0.15 with {\arrow{<}},
  mark=at position 0.4 with {\arrow{<}},
  mark=at position 0.8 with {\arrow{<}}
 }, postaction={decorate}] (0.7,0) .. controls (1.1,-0.4) and (0.7,-1.15) .. (-0.3,-1) .. controls (-1.3,-0.85) and (-1.8,-0.6) .. (-1.8,0);
\filldraw (zeta) circle[radius=1pt] node[above] {$\gamma(\lambda)\,$};
\filldraw (-1.8,0) circle[radius=1pt];
\end{tikzpicture}
\\
(a)&(b)
\end{tabular}
\caption{(a)~The spine $\Gamma_f$ (purple line) of a Rogers function $f$. (b)~The spine $\Gamma$ (solid purple line), the symmetrised spine $\Gamma^\star$ (solid and dashed purple line), and the regions $D_F^+$ (red) and $D_F^-$ (blue), for the symbol $F$ of an $\amcm$ sequence.}
\label{fig:spine}
\end{figure}
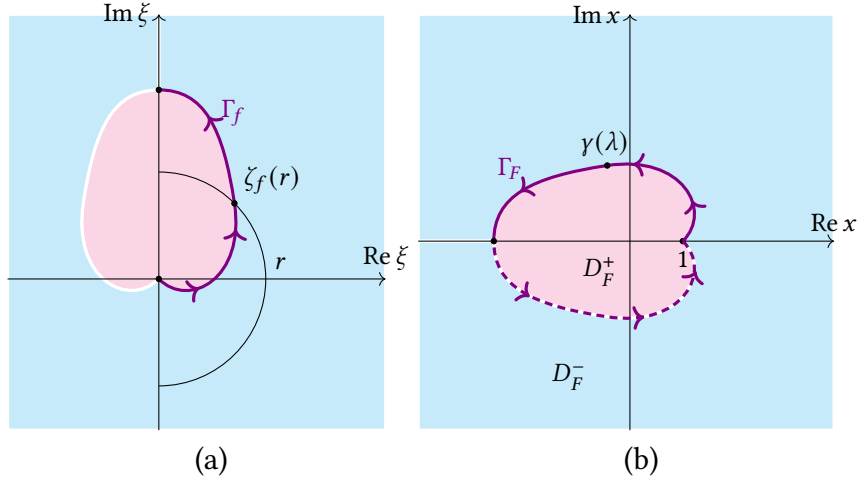

The \emph{spine} of a nonconstant Rogers function $f$ is the system of curves
\formula{
 \Gamma_f & = \{\xi \in \hp : f(\xi) \in (0, \infty)\}
}
(Definition~4.1 in~\cite{k19}); see Figure~\ref{fig:spine}(a). Basic properties of the spine are given in the following result, taken from~\cite{k19}. Note, however, that item~\ref{it:rogers:spine:d}, with a better constant, was proved earlier in~\cite{dh10}, and an optimal constant was found in~\cite{s13}.

\begin{proposition}[Theorem~4.2 in~\cite{k19}]
\label{prop:rogers:spine}
Let $f$ be a nonconstant Rogers function. There is a unique continuous complex-valued function $\zeta_f$ on $(0, \infty)$ with the following properties.
\begin{enumerate}[label=\rm (\alph*)]
\item\label{it:rogers:spine:a} We have $\lv \zeta_f(r) \rv = r$ and $\arg \zeta_f(r) \in [-\tfrac{\pi}{2}, \tfrac{\pi}{2}]$ for all $r > 0$.
\item\label{it:rogers:spine:b} If $\re \xi > 0$ and $r = \lv \xi \rv$, then
\formula{
 \sign \im f(\xi) & = \sign (\arg \xi - \arg \zeta_f(r)) .
}
\item\label{it:rogers:spine:c}
The spine $\Gamma_f$ is the union of pairwise disjoint simple real-analytic curves, which begin and end at the imaginary axis or at infinity. Furthermore, $\Gamma_f$ has parameterisation
\formula{
 \Gamma_f & = \{\zeta_f(r) : r \in Z_f\} ,
}
where
\formula{
 Z_f & = \{r \in (0, \infty) : \arg \zeta_f(r) \in (-\tfrac{\pi}{2}, \tfrac{\pi}{2})\} .
}
\item\label{it:rogers:spine:d}
For every $r > 0$, the spine $\Gamma_f$ restricted to the annular region $r \le \lv \xi \rv \le 2 r$ is a system of rectifiable curves of total length at most $C r$, where one can take $C = 300$.
\end{enumerate}
\end{proposition}

The next result describes the values of $f$ along the spine.

\begin{proposition}[Theorem~4.3 in~\cite{k19}]
\label{prop:rogers:lambda}
Suppose that $f$ is a nonconstant Rogers function.
\begin{enumerate}[label=\rm (\alph*)]
\item\label{it:rogers:lambda:a} For every $r \in (0, \infty) \setminus \partial Z_f$ we have $\zeta_f(r) \in D_f$.
\item\label{it:rogers:lambda:b} The function $\lambda_f(r)$, defined for $r \in (0, \infty) \setminus \partial Z_f$ by
\formula{
 \lambda_f(r) & = f(\zeta_f(r)) ,
}
extends in a unique way to a continuous, strictly increasing function of $r \in (0, \infty)$, and $\lambda_f'(r) > 0$ for every $r \in (0, \infty) \setminus \partial Z_f$.
\item\label{it:rogers:lambda:c} We have $\lambda_f(0^+) = f(0^+)$, and $\lambda_f(\infty^-) = f(\infty^-)$.
\end{enumerate}
\end{proposition}

We also need the following regularity property of the spine.

\begin{proposition}[Propositions~3.14 and~3.15 in~\cite{k25}]
\label{prop:rogers:holder}
If $f$ is a nonconstant Rogers function, then $\zeta_f$ and $\lambda_f^{-1}$ are locally Hölder continuous functions on $(0, \infty)$ and $(\lambda_f(0^+), \lambda_f(\infty^-))$, with exponents $\tfrac{1}{30}$ and $\tfrac{1}{3}$, respectively.
\end{proposition}

Using $r = \lv \zeta_f(r) \rv$ as a parameter in the description of $\Gamma_f$ was very natural in~\cite{k19,k25}. In our case, however, it will be more convenient to use $\lambda_f(r) = f(\zeta_f(r))$ instead. Since $\lambda_f$ is a continuous increasing function on $(0, \infty)$, $\zeta_f(\lambda_f^{-1}(\lambda))$ is well-defined for $\lambda \in (\lambda_f(0^+), \lambda_f(\infty^-))$, and it is a parameterisation of $\Gamma_f$ when $\lambda$ is restricted to the set $\lambda_f(Z_f)$.


\subsection{Spine of the symbol of an \texorpdfstring{$\amcm$}{AM/CM} sequence}
\label{sec:spine}

We turn to the analogues of the above results for the symbol $F$ of an $\amcm$ sequence.

\begin{definition}
\label{def:spine}
Let $F$ be the symbol of a nontrivial summable $\amcm$ sequence. The \emph{spine} of $F$ is the system of curves
\formula{
 \Gamma & = \{z \in \C : \im z > 0 , F(z) \in (0, \infty)\}
}
(see Figure~\ref{fig:spine}(b)). We also define
\formula{
 \Lambda & = F(\Gamma) = \{F(z) : z \in \Gamma\} .
}
\end{definition}

The above definition has the following symmetry property: when $(a_k)$ is changed to the reversed sequence $(\check a_k) = (a_{-k})$, then $\check\Gamma = \overline{\Gamma^{-1}} = \{z \in \C \setminus \{0\} : \overline{z^{-1}} \in \Gamma\}$ is the spine of the corresponding dual symbol $\check F(z) = F(z^{-1})$, and the set $\Lambda$ remains unchanged.

Propositions~\ref{prop:rogers:spine} and~\ref{prop:rogers:lambda} apply to both Rogers functions $f(\xi) = F(1 + i \xi)$ and $g(\xi) = F((1 + i \xi)^{-1})$, and this leads to further regularity of $\Gamma$.

\begin{proposition}
\label{prop:spine}
Let $F$ be the symbol of a two-sided summable $\amcm$ sequence. Then there is a number $\limit \in (0, \infty)$ such that $\limit \ge \sup \Lambda$, and a unique continuous complex-valued function $\gamma$ on $(0, \limit)$, with the following properties.
\begin{enumerate}[label=\rm (\alph*)]
\item\label{it:spine:a} We have $\im \gamma(\lambda) \ge 0$ for all $\lambda \in (0, \limit)$. We also have $\gamma(\lambda) \in D_F$ and $F(\gamma(\lambda)) = \lambda$ when $\lambda \in (0, \limit) \setminus \partial \Lambda$. Furthermore, $\gamma(0^+) = 1$, and $\gamma(\limit^-)$ is either a nonpositive real number or complex infinity. If $\gamma(\limit^-) \in (-\infty, 0)$, then $\limit = \sup \Lambda$ and $F(\gamma(\limit^-)) = \limit$.
\item\label{it:spine:b}
Both $\lv \gamma(\lambda) - 1 \rv$ and $\lv (\gamma(\lambda))^{-1} - 1 \rv$ are strictly increasing functions of $\lambda \in (0, \limit)$.
\item\label{it:spine:c}
The spine $\Gamma$ is the union of a countable family of pairwise disjoint simple real-analytic curves. All of them begin and end on $(0, \infty)$, except at most one, which may end at $\gamma(\limit^-)$.
\item\label{it:spine:d}
For every $r > 0$, the spine $\Gamma$ restricted to the disk $\{z \in \C : \lv z \rv < r\}$ is a system of rectifiable curves of total length at most $C (1 + r)$, where one can take $C = 300$. The same applies to the image of the spine under the inversion $z \mapsto z^{-1}$. In particular, if $\gamma(\limit^-) \in (-\infty, 0)$, then both $\Gamma$ and its image under the inversion have finite length.
\item\label{it:spine:e}
For every $z \in (-\infty, 0)$ we have $F(z) \ge \limit$, and equality holds only when $z = \gamma(\limit^-)$. Furthermore, $\limit \le 2 \mass$.
\item\label{it:spine:f}
The function $\gamma$ is locally Hölder continuous on $(0, \limit)$, with exponent $\tfrac{1}{90}$.
\end{enumerate}
For a one-sided $\amcm$ sequence, we have $\Gamma = \varnothing$ and $\Lambda = \varnothing$.
\end{proposition}

\begin{proof}
Suppose that $F$ is the symbol of a nontrivial $\amcm$ sequence. We divide the argument into seven steps.

\begin{figure}
\centering
\setlength{\tabcolsep}{0pt}
\begin{tabular}{cc}
\begin{tikzpicture}
\footnotesize
\coordinate (X) at (2.4,0);
\coordinate (Y) at (0,3);
\coordinate (Xn) at (-4.8,0);
\coordinate (Yn) at (0,-1.6);
\node[cyan, above] at (-4.2,0.1) {$\gamma_+(\lambda)$};
\node[magenta, above] at (-1.5,0.1) {$\gamma_-(\lambda)$};
\draw[->] (Xn) -- (X) node[above] {$\re z$};
\draw[->] (Yn) -- (Y) node[left] {$\im z$};
\draw[very thick, side by side=magenta:cyan, violet, decoration={markings,
  mark=at position 0.15 with {\arrow{>}},
  mark=at position 0.4 with {\arrow{>}},
  mark=at position 0.8 with {\arrow{>}}
 }, postaction={decorate}] (1.4,0) .. controls (2.2,0.8) and (1.4,2.3) .. (-0.6,2) .. controls (-2.6,1.7) and (-3.6,1.2) .. (-3.6,0);
\draw[very thick, cyan, decoration={markings,
  mark=at position 0.6 with {\arrow{>}}
 }, postaction={decorate}] (-3.6,0) -- (Xn);
\draw[very thick, magenta, decoration={markings,
  mark=at position 0.6 with {\arrow{>}}
 }, postaction={decorate}] (-3.6,0) -- (0,0);
\filldraw (1.4,0) circle[radius=1pt] node[below] {$1$};
\filldraw (-3.6,0) circle[radius=1pt] node[violet, below] {$\gamma(\limit^-)$};
\end{tikzpicture}
&
\begin{tikzpicture}
\footnotesize
\coordinate (X) at (2.4,0);
\coordinate (Y) at (0,3);
\coordinate (Xn) at (-4.8,0);
\coordinate (Yn) at (0,-1.6);
\node[cyan, below left] at (1,0.6) {$\check\gamma_+(\lambda)$};
\node[magenta, above] at (-0.6,1) {$\gamma_-(\lambda)$};
\draw (0,0) circle[radius=1.4];
\draw[->] (Xn) -- (X) node[above] {$\re z$};
\draw[->] (Yn) -- (Y) node[left] {$\im z$};
\draw[very thick, magenta, decoration={markings,
  mark=at position 0.15 with {\arrow{>}},
  mark=at position 0.4 with {\arrow{>}},
  mark=at position 0.8 with {\arrow{>}}
 }, postaction={decorate}] (1.4,0) .. controls (2.2,0.8) and (1.4,2.3) .. (-0.6,2) .. controls (-2.6,1.7) and (-3.6,1.2) .. (-3.6,0);
\draw[very thick, magenta, decoration={markings,
  mark=at position 0.6 with {\arrow{>}}
 }, postaction={decorate}] (-0.5444,0) -- (0,0);
\draw[very thick, cyan, decoration={markings,
  mark=at position 0.15 with {\arrow{>}},
  mark=at position 0.4 with {\arrow{>}},
  mark=at position 0.8 with {\arrow{>}}
 }, postaction={decorate}] (1.4,0) .. controls (0.6,0.8) and (0.2802,1.154) .. (-0.2698,0.899) .. controls (-0.8196,0.6442) and (-0.5444,0.1814) .. (-0.5444,0);
\draw[very thick, cyan, decoration={markings,
  mark=at position 0.6 with {\arrow{>}}
 }, postaction={decorate}] (-3.6,0) -- (Xn);
\draw[very thick, side by side=cyan:magenta, decoration={markings,
  mark=at position 0.4 with {\arrow[cyan]{<}},
  mark=at position 0.6 with {\arrow[magenta]{>}}
 }, postaction={decorate}] (-3.6,0) -- (-0.5444,0);
\filldraw (1.4,0) circle[radius=1pt] node[below right] {$1$};
\filldraw (-3.6,0) circle[radius=1pt] node[violet, below] {$\gamma(\limit^-)$};
\end{tikzpicture}
\\
(a)&(b)
\end{tabular}
\caption{(a)~The curves parameterised by $\gamma_+$ (cyan) and $\gamma_-$ (magenta). (b)~The curve parameterised by $\gamma_-$ (magenta) and its image under inversion with respect to the unit circle, parameterised by $\check\gamma_+$ (cyan).}
\label{fig:spine:pm}
\end{figure}
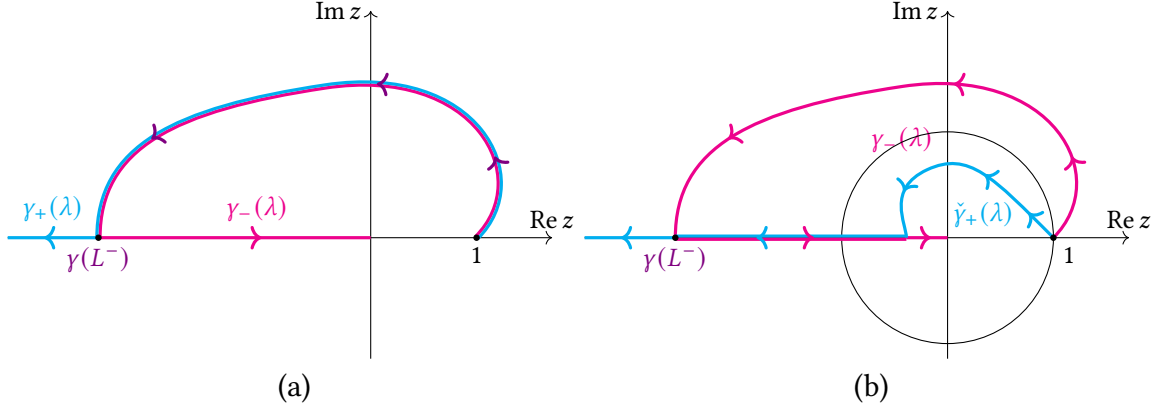

\emph{Step 1.}
By Lemma~\ref{lem:symbol:rogers}, $f(\xi) = F(1 + i \xi)$ defines a Rogers function of $\xi$, with $f(0^+) = 0$. We denote $\limit_+ = f(\infty^-) \in (0, \infty]$. If we set $z = 1 + i \xi$, then $\im z > 0$ if and only if $\re \xi > 0$, and $F(z) \in (0, \infty)$ if and only if $f(\xi) \in (0, \infty)$. In other words, $z \in \Gamma$ if and only if $\xi \in \Gamma_f$. By Proposition~\ref{prop:rogers:spine},
\formula{
 \Gamma & = \{ 1 + i \zeta_f(r) : r \in Z_f \}
}
and $\Lambda = F(\Gamma) = \lambda_f(Z_f) \subseteq (0, \limit_+)$.

It will be more convenient for us to use $\lambda = F(1 + i \zeta_f(r)) = f(\zeta_f(r)) = \lambda_f(r)$ instead of $r$ as a parameter. Recall that by Proposition~\ref{prop:rogers:lambda}, $\lambda_f$ is a continuous increasing function which maps $(0, \infty)$ onto $(f(0^+), f(\infty^-)) = (0, \limit_+)$. The inverse function $\lambda_f^{-1}$ maps $(0, \limit_+)$ onto $(0, \infty)$, and for $\lambda \in (0, \limit_+)$ we define
\formula[eq:gamma:p]{
 \gamma_+(\lambda) & = 1 + i \zeta_f(\lambda_f^{-1}(\lambda))
}
(see Figure~\ref{fig:spine:pm}(a)). We set $\gamma_+(0) = 1$ and $\gamma_+(\limit_+)$ to be the complex infinity $\cinfty$, so that $\gamma_+$ is continuous on $[0, \limit_+]$ with values on the Riemann sphere $\C \cup \{\cinfty\}$. Since $Z_f = \lambda_f^{-1}(\Lambda)$, we have
\formula[]{
\label{eq:lambda:p:g}
 \Gamma & = \{ \gamma_+(\lambda) : \lambda \in \Lambda \} , \\
\label{eq:lambda:p:l}
 \Lambda & = \{ \lambda \in (0, \limit_+) : \im \gamma_+(\lambda) > 0 \} .
}
Finally, consider $\lambda \in (0, \limit_+) \setminus \partial \Lambda$. Denote $z = \gamma_+(\lambda)$ and $r = \lambda_f^{-1}(\lambda)$. Then $r \in (0, \infty) \setminus Z_f$, and hence, by Proposition~\ref{prop:rogers:spine}, we have $\zeta_f(r) \in D_f$. It follows that $\gamma_+(\lambda) = 1 + i \zeta_f(r) \in D_F$. Furthermore, $\lambda = \lambda_f(r) = f(\zeta_f(r)) = F(\gamma_+(\lambda)) = F(z)$. In other words,
\formula[eq:gamma:p:f1]{
 F(\gamma_+(\lambda)) & = \lambda
}
for $\lambda \in (0, \limit_+) \setminus \partial \Lambda$.

\emph{Step 2.}
We repeat the first step for the dual symbol $\check F(z) = F(z^{-1})$ and the corresponding Rogers function $g$ from Lemma~\ref{lem:symbol:rogers}. This way we obtain a function
\formula{
 \check \gamma_+(\lambda) & = 1 + i \zeta_g(\lambda_g^{-1}(\lambda)) ,
}
defined for $\lambda \in [0, \limit_-]$, where $\limit_- = g(\infty^-) \in (0, \infty]$. The function $\check \gamma_+$ restricted to $\Lambda$ parameterises the spine $\check \Gamma = \{z \in \C : \im z > 0, F(z^{-1}) \in (0, \infty)\}$ of $\check F$.

Recall that $F(\overline z) = \overline{F(z)}$. Thus, the following conditions are equivalent when $\im z > 0$: $z \in \check \Gamma$; $F(z^{-1}) \in (0, \infty)$, $F(\overline z^{-1}) \in (0, \infty)$; $\overline z^{-1} \in \Gamma$. It follows that $\Gamma$ is parameterised by $(\overline{\check \gamma_+(\lambda)})^{-1}$, which we denote by $\gamma_-(\lambda)$. To be specific,
\formula[eq:gamma:m]{
 \gamma_-(\lambda) & = (\overline{\check \gamma_+(\lambda)})^{-1} = (1 - i \overline{\zeta_g(\lambda_g^{-1}(\lambda))})^{-1}
}
for $\lambda \in (0, \limit_-)$, with $0^{-1} = \cinfty$, and $\gamma_-(0) = 1$, $\gamma_-(\limit_-) = 0$; see Figure~\ref{fig:spine:pm}(b). We have already proved that
\formula{
 \Gamma & = \{ \gamma_-(\lambda) : \lambda \in \Lambda \} , \\
 \Lambda & = \{ \lambda \in (0, \limit_-) : \im \gamma_-(\lambda) > 0 \} .
}
Finally, for $\lambda \in (0, \limit_-) \setminus \partial \Lambda$, except possibly one value for which $\gamma_+(\lambda) = \cinfty$, we have
\formula{
 F(\gamma_-(\lambda)) & = F(\overline{\gamma_-(\lambda)}) = \check F(\check \gamma_+(\lambda)) = \lambda .
}

\emph{Step 3.}
Observe that if $\lambda \in \Lambda$, then $\gamma_+(\lambda), \gamma_-(\lambda) \in \Gamma$ and $F(\gamma_+(\lambda)) = \lambda = F(\gamma_-(\lambda))$. Since $F$ is one-to-one on $\Gamma$, it follows that
\formula[eq:gamma:pm]{
 \gamma_+(\lambda) & = \gamma_-(\lambda)
}
for $\lambda \in \Lambda$ and hence, by continuity, for $\lambda \in \Cl \Lambda$. By definition, the above equality also holds for $\lambda = 0$: $\gamma_+(0) = \gamma_-(0) = 1$.

In Steps~5 and~6, we extend~\eqref{eq:gamma:pm} to $\lambda \in [0, \limit]$ for an appropriate $\limit$.

\emph{Step 4.}
Consider $\lambda_1, \lambda_2 \in \{0\} \cup \partial \Lambda$ such that $\lambda_1 < \lambda_2$. Since~\eqref{eq:gamma:pm} holds for $\lambda = \lambda_1$ and $\lambda = \lambda_2$, we may define $x_1 = \gamma_+(\lambda_1) = \gamma_-(\lambda_1)$ and $x_2 = \gamma_+(\lambda_2) = \gamma_-(\lambda_2)$. By~\eqref{eq:lambda:p:g} and~\eqref{eq:lambda:p:l}, we have $x_1, x_2 \in \R \cup \{\cinfty\}$, and by~\eqref{eq:gamma:p}, $\lv x_1 - 1 \rv < \lv x_2 - 1 \rv$. Similarly, by~\eqref{eq:gamma:m}, $\lv x_1^{-1} - 1 \rv < \lv x_2^{-1} - 1 \rv$. We consider four possibilities.
\begin{itemize}
\item If $x_1 \in (-\infty, 0)$, then $\lv x_2 - 1 \rv > 1 - x_1 > 1$ and $\lv x_2^{-1} - 1 \rv > 1 - x_1^{-1} > 1$. The former condition fails for $x_2 \in [0, 1]$, while the latter one for $x_2 \in [1, \infty) \cup \{\cinfty\}$. Hence, $x_2 < 0$, and we have $1 - x_2 > 1 - x_1$ and $1 - x_2^{-1} > 1 - x_1^{-1} > 1$, a contradiction. Thus, $x_1 \in (-\infty, 0)$ is impossible.
\item Since $\lv x_1 - 1 \rv < \lv x_2 - 1 \rv$, we cannot have $x_1 = \cinfty$. Similarly, $\lv x_1^{-1} - 1 \rv < \lv x_2^{-1} - 1 \rv$ makes $x_1 = 0$ impossible.
\end{itemize}
It follows that necessarily $x_1 \in (0, \infty)$. Furthermore, if $\lambda_2 < \sup \Lambda$, then the above arguments applied to the pair $(\lambda_2, \sup \Lambda)$ instead of $(\lambda_1, \lambda_2)$ imply that $x_2 \in (0, \infty)$. Thus, we are left with two possibilities:
\begin{itemize}
\item $x_1, x_2 \in (0, \infty)$;
\item $x_1 \in (0, \infty)$, $x_2 \in \{\cinfty\} \cup (-\infty, 0]$, and $\lambda_2 = \sup \Lambda$.
\end{itemize}

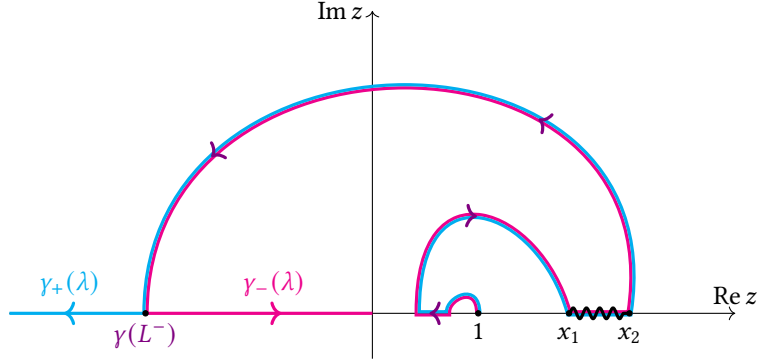
\begin{figure}
\centering
\begin{tikzpicture}
\footnotesize
\coordinate (X) at (4.8,0);
\coordinate (Y) at (0,4);
\coordinate (Xn) at (-4.8,0);
\coordinate (Yn) at (0,-0.6);
\node[cyan, above] at (-4,0.1) {$\gamma_+(\lambda)$};
\node[magenta, above] at (-1.3,0.1) {$\gamma_-(\lambda)$};
\draw[->] (Xn) -- (X) node[above] {$\re z$};
\draw[->] (Yn) -- (Y) node[left] {$\im z$};
\draw[very thick, side by side=magenta:cyan, violet, decoration={markings,
  mark=at position 0.06 with {\arrow{>}},
  mark=at position 0.18 with {\arrow{>}},
  mark=at position 0.55 with {\arrow{>}},
  mark=at position 0.85 with {\arrow{>}}
 }, postaction={decorate}] (1.4,0) .. controls (1.4,0.4) and (1,0.2) .. (1,0) -- (0.6,0) .. controls (0.6,2) and (2.2,1.4) .. (2.6,0) -- (3.4,0) .. controls (4,4) and (-3,4) .. (-3,0);
\draw[very thick, cyan, decoration={markings,
  mark=at position 0.6 with {\arrow{>}}
 }, postaction={decorate}] (-3,0) -- (Xn);
\draw[very thick, magenta, decoration={markings,
  mark=at position 0.6 with {\arrow{>}}
 }, postaction={decorate}] (-3,0) -- (0,0);
\filldraw (1.4,0) circle[radius=1pt] node[below] {$1$};
\filldraw (-3,0) circle[radius=1pt] node[violet, below] {$\gamma(\limit^-)$};
\filldraw (2.6,0) circle[radius=1pt] node[below] {$x_1$\vphantom{0}};
\filldraw (3.4,0) circle[radius=1pt] node[below] {$x_2$\vphantom{0}};
\draw[very thick, decorate, decoration={snake, segment length=5pt, amplitude=2pt}] (2.6,0) -- (3.4,0);
\end{tikzpicture}
\caption{An interval $[x_1, x_2]$ in Step~5 of the proof of Proposition~\ref{prop:spine}.}
\label{fig:spine:interval}
\end{figure}

\emph{Step 5.}
Consider a maximal interval $(\lambda_1, \lambda_2)$ in $(0, \sup \Lambda) \setminus \Cl \Lambda$. Then $0 \le \lambda_1 < \lambda_2 < \sup \Lambda$ and $\lambda_1, \lambda_2 \in \partial \Lambda \cup \{0\}$. Hence, \eqref{eq:gamma:pm} holds for $\lambda = \lambda_1$ and $\lambda = \lambda_2$: we have $\gamma_+(\lambda_1) = \gamma_-(\lambda_1)$ and $\gamma_+(\lambda_2) = \gamma_-(\lambda_2)$. We denote these two points by $x_1$ and $x_2$, respectively. By Step~4, we have $x_1, x_2 \in (0, \infty)$.

Let $r_1 = \lambda_f^{-1}(\lambda_1) \in [0, \infty)$ and $r_2 = \lambda_f^{-1}(\lambda_2) \in (0, \infty)$. Then $\zeta_f$ is a linear function on $[r_1, r_2]$: either $\zeta_f(r) = i r$ for $r \in [r_1, r_2]$, or $\zeta_f(r) = -i r$ for $r \in [r_1, r_2]$. Furthermore, $\zeta_f(r) \in D_f$ for $r \in (r_1, r_2)$, and $\lambda_f(r) = f(\zeta_f(r))$ for $r \in (r_1, r_2)$.

It follows that $\gamma_+$ maps $[\lambda_1, \lambda_2]$ onto the interval $[\gamma_+(\lambda_1), \gamma_+(\lambda_2)] = [x_1, x_2]$, and either $x_1 = 1 - r_1$, $x_2 = 1 - r_2$, or $x_1 = 1 + r_1$, $x_2 = 1 + r_2$. Therefore, $0 < x_2 < x_1 \le 1$ or $1 \le x_1 < x_2$. Furthermore, $\gamma_+(\lambda) \in D_F$ for $\lambda \in (\lambda_1, \lambda_2)$, and
\formula[eq:gamma:p:f2]{
 F(\gamma_+(\lambda)) & = \lambda
}
for $\lambda \in (\lambda_1, \lambda_2)$. In particular, $F$ is strictly monotone on $(x_1, x_2)$.

By the same argument, $\check \gamma_+$ maps $[\lambda_1, \lambda_2]$ onto the interval $[\check \gamma_+(\lambda_1), \check \gamma_+(\lambda_2)] = [\check x_1, \check x_2]$. Furthermore, $\check \gamma_+(\lambda) \in D_{\check F}$ and $\check F(\check \gamma_+(\lambda)) = \lambda$ for $\lambda \in (\lambda_1, \lambda_2)$.

Observe that $\gamma_-(\lambda) = (\overline{\check \gamma_+(\lambda)})^{-1} = (\check \gamma_+(\lambda))^{-1}$ for $\lambda \in [\lambda_1, \lambda_2]$. Also, $x_1 = \gamma_-(\lambda_1) = (\check \gamma_+(\lambda_1))^{-1} = (\check x_1)^{-1}$, and similarly $x_2 = (\check x_2)^{-1}$. It follows that $\gamma_-$ maps $[\lambda_1, \lambda_2]$ onto the interval $[x_1, x_2]$. Furthermore, $\gamma_-(\lambda) \in D_F$ and $F(\gamma_-(\lambda)) = \lambda$ for $\lambda \in (\lambda_1, \lambda_2)$.

It follows that for every $\lambda \in (\lambda_1, \lambda_2)$, both $x = \gamma_+(\lambda)$ and $x = \gamma_-(\lambda)$ satisfy $x \in (x_1, x_2)$ and $F(x) = \lambda$. Since $F$ is strictly monotone on $(x_1, x_2)$, we have $\gamma_+(\lambda) = \gamma_-(\lambda)$, and so formula~\eqref{eq:gamma:pm} holds for $\lambda \in (\lambda_1, \lambda_2)$.

Since $(\lambda_1, \lambda_2)$ is an arbitrary maximal interval in $(0, \sup \Lambda) \setminus \Cl \Lambda$, and~\eqref{eq:gamma:pm} holds for $\lambda \in \{0\} \cup \Cl \Lambda$, we conclude that in fact~\eqref{eq:gamma:pm} holds for every $\lambda \in [0, \sup \Lambda]$.

\emph{Step 6.}
Let $\limit_0 = \sup \Lambda$. We now perform a similar analysis of $\gamma_+$ and $\gamma_-$ on intervals $[\limit_0, \limit_+]$ or $[\limit_0, \limit_-]$, respectively. Note that these intervals may be degenerate if $\limit_0 = \limit_+$ or $\limit_0 = \limit_-$.

We have $\gamma_+(\limit_+) = \cinfty$. Thus, if $x_1 = \gamma_+(\limit_0) = \gamma_-(\limit_0) \in \R \cup \{\cinfty\}$, then $\gamma_+$ maps $[\limit_0, \limit_+]$ onto an interval with endpoints $x_1$ and $\cinfty$: $[x_1, \infty) \cup \{\cinfty\}$ if $x_1 \in (1, \infty)$, $\{\cinfty\} \cup (-\infty, x_1]$ if $x_1 \in (-\infty, 1)$, either of them if $x_1 = 1$, and a degenerate interval $\{\cinfty\}$ if $x_1 = \cinfty$ (and necessarily $\limit_0 = \limit_+$). Furthermore, $\gamma_+(\lambda) \in D_F$ and $F(\gamma_+(\lambda)) = \lambda > \limit_0$ for $\lambda \in (\limit_0, \limit_+)$.

Similarly, $\check \gamma_+$ maps $[\limit_0, \limit_-]$ onto an interval with endpoints $\check x_1 = x_1^{-1}$ and $\cinfty$. It follows that $\gamma_-$ maps $[\limit_0, \limit_-]$ onto an arc of the circle $\R \cup \{\cinfty\}$ on the Riemann sphere with endpoints $x_1$ and $0$: $x_1$ is the left endpoint if $x_1^{-1} < 1$, the right endpoint if $x_1^{-1} > 1$, either of them if $x_1 = 1$, and the interval is degenerate if $x_1 = 0$ (and necessarily $\limit_0 = \limit_-$). Furthermore, $\gamma_-(\lambda) \in D_F$ and $F(\gamma_-(\lambda)) = \lambda > \limit_0$ for $\lambda \in (\limit_0, \limit_-)$, except perhaps a single value for which $\gamma_-(\lambda) = \cinfty$.

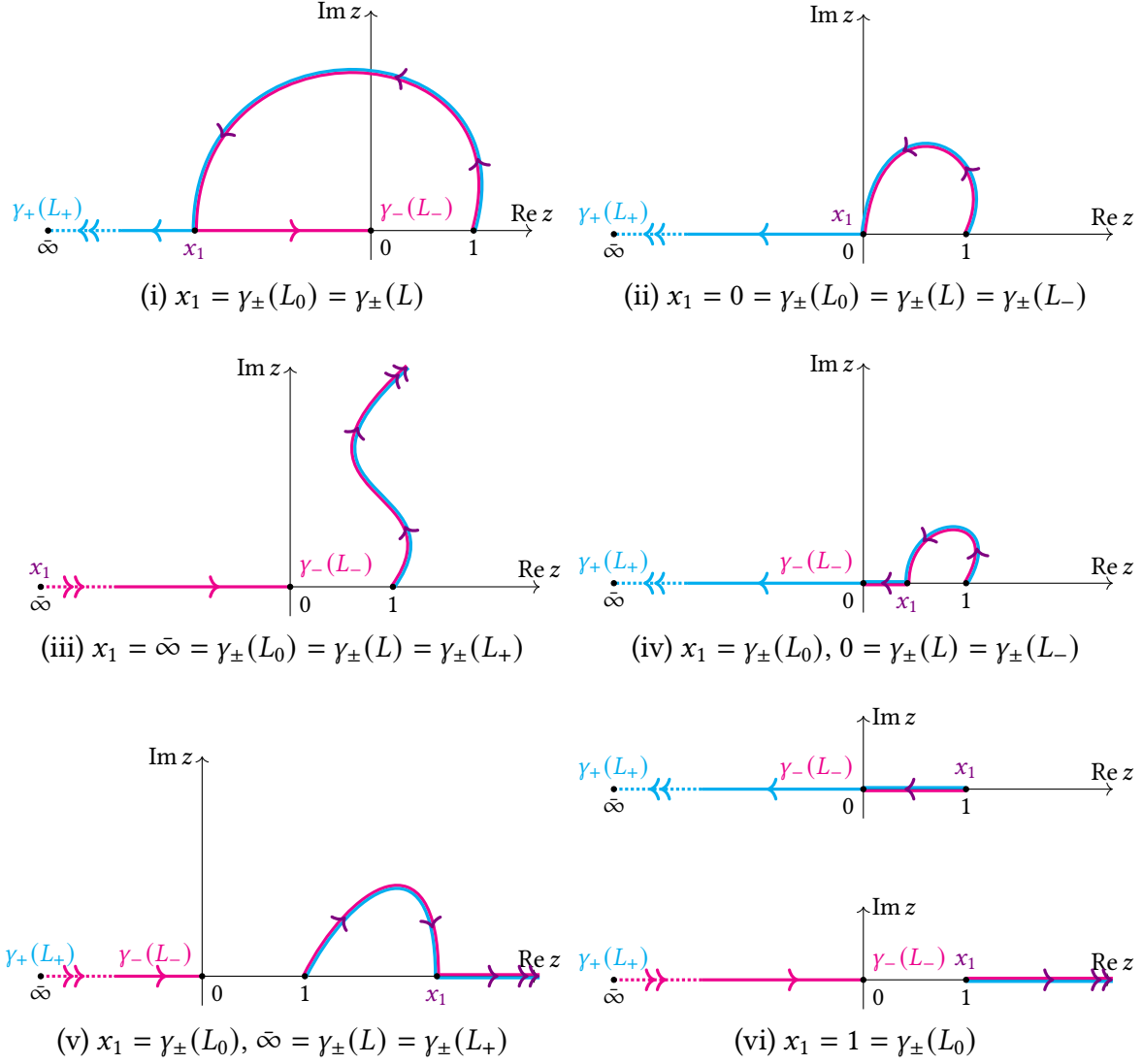
\begin{figure}
\centering
\setlength{\tabcolsep}{0pt}
\begin{tabular}{cc}
\begin{tikzpicture}
\footnotesize
\coordinate (X) at (2.2,0);
\coordinate (Y) at (0,3);
\coordinate (Xn) at (-3.4,0);
\coordinate (Yn) at (0,-0.4);
\draw[->] (Xn) -- (X) node[above] {$\re z$};
\draw[->] (Yn) -- (Y) node[left] {$\im z$};
\draw[very thick, side by side=magenta:cyan, violet, decoration={markings,
  mark=at position 0.15 with {\arrow{>}},
  mark=at position 0.4 with {\arrow{>}},
  mark=at position 0.8 with {\arrow{>}}
 }, postaction={decorate}] (1.4,0) .. controls (2.2,2.8) and (-2.4,3) .. (-2.4,0);
\draw[very thick, cyan, decoration={markings,
  mark=at position 0.6 with {\arrow{>}}
 }, postaction={decorate}] (-2.4,0) -- (Xn);
\draw[very thick, cyan, densely dotted, decoration={markings,
  mark=at position 0.6 with {\arrow{>>}}
 }, postaction={decorate}] (Xn) -- (-4.4,0);
\draw[very thick, magenta, decoration={markings,
  mark=at position 0.6 with {\arrow{>}}
 }, postaction={decorate}] (-2.4,0) -- (0,0);
\filldraw (1.4,0) circle[radius=1pt] node[below] {$1$};
\filldraw (-2.4,0) circle[radius=1pt] node[violet, below] {$x_1$\vphantom{0}};
\filldraw (0,0) circle[radius=1pt] node[magenta, above right] {$\gamma_-(\limit_-)$} node[below right] {$0$};
\filldraw (-4.4,0) circle[radius=1pt] node[cyan, above] {$\gamma_+(\limit_+)$} node[below] {$\cinfty$\vphantom{0}};
\end{tikzpicture}
&
\begin{tikzpicture}
\footnotesize
\coordinate (X) at (3.4,0);
\coordinate (Y) at (0,3);
\coordinate (Xn) at (-2.4,0);
\coordinate (Yn) at (0,-0.4);
\draw[->] (Xn) -- (X) node[above] {$\re z$};
\draw[->] (Yn) -- (Y) node[left] {$\im z$};
\draw[very thick, side by side=magenta:cyan, violet, decoration={markings,
  mark=at position 0.3 with {\arrow{>}},
  mark=at position 0.6 with {\arrow{>}}
 }, postaction={decorate}] (1.4,0) .. controls (2,1.2) and (0.2,2) .. (0,0);
\draw[very thick, cyan, decoration={markings,
  mark=at position 0.6 with {\arrow{>}}
 }, postaction={decorate}] (0,0) -- (Xn);
\draw[very thick, cyan, densely dotted, decoration={markings,
  mark=at position 0.6 with {\arrow{>>}}
 }, postaction={decorate}] (Xn) -- (-3.4,0);
\filldraw (1.4,0) circle[radius=1pt] node[below] {$1$};
\filldraw (0,0) circle[radius=1pt] node[violet, above left] {$x_1$} node[below left] {$0$};
\filldraw (-3.4,0) circle[radius=1pt] node[cyan, above] {$\gamma_+(\limit_+)$} node[below] {$\cinfty$\vphantom{0}};
\end{tikzpicture}
\\
(i) $x_1 = \gamma_\pm(\limit_0) = \gamma_\pm(\limit)$ &
(ii) $x_1 = 0 = \gamma_\pm(\limit_0) = \gamma_\pm(\limit) = \gamma_\pm(\limit_-)$ \\[1em]
\begin{tikzpicture}
\footnotesize
\coordinate (X) at (3.4,0);
\coordinate (Y) at (0,3);
\coordinate (Xn) at (-2.4,0);
\coordinate (Yn) at (0,-0.4);
\draw[->] (Xn) -- (X) node[above] {$\re z$};
\draw[->] (Yn) -- (Y) node[left] {$\im z$};
\draw[very thick, side by side=magenta:cyan, violet, decoration={markings,
  mark=at position 0.25 with {\arrow{>}},
  mark=at position 0.7 with {\arrow{>}},
  mark=at position 0.999 with {\arrow{>>}}
 }, postaction={decorate}] (1.4,0) .. controls (2.4,1.4) and (-0.4,1.2) .. (1.6,3);
\draw[very thick, magenta, decoration={markings,
  mark=at position 0.6 with {\arrow{>}}
 }, postaction={decorate}] (Xn) -- (0,0);
\draw[very thick, magenta, densely dotted, decoration={markings,
  mark=at position 0.6 with {\arrow{>>}}
 }, postaction={decorate}] (-3.4,0) -- (Xn);
\filldraw (1.4,0) circle[radius=1pt] node[below] {$1$};
\filldraw (0,0) circle[radius=1pt] node[magenta, above right] {$\gamma_-(\limit_-)$} node[below right] {$0$};
\filldraw (-3.4,0) circle[radius=1pt] node[cyan, above] {\phantom{$\gamma_+(\limit_+)$}} node[violet, above] {$x_1$} node[below] {$\cinfty$\vphantom{0}};
\end{tikzpicture}
&
\begin{tikzpicture}
\footnotesize
\coordinate (X) at (3.4,0);
\coordinate (Y) at (0,3);
\coordinate (Xn) at (-2.4,0);
\coordinate (Yn) at (0,-0.4);
\draw[->] (Xn) -- (X) node[above] {$\re z$};
\draw[->] (Yn) -- (Y) node[left] {$\im z$};
\draw[very thick, side by side=magenta:cyan, violet, decoration={markings,
  mark=at position 0.25 with {\arrow{>}},
  mark=at position 0.7 with {\arrow{>}}
 }, postaction={decorate}] (1.4,0) .. controls (2,1) and (0.6,1) .. (0.6,0);
\draw[very thick, side by side=magenta:cyan, violet, decoration={markings,
  mark=at position 0.6 with {\arrow{>}}
 }, postaction={decorate}] (0.6,0) -- (0,0);
\draw[very thick, cyan, decoration={markings,
  mark=at position 0.6 with {\arrow{>}}
 }, postaction={decorate}] (0,0) -- (Xn);
\draw[very thick, cyan, densely dotted, decoration={markings,
  mark=at position 0.6 with {\arrow{>>}}
 }, postaction={decorate}] (Xn) -- (-3.4,0);
\filldraw (1.4,0) circle[radius=1pt] node[below] {$1$};
\filldraw (0.6,0) circle[radius=1pt] node[violet, below] {$x_1$\vphantom{0}};
\filldraw (0,0) circle[radius=1pt] node[magenta, above left] {$\gamma_-(\limit_-)$} node[below left] {$0$};
\filldraw (-3.4,0) circle[radius=1pt] node[cyan, above] {$\gamma_+(\limit_+)$} node[below] {$\cinfty$\vphantom{0}};
\end{tikzpicture}
\\
(iii) $x_1 = \cinfty = \gamma_\pm(\limit_0) = \gamma_\pm(\limit) = \gamma_\pm(\limit_+)$ &
(iv) $x_1 = \gamma_\pm(\limit_0)$, $0 = \gamma_\pm(\limit) = \gamma_\pm(\limit_-)$ \\[1em]
\begin{tikzpicture}
\footnotesize
\coordinate (X) at (4.6,0);
\coordinate (Y) at (0,3);
\coordinate (Xn) at (-1.2,0);
\coordinate (Yn) at (0,-0.4);
\draw[->] (Xn) -- (X) node[above] {$\re z$};
\draw[->] (Yn) -- (Y) node[left] {$\im z$};
\draw[very thick, side by side=magenta:cyan, violet, decoration={markings,
  mark=at position 0.3 with {\arrow{>}},
  mark=at position 0.8 with {\arrow{>}}
 }, postaction={decorate}] (1.4,0) .. controls (2,1.2) and (3.2,2) .. (3.2,0);
\draw[very thick, side by side=magenta:cyan, violet, decoration={markings,
  mark=at position 0.6 with {\arrow{>}},
  mark=at position 0.999 with {\arrow{>>}}
 }, postaction={decorate}] (3.2,0) -- (X);
\draw[very thick, magenta, decoration={markings,
  mark=at position 0.6 with {\arrow{>}}
 }, postaction={decorate}] (Xn) -- (0,0);
\draw[very thick, magenta, densely dotted, decoration={markings,
  mark=at position 0.6 with {\arrow{>>}}
 }, postaction={decorate}] (-2.2,0) -- (Xn);
\filldraw (1.4,0) circle[radius=1pt] node[below] {$1$};
\filldraw (0,0) circle[radius=1pt] node[magenta, above left] {$\gamma_-(\limit_-)$} node[below right] {$0$};
\filldraw (3.2,0) circle[radius=1pt] node[violet, below] {$x_1$\vphantom{0}};
\filldraw (-2.2,0) circle[radius=1pt] node[cyan, above] {$\gamma_+(\limit_+)$} node[below] {$\cinfty$\vphantom{0}};
\end{tikzpicture}
&
\begin{tikzpicture}
\footnotesize
\draw[->] (-2.2,0) -- (3.4,0) node[above] {$\re z$};
\draw[->] (-2.2,2.6) -- (3.4,2.6) node[above] {$\re z$};
\draw[->] (0,-0.4) -- (0,1) node[right] {$\im z$};
\draw[->] (0,2.2) -- (0,3.6) node[right] {$\im z$};
\draw[very thick, side by side=magenta:cyan, violet, decoration={markings,
  mark=at position 0.6 with {\arrow{>}},
  mark=at position 0.999 with {\arrow{>>}}
 }, postaction={decorate}] (1.4,0) -- (3.4,0);
\draw[very thick, side by side=magenta:cyan, violet, decoration={markings,
  mark=at position 0.6 with {\arrow{>}}
 }, postaction={decorate}] (1.4,2.6) -- (0,2.6);
\draw[very thick, cyan, decoration={markings,
  mark=at position 0.6 with {\arrow{>}}
 }, postaction={decorate}] (0,2.6) -- (-2.2,2.6);
\draw[very thick, magenta, decoration={markings,
  mark=at position 0.6 with {\arrow{>}}
 }, postaction={decorate}] (-2.2,0) -- (0,0);
\draw[very thick, cyan, densely dotted, decoration={markings,
  mark=at position 0.6 with {\arrow{>>}}
 }, postaction={decorate}] (-2.2,2.6) -- (-3.4,2.6);
\draw[very thick, magenta, densely dotted, decoration={markings,
  mark=at position 0.6 with {\arrow{>>}}
 }, postaction={decorate}] (-3.4,0) -- (-2.2,0);
\filldraw (1.4,0) circle[radius=1pt] node[below] {$1$} node[violet, above] {$x_1$\vphantom{$\gamma_-$}};
\filldraw (1.4,2.6) circle[radius=1pt] node[below] {$1$} node[violet, above] {$x_1$\vphantom{$\gamma_-$}};
\filldraw (0,0) circle[radius=1pt] node[magenta, above right] {$\gamma_-(\limit_-)$} node[below right] {$0$};
\filldraw (0,2.6) circle[radius=1pt] node[magenta, above left] {$\gamma_-(\limit_-)$} node[below left] {$0$};
\filldraw (-3.4,0) circle[radius=1pt] node[cyan, above] {$\gamma_+(\limit_+)$} node[below] {$\cinfty$\vphantom{0}};
\filldraw (-3.4,2.6) circle[radius=1pt] node[cyan, above] {$\gamma_+(\limit_+)$} node[below] {$\cinfty$\vphantom{0}};
\end{tikzpicture}
\\
(v) $x_1 = \gamma_\pm(\limit_0)$, $\cinfty = \gamma_\pm(\limit) = \gamma_\pm(\limit_+)$ &
(vi) $x_1 = 1 = \gamma_\pm(\limit_0)$
\end{tabular}
\caption{The six scenarios in Step~6 in the proof of Proposition~\ref{prop:spine}.}
\label{fig:spine:cases}
\end{figure}

We now consider six cases; see Figure~\ref{fig:spine:cases}. All maps listed below are one-to-one and onto.
\begin{enumerate}[label=(\roman*)]
\item\label{it:spine:last:1} If $x_1 \in (-\infty, 0)$, then $\gamma_+ : [\limit_0, \limit_+] \to \{\cinfty\} \cup (-\infty, x_1]$, $\gamma_- : [\limit_0, \limit_-] \to [x_1, 0]$. Furthermore, $\gamma_+$ is decreasing on $(\limit_0, \limit_+)$, $\gamma_-$ is increasing on $(\limit_0, \limit_-)$.
\item\label{it:spine:last:2} If $x_1 = 0$, then $\limit_0 = \limit_-$, while $\gamma_+ : [\limit_0, \limit_+] \to \{\cinfty\} \cup (-\infty, 0]$ is decreasing on $(\limit_0, \limit_+)$.
\item\label{it:spine:last:3} If $x_1 = \cinfty$, then $\limit_0 = \limit_+$, while $\gamma_- : [\limit_0, \limit_-] \to \{\cinfty\} \cup (-\infty, 0]$ is increasing on $(\limit_0, \limit_-)$.
\item\label{it:spine:last:4} If $x_1 \in (0, 1)$, then $\gamma_+ : [\limit_0, \limit_+] \to \{\cinfty\} \cup (-\infty, x_1]$, $\gamma_- : [\limit_0, \limit_-] \to [0, x_1]$. Furthermore, $\gamma_+$ and $\gamma_-$ are decreasing on $(\limit_0, \limit_+)$ and $(\limit_0, \limit_-)$, respectively. In this case $(0, x_1) \subseteq D_F$, and~\eqref{eq:gamma:pm} extends to $\lambda \in [0, \limit_-]$.
\item\label{it:spine:last:5} If $x_1 \in (1, \infty)$, then $\gamma_+ : [\limit_0, \limit_+] \to [x_1, \infty) \cup \{\cinfty\}$, $\gamma_- : [\limit_0, \limit_-] \to [x_1, \infty) \cup \{\cinfty\} \cup (-\infty, 0]$. Furthermore, $\gamma_+$ is increasing on $(\limit_0, \limit_+)$, and $\gamma_-$ is piecewise increasing on $(\limit_0, \limit_-)$. In this case $(x_1, \infty) \subseteq D_F$, and~\eqref{eq:gamma:pm} extends to $\lambda \in [0, \limit_+]$.
\item\label{it:spine:last:6} Finally, suppose that $x_1 = 1$. If $\gamma_+$ decreases on $[\limit_0, \limit_+)$ and maps it onto $(-\infty, 1]$, then $(0, 1) \subseteq D_F$, and so $\ph_-$ is equal to zero almost everywhere. If $\gamma_+$ increases on $[\limit_0, \limit_+)$ and maps it onto $[1, \infty)$, then $(1, \infty) \subseteq D_F$, and so $\ph_+$ is equal to zero almost everywhere. Therefore, the corresponding $\amcm$ sequence is one-sided. A similar argument applies to $\gamma_-$. Since we assumed that the corresponding $\amcm$ sequence is nontrivial, $\ph_+$ and $\ph_-$ cannot be simultaneously equal to zero almost everywhere, and hence $x_1$ is either the left or the right endpoint of both intervals $\gamma_+([\limit_0, \limit_+])$ and $\gamma_-([\limit_0, \limit_-])$. This leads us to the conclusion that one of the scenarios~\ref{it:spine:last:4}, \ref{it:spine:last:5} holds, but with $x_1 = 1$.
\end{enumerate}
Recall that $\limit_0 = \sup \Lambda$. We set $\limit = \limit_0$ in case~\ref{it:spine:last:1}, $\limit = \limit_0 = \limit_-$ in case~\ref{it:spine:last:2}, $\limit = \limit_0 = \limit_+$ in case~\ref{it:spine:last:3}, $\limit = \limit_-$ in case~\ref{it:spine:last:4} (possibly with $x_1 = 1$ as in case~\ref{it:spine:last:6}), and $\limit = \limit_+$ in case~\ref{it:spine:last:5} (possibly with $x_1 = 1$ as in case~\ref{it:spine:last:6}).

Observe that in each case identity~\eqref{eq:gamma:pm} holds for $\lambda \in [0, \limit]$, and we have $\gamma_+(\limit) = \gamma_-(\limit) \in \{\cinfty\} \cup (-\infty, 0]$. Furthermore, $\gamma_+(\limit) = \gamma_-(\limit) \in (0, \infty)$ only in scenario~\ref{it:spine:last:1}.

\emph{Step 7.}
The conclusion of the previous step allows us to define $\gamma(\lambda) = \gamma_+(\lambda) = \gamma_-(\lambda)$ for $\lambda \in (0, \limit)$. We now verify all assertions of the proposition.

First of all, if $F$ is the symbol of a one-sided $\amcm$ sequence, then, by~\eqref{eq:poisson}, $\im F$ has a constant sign in the region $\{z \in \C : \im z > 0\}$, and therefore $\Gamma = \varnothing$ and $\Lambda = \varnothing$. Conversely, if $\Lambda = \varnothing$, then we have $x_1 = \gamma_+(\limit) = 1$ in Step~6, scenario~\ref{it:spine:last:6} applies, and therefore $F$ is the symbol of a one-sided $\amcm$ sequence.

Suppose that $F$ is the symbol of a two-sided $\amcm$ sequence, that is, $\Lambda \ne \varnothing$.
\begin{itemize}
\item By the definition~\eqref{eq:gamma:p} of $\gamma_+$, we have $\gamma(0^+) = 1$, and $\im \gamma(\lambda) \ge 0$ for $\lambda \in (0, \limit)$.
\item If $\lambda \in \Lambda$, then $\gamma(\lambda) \in \Gamma$ and $F(\gamma(\lambda)) = \lambda$ by~\eqref{eq:gamma:p:f1}. If $\lambda \in (0, \sup \Lambda) \setminus \Cl \Lambda$, then again $\gamma(\lambda) \in D_F$ and $F(\gamma(\lambda)) = \lambda$ by~\eqref{eq:gamma:p:f2}. Finally, if $\lambda \in (\sup \Lambda, \limit)$, then we are in scenario~\ref{it:spine:last:4}, \ref{it:spine:last:5} or~\ref{it:spine:last:6}, and once again $\gamma(\lambda) \in D_F$ and $F(\gamma(\lambda)) = \lambda$.
\item We have $\gamma(\limit^-) = \gamma_+(\limit) \in \{\cinfty\} \cup (-\infty, 0]$. Furthermore, if $\gamma_+(\limit) \in (-\infty, 0)$, then we are in scenario~\ref{it:spine:last:1}, and so $F(\gamma_+(\limit)) = \limit \in (0, \infty)$.
\item If $\gamma(\limit^-) \ne \cinfty$, then $\gamma_+(\limit) = \gamma(\limit^-) \ne \cinfty = \gamma_+(\limit_+)$, so that $\limit < \limit_+$, and hence $\limit \in (0, \infty)$. Similarly, if $\gamma(\limit^-) \ne 0$, then $\gamma_-(\limit) = \gamma(\limit^-) \ne 0 = \gamma_-(\limit_-)$, and so $\limit < \limit_-$ and $\limit \in (0, \infty)$.
\end{itemize}
This completes the proof of item~\ref{it:spine:a} and shows that $\limit \in (0, \infty)$.

Item~\ref{it:spine:b} is a direct consequence of the definitions~\eqref{eq:gamma:p} and~\eqref{eq:gamma:m} of $\gamma_+$ and $\gamma_-$.

Since $\zeta_f$ and $\lambda_f$ are real-analytic on $Z_f$ and $\lambda_f' \ne 0$ on $Z_f$, the function $\gamma_+ = 1 + i \zeta_f \circ \lambda_f^{-1}$ (see~\eqref{eq:gamma:p}) is real-analytic on $\Lambda = \lambda_f(Z_f)$. Furthermore, every connected component of $\Gamma$ is parameterised by $\gamma$ restricted to a maximal interval $(\lambda_1, \lambda_2)$ contained in $\Lambda$. Since $\lambda_1, \lambda_2 \in \partial \Lambda$, by Step~5 we have $\gamma(\lambda_1^+) = \gamma_+(\lambda_1) \in (0, \infty)$, and $\gamma(\lambda_2^-) = \gamma_+(\lambda_2)$ is in $(0, \infty)$ if $\lambda_2 < \limit$. Thus, all connected components of $\Gamma$ begin and end at $(0, \infty)$, except possibly one with $\lambda_2 = \limit$, which ends at $\gamma(\limit^-) = \gamma_+(\limit) \in \{\cinfty\} \cup (-\infty, 0]$. This completes the proof of item~\ref{it:spine:c}.

Item~\ref{it:spine:d} follows immediately from Proposition~\ref{prop:rogers:spine} and the definitions~\eqref{eq:gamma:p} and~\eqref{eq:gamma:m} of $\gamma_+$ and $\gamma_-$.

Item~\ref{it:spine:e} is a consequence of the analysis conducted in Step~6: we have $F(\gamma_+(\lambda)) > \limit$ for $\lambda \in (\limit, \limit_+)$ and $F(\gamma_-(\lambda)) > \limit$ for $\lambda \in (\limit, \limit_-)$, and in each of the six scenarios the intervals $\gamma_+([\limit, \limit_+])$ and $\gamma_-([\limit, \limit_-])$ cover $(-\infty, 0)$, except perhaps a single point $x_1 = \gamma_+(\limit) = \gamma_-(\limit) = \gamma(\limit^-)$. By continuity, $F(x) \ge \limit$ for every $x \in (-\infty, 0)$. Since $\lv F(z) \rv \le 2 \mass$ for $z \in \torus$, in particular we have $\limit \le F(-1) \le 2 \mass$.

Finally, by Proposition~\ref{prop:rogers:holder}, $\zeta_f$ and $\lambda_f^{-1}$ are locally Hölder continuous on their domains with exponents $\tfrac{1}{30}$ and $\tfrac{1}{3}$, respectively. Since the composition of Hölder continuous functions is Hölder continuous, it follows that $\gamma(\lambda) = \gamma_+(\lambda) = 1 + i \zeta_f(\lambda_f^{-1}(\lambda))$ is Hölder continuous with exponent $\tfrac{1}{30} \cdot \tfrac{1}{3}$.
\end{proof}

It is clear that the parameterisation $\check\gamma$ of the spine $\check\Gamma$ of the symbol $\check F(z) = F(z^{-1})$ of the reversed sequence $(\check a_k) = (a_{-k})$ is given by $\check\gamma(\lambda) = (\overline{\gamma}(\lambda))^{-1}$.

Following the analogous definitions for Rogers functions in~\cite{k19,k25}, if $\Gamma$ is the spine of the symbol $F$ of an $\amcm$ sequence, then we define the \emph{symmetrised spine} $\Gamma^\star$ to be the union of $\Gamma$, all endpoints of $\Gamma$, and the image of $\Gamma$ under complex conjugation.

More precisely, we extend the definition of $\gamma$ to $(-\limit, \limit)$ so that $\gamma(0) = \gamma(0^+)$ and $\gamma(-\lambda) = \overline{\gamma}(\lambda)$ for $\lambda \in (0, \limit)$. Then
\formula[eq:spine:star]{
 \Gamma^\star & = \bigcup_{(\lambda_1, \lambda_2)} \{\gamma(\lambda) : \lambda \in (-\lambda_2, -\lambda_1) \cup [\lambda_1, \lambda_2]\} ,
}
where the union is taken over all maximal intervals $(\lambda_1, \lambda_2) \subseteq \Lambda$.

The symmetrised spine $\Gamma^\star$ (or, more precisely, its closure) divides the complex plane $\C$ into two parts: $D_F^+$ lying to the left of $\Gamma^\star$, and $D_F^-$ lying to the right of $\Gamma^\star$; see Figure~\ref{fig:spine}(b). We have $\sign \im F(x) = \sign \im x$ when $x \in D_F^+ \cap D_F$ and $\sign \im F(x) = -\im x$ when $x \in D_F^-$ (see Section~4.2 in~\cite{k19} for a similar discussion in the context of Rogers functions). More formally, $D_F^+$ and $D_F^-$ are defined using the function $\gamma_+$ introduced in the proof of Proposition~\ref{prop:spine}: we have
\formula[]{
\label{eq:dom:m}
 D_F^- & = \Int \{z \in \C : \lv \arg(z - 1) \rv \le \arg(\gamma_+(\lambda) - 1)\} , \\
\label{eq:dom:p}
 D_F^+ & = \Int \{z \in \C : \lv \arg(1 - z) \rv \le \lv \arg(1 - \gamma_+(\lambda)) \rv\} ;
}
here $\Int$ denotes the interior of a set, the complex argument function $\arg$ takes values in $(-\pi, \pi]$, and $\arg 0 = 0$.


\subsection{Wiener--Hopf factorisation of Rogers functions}

The Wiener--Hopf factorisation of a function $F$ on $\torus$ is the representation of $F$ as the product of two factors, one holomorphic in the unit disk $\disk$, and another one holomorphic in its complement $\C \setminus \overline \disk$. This can be written in the following form:
\formula{
 F(z) & = F^+(z) F^-(z^{-1}) ,
}
where $F^+$ and $F^-$ are holomorphic in $\disk$ and continuous up to the boundary.

If $f$ is a nonzero Rogers function with exponential representation~\eqref{eq:rogers:exp}, then the Wiener--Hopf factors of $f$ are given by
\formula[]{
\label{eq:rogers:wh:exp:p}
 f^+(\xi) & = b^+ \, \exp\biggl(\frac{1}{\pi} \int_0^\infty \biggl(\frac{\xi}{\xi + r} - \frac{1}{1 + r}\biggr) \, \frac{\ph(r)}{r} \, dr\biggr) , \\
\label{eq:rogers:wh:exp:m}
 f^-(\xi) & = b^- \, \exp\biggl(\frac{1}{\pi} \int_0^\infty \biggl(\frac{\xi}{\xi + r} - \frac{1}{1 + r}\biggr) \, \frac{\ph(-r)}{r} \, dr\biggr) ,
}
where $b^+, b^- > 0$ satisfy $b^+ b^- = b$; see Equation~(5.2) in~\cite{k19}. By Theorem~5.1 in~\cite{k19}, $f^+$ and $f^-$ are complete Bernstein functions, and
\formula[eq:rogers:wh]{
 f(\xi) & = f^+(-i \xi) f^-(i \xi)
}
for all $\xi \in D_f$. Furthermore, the above equation determines the complete Bernstein functions $f^+$ and $f^-$ uniquely up to multiplication by a constant.

By Proposition~3.12(b) in~\cite{k19}, if $f$ is a nonzero Rogers function, then $\xi^2 / f(\xi)$ is also a Rogers function. Suppose that $f$ has the Wiener--Hopf factorisation~\eqref{eq:rogers:wh}. Then the Wiener--Hopf factors of $\xi^2 / f(\xi)$ are given by $\xi / f^+(\xi)$ and $\xi / f^-(\xi)$, that is, the corresponding Wiener--Hopf factorisation is
\formula{
 \frac{\xi^2}{f(\xi)} & = \frac{-i \xi}{f^+(-i \xi)} \, \frac{i \xi}{f^-(i \xi)} \, .
}
This follows directly from the fact that $\xi / f^+(\xi)$ and $\xi / f^-(\xi)$ are complete Bernstein functions, and the uniqueness of the Wiener--Hopf factorisation.

We have the following fundamental result.

\begin{proposition}[Proposition~3.11 in~\cite{k25}]
\label{prop:rogers:wh:int}
If $f$ is a nonconstant Rogers function such that $f(0^+) = 0$, $\sigma > 0$ and $f_\sigma(\xi) = \sigma + f(\xi)$, then for $\xi, \eta > 0$ we have
\formula{
 f_\sigma^+(\xi) f_\sigma^-(\eta) & = \sigma \exp\biggl(\frac{1}{\pi} \int_0^\infty \bigl(\arg(\zeta_f(r) + i \eta) - \arg(\zeta_f(r) - i \xi)\bigr) \, \frac{d\lambda_f(r)}{\sigma + \lambda_f(r)}\biggr) ,
}
and the above expression defines a complete Bernstein function of $\sigma$. (The integral on the right-hand side is a Lebesgue--Stieltjes one.)
\end{proposition}


\subsection{Wiener--Hopf factorisation of symbols of \texorpdfstring{$\amcm$}{AM/CM} sequences}

A formal definition of the Wiener--Hopf factors of symbols of $\amcm$ sequences involves the exponential representation from Proposition~\ref{prop:exp}.

\begin{definition}
\label{def:wh}
Suppose that $F$ is the symbol of a nontrivial summable $\amcm$ sequence, $\sigma \ge 0$, and that $F_\sigma = \sigma + F$ has the exponential representation~\eqref{eq:exp}:
\formula{
 F_\sigma(z) & = c \exp \biggl(\frac{1}{\pi} \int_0^1 \frac{1}{s - z^{-1}} \, \ph_+(s) ds + \frac{1}{\pi} \int_0^1 \frac{1}{s - z} \, \ph_-(s) ds\biggr) .
}
For $x, y \in \C \setminus [1, \infty)$ we define
\formula[]{
\label{eq:wh:exp:p}
 F_\sigma^+(x) & = c^+ \exp \biggl(\frac{1}{\pi} \int_0^1 \frac{1}{s - x^{-1}} \, \ph_+(s) ds\biggr) , \\
\label{eq:wh:exp:m}
 F_\sigma^-(y) & = c^- \exp \biggl(\frac{1}{\pi} \int_0^1 \frac{1}{s - y^{-1}} \, \ph_-(s) ds\biggr) ,
}
where $c^+, c^- > 0$ satisfy $c^+ c^- = c$.
\end{definition}

There is no canonical way to choose constants $c^+$ and $c^-$ in the above definition, so we leave them arbitrary (as long as $c^+ c^- = c$). This causes no ambiguity, because in the remaining part of the article all expressions involving $F_\sigma^+$ and $F_\sigma^-$ do not depend on the choice of $c^+$ and $c^-$.

The definition of $F_\sigma^+(x)$ extends to $\{x \in \C : x^{-1} \notin \esssupp \ph_+\}$, while $F_\sigma^-(y)$ extends to $\{y \in \C : y^{-1} \notin \esssupp \ph_-\}$. Clearly, for $\sigma \ge 0$ and $z \in D_F \setminus \{0\}$,
\formula[eq:wh]{
 F_\sigma(z) & = F_\sigma^+(z) F_\sigma^-(z^{-1}) .
}
In other words, $F_\sigma^+$ and $F_\sigma^-$ are the Wiener--Hopf factors of $F_\sigma$. Their role in the proof of Theorem~\ref{thm:main} will be explained in Section~\ref{sec:wh}.

\begin{remark}
\label{rem:wh:symbol}
By Proposition~\ref{prop:exp:inv}, if we choose $\sigma^+, \sigma^- \ge 0$ appropriately, then $F_\sigma^+ - \sigma^+$ and $F_\sigma^- - \sigma^-$ are symbols of one-sided summable $\amcm$ sequences. More precisely, by~\eqref{eq:exp:sigma}, we have
\formula{
 \sigma^+ & = c^+ \exp \biggl(-\frac{1}{\pi} \int_0^1 \frac{1}{1 - s} \, \ph_+(s) ds\biggr) , \\
 \sigma^- & = c^- \exp \biggl(-\frac{1}{\pi} \int_0^1 \frac{1}{1 - s} \, \ph_-(s) ds\biggr) .
}
Additionally, by Lemma~\ref{lem:one}, $\sigma^+$ and $\sigma^-$ are limits of $F_\sigma^+(z)$ and $F_\sigma^-(z)$ as $z \to 1$ nontangentially in $z \in \C \setminus [1, \infty)$. In particular, if $\sigma > 0$, then also $\sigma^+, \sigma^- > 0$.
\end{remark}

It is clear that the role of the Wiener--Hopf factors $F_\sigma^+$ and $F_\sigma^-$ is symmetric: the reversed sequence $(\check a_k) = (a_{-k})$ and the corresponding dual symbol $\check F(z) = F(z^{-1})$ have $\check F_\sigma^+ = F_\sigma^-$ and $\check F_\sigma^- = F_\sigma^+$ as Wiener--Hopf factors.

Below we study the relation between the functions $F_\sigma^+$ and $F_\sigma^-$ and the Wiener--Hopf factors of the Rogers function $f(\xi) = F(1 + i \xi)$. We assume that $F_\sigma = \sigma + F$ has exponential representation~\eqref{eq:exp}, while $f_\sigma = \sigma + f$ has exponential representation~\eqref{eq:rogers:exp}.

Using the terminology of Remark~\ref{rem:compatible}, the triple $\ph, \ph_+, \ph_-$ is compatible. It follows that also the triples $\ph \ind_{(0, \infty)}, \ph_+, 0$ and $\ph \ind_{(-1, 0)}, 0, \ph_-$ are compatible. Observe that, by~\eqref{eq:rogers:wh:exp:p} and~\eqref{eq:rogers:wh:exp:m}, $\ph \ind_{(0, \infty)}$ and $\ph \ind_{(-1, 0)}$ appear in the exponential representation~\eqref{eq:rogers:exp} of the Rogers functions $f^+(-i \xi)$ and $f^-(i \xi)$, respectively. Similarly, by~\eqref{eq:wh:exp:p} and~\eqref{eq:wh:exp:m}, the pairs $\ph_+, 0$ and $0, \ph_-$ appear in the exponential representation~\eqref{eq:exp} of symbols $F_\sigma^+(z)$ and $F_\sigma^-(z^{-1})$ of one-sided $\amcm$ sequences. By Remark~\ref{rem:compatible}, there is a constant $c_\sigma > 0$ such that if $x = 1 + i \xi$, then
\formula{
 F_\sigma^+(x) & = c_\sigma f_\sigma^+(-i \xi) = c_\sigma f_\sigma^+(1 - x) ,
}
and if $x^{-1} = 1 + i \xi$, then
\formula{
 F_\sigma^-(x^{-1}) & = c_\sigma^{-1} f_\sigma^-(i \xi) = c_\sigma^{-1} f_\sigma^-(x - 1) .
}
In particular, if $x, y \in (0, 1)$, then
\formula[eq:wh:rogers]{
 F_\sigma^+(x) F_\sigma^-(y) & = f_\sigma^+(1 - x) f_\sigma^-(y^{-1} - 1) .
}
The above identity enables us to rephrase Proposition~\ref{prop:rogers:wh:int} in terms of symbols of $\amcm$ sequences.

\begin{proposition}
\label{prop:wh:int}
If $F$ is a symbol of a nontrivial summable $\amcm$ sequence, $\sigma > 0$ and $F_\sigma = \sigma + F$, then for $x, y \in (0, 1)$ we have
\formula[eq:wh:int]{
 F_\sigma^+(x) F_\sigma^-(y) & = \sigma \exp\biggl(\frac{1}{\pi} \int_0^{\limit} \bigl(\arg(\gamma(\lambda) - y^{-1}) - \arg(\gamma(\lambda) - x)\bigr) \, \frac{d\lambda}{\sigma + \lambda}\biggr) .
}
and the above expression defines a complete Bernstein function of $\sigma$.
\end{proposition}

\begin{proof}
With the notation introduced above, formula~\eqref{eq:wh:rogers} holds. By Proposition~\ref{prop:rogers:wh:int} (with $\xi = 1 - x$ and $\eta = y^{-1} - 1$),
\formula{
 F_\sigma^+(x) F_\sigma^-(y) & = f_\sigma^+(1 - x) f_\sigma^-(y^{-1} - 1) \\
 & = \sigma \exp\biggl(\frac{1}{\pi} \int_0^\infty \bigl(\arg(\zeta_f(r) + i y^{-1} - i) - \arg(\zeta_f(r) + i x - i)\bigr) \, \frac{d\lambda_f(r)}{\sigma + \lambda_f(r)}\biggr) ,
}
and the right-hand side is a complete Bernstein function of $\sigma$. It remains to prove that the right-hand side is equal to the right-hand side of~\eqref{eq:wh:int}.

We use the notation of the proof of Proposition~\ref{prop:spine}: $\limit_+ = f(\infty^-)$ and $\gamma_+(\lambda) = 1 + i \zeta_f(\lambda_f^{-1}(\lambda))$ for $\lambda \in (0, \limit_+)$ (see~\eqref{eq:gamma:p}). Substituting $\lambda = \lambda_f(r)$, $r = \lambda_f^{-1}(\lambda)$, we obtain
\formula{
 F_\sigma^+(x) F_\sigma^-(y) & = \sigma \exp\biggl(\frac{1}{\pi} \int_0^{\limit_+} \bigl(\arg(-i \gamma_+(\lambda) + i y^{-1}) - \arg(-i \gamma_+(\lambda) + i x)\bigr) \, \frac{d\lambda}{\sigma + \lambda}\biggr) .
}
Since $\arg$ takes values in $(-\pi, \pi]$ and $\im \gamma_+(\lambda) \ge 0$, we have
\formula{
 \arg(-i \gamma_+(\lambda) + i y^{-1}) & = \arg(\gamma_+(\lambda) - y^{-1}) - \tfrac{\pi}{2} , \\
 \arg(-i \gamma_+(\lambda) + i x) & = \arg(\gamma_+(\lambda) - x) - \tfrac{\pi}{2} .
}
Therefore,
\formula[eq:wh:int:aux]{
 F_\sigma^+(x) F_\sigma^-(y) & = \sigma \exp\biggl(\frac{1}{\pi} \int_0^{\limit_+} \bigl(\arg(\gamma_+(\lambda) - y^{-1}) - \arg(\gamma_+(\lambda) - x)\bigr) \, \frac{d\lambda}{\sigma + \lambda}\biggr) .
}
Recall that $\gamma(\lambda) = \gamma_+(\lambda)$ for $\lambda \in (0, \limit)$, and $\gamma_+(\lambda) \in (-\infty, 0)$ for $\lambda \in (\limit, \limit_+)$. It follows that if $\lambda \in (\limit, \limit_+)$, then
\formula{
 \arg(\gamma_+(\lambda) - y^{-1}) - \arg(\gamma_+(\lambda) - x) & = \pi - \pi = 0 .
}
Thus, the integral in~\eqref{eq:wh:int:aux} is effectively over $[0, \limit]$, and the proof is complete.
\end{proof}


\subsection{Difference quotients of Rogers functions}

Difference quotients of Rogers functions played a key role in the description of generalised eigenfunctions of Lévy processes with completely monotone jumps. We recall two key results from~\cite{k25}.

\begin{proposition}[Proposition 3.16 in~\cite{k25}]
\label{prop:rogers:h}
Let $f$ be a nonconstant Rogers function and $\zeta \in \Gamma_f$. Then
\formula[eq:rogers:h]{
 h(\xi) & = \frac{(\xi - \zeta) (\xi + \overline{\zeta})}{f(\xi) - f(\zeta)} \, ,
}
defined for $\xi \in \hp \setminus \{\zeta\}$ and extended continuously at $\xi = \zeta$, is a Rogers function.
\end{proposition}

Since $\xi^2 / h(\xi)$ is a Rogers function whenever $h$ is, we find that
\formula[eq:rogers:h:dual]{
 \tilde h(\xi) & = \frac{\xi^2}{(\xi - \zeta) (\xi + \overline{\zeta})} \, (f(\xi) - f(\zeta))
}
is a Rogers function.

The function $h$ with $\zeta = \zeta_f(r)$ is denoted in~\cite{k25} by $f(r; \xi)$. Thus,
\formula[eq:rogers:fr]{
 f(r; \xi) & = \frac{(\xi - \zeta_r(r)) (\xi + \overline{\zeta_f(r)})}{f(\xi) - \lambda_f(r)} \, .
}
The corresponding Wiener--Hopf factors are denoted by $f^+(r; \xi)$ and $f^-(r; \xi)$. Observe that the Wiener--Hopf factors of $\tilde h(\xi) = \xi^2 / f(r; \xi)$ are given by $\xi / f^+(r; \xi)$ and $\xi / f^-(r; \xi)$.


\subsection{Difference quotients of symbols of \texorpdfstring{$\amcm$}{AM/CM} sequences}

Suppose that $F$ is a symbol of a nonzero $\amcm$ sequence, and let $f(\xi) = F(1 + i \xi)$ be the corresponding Rogers function. We have already seen many analogies between $\amcm$ sequences and their symbols on the one hand, and $\amdcm$ functions and the associated Rogers functions on the other hand. When it comes to difference quotients, the link is less straightforward: a direct analogue of the function $h$ given by~\eqref{eq:rogers:h}, that is,
\formula{
 & \frac{(z - \gamma(\lambda)) (z - \overline{\gamma}(\lambda))}{F(z) - \lambda} \, ,
}
does not define the symbol of an $\amcm$ sequence. It turns out, however, that the Rogers function $\xi^2 / h(\xi)$ corresponds to the symbol of an $\amcm$ sequence.

\begin{proposition}
\label{prop:h}
If $F$ is the symbol of a nonzero summable $\amcm$ sequence, and $\lambda \in \Lambda$, then
\formula{
 H(z) & = \frac{(z - 1)^2}{(\gamma(\lambda) - z) (\overline{\gamma}(\lambda) - z)} \, (F(z) - \lambda) ,
}
defined for $z \in D_F \setminus \{\gamma(\lambda), \overline{\gamma}(\lambda)\}$ and extended continuously at $z = \gamma(\lambda)$ and $z = \overline{\gamma}(\lambda)$, is the symbol of a summable $\amcm$ sequence.
\end{proposition}

\begin{proof}
Denote $f(\xi) = F(1 + i \xi)$ and $\zeta = -i (\gamma(\lambda) - 1)$. Then $f$ is a Rogers function, and $\zeta \in \Gamma_f$. Let $\tilde h$ be the Rogers function given by~\eqref{eq:rogers:h:dual}. We have $\gamma(\lambda) = 1 + i \zeta$, $\lambda = f(\zeta)$, and
\formula{
 H(1 + i \xi) & = \frac{(-\xi^2)}{(i \xi - i \zeta) (i \xi + i \overline{\zeta})} \, (f(\xi) - f(\zeta)) = \tilde h(\xi) .
}
In order to prove that $H$ is the symbol of a summable $\amcm$ sequence, we will apply Proposition~\ref{prop:characterisation}. Thus, we need to verify three properties of $\tilde h$.

Clearly, $\tilde h(0^+) = 0$ by definition~\eqref{eq:rogers:h:dual}. By Proposition~\ref{prop:characterisation} applied to $F$, the Rogers function $f$ is holomorphic in $\C \setminus (-i \infty, i]$, and so $\tilde h$ is holomorphic in $\C \setminus (-i \infty, i]$. Finally, for $r > 1$ we have
\formula{
 \tilde h(i r) & = \frac{-r^2}{(i r - \zeta) (i r + \overline{\zeta})} \, (f(i r) - f(\zeta)) = \frac{r^2}{\lv r + i \zeta \rv^2} \, (F(1 - r) - \lambda) .
}
By Proposition~\ref{prop:spine}\ref{it:spine:e}, $F(1 - r) \ge \limit \ge \sup \Lambda > \lambda$, and so $\tilde h(i r) > 0$. We conclude that, by Proposition~\ref{prop:characterisation}, $H$ is indeed a Rogers function.
\end{proof}

Recall that if $\zeta = \zeta_f(r)$, then the function $h$ introduced in Proposition~\ref{prop:rogers:h} is denoted by $f(r; \xi)$; see~\eqref{eq:rogers:fr}. In a similar way, we denote by $F(\lambda; z)$ the function $H$ defined in Proposition~\ref{prop:h}, multiplied by $\lv \gamma(\lambda) \rv$.

\begin{definition}
\label{def:fr}
Suppose that $F$ is the symbol of a summable $\amcm$ sequence. For $\lambda \in \Lambda$ and $z \in D_F \setminus \{\gamma(\lambda), \overline{\gamma}(\lambda)\}$ we denote
\formula[eq:fr]{
 F(\lambda; z) & = \frac{\lv \gamma(\lambda) \rv (z - 1)^2}{(\gamma(\lambda) - z) (\overline{\gamma}(\lambda) - z)} \, (F(z) - \lambda) .
}
This definition is extended continuously at $z = \gamma(\lambda)$ and $z = \overline{\gamma}(\lambda)$.
\end{definition}

Clearly,
\formula{
 F(\lambda; 1 + i \xi) & = \lv \gamma(\lambda) \rv \tilde h(\xi) = \frac{\lv \gamma(\lambda) \rv \xi^2}{h(\xi)} = \frac{\lv \gamma(\lambda) \rv \xi^2}{f(r; \xi)} \, ,
}
where $r \in Z_f$ is such that $\gamma(\lambda) = 1 + i \zeta_f(r)$.

The factor $\lv \gamma(\lambda) \rv$ in~\eqref{eq:fr} makes the definition of $F(\lambda; z)$ symmetric: the corresponding function for the reversed sequence $(\check a_k) = (a_{-k})$ and its symbol $\check F(z) = F(z^{-1})$ is given by
\formula{
 \check F(\lambda; z) & = \frac{\lv \check\gamma(\lambda) \rv (z - 1)^2}{(z - \check\gamma(\lambda)) (z - \overline{\check\gamma(\lambda)})} \, (\check F(z) - \lambda) \\
 & = \frac{\lv \gamma(\lambda) \rv^{-1} (z - 1)^2}{(z - (\overline{\gamma}(\lambda))^{-1}) (z - (\gamma(\lambda))^{-1})} \, (F(z^{-1}) - \lambda) \\
 & = \frac{\lv \gamma(\lambda) \rv^{-1} z^2 (1 - z^{-1})^2}{\lv \gamma(\lambda) \rv^{-2} z^2 (\overline{\gamma}(\lambda) - z^{-1}) (\gamma(\lambda) - z^{-1})} \, (F(z^{-1}) - \lambda) \\
 & = F(\lambda; z^{-1}) .
}


\subsection{Wiener--Hopf factorisation of difference quotients}

The Wiener--Hopf factors of $F(\lambda; z)$ play a crucial role in our development.

\begin{definition}
\label{def:wh:fr}
Suppose that $F$ is the symbol of a nontrivial summable $\amcm$ sequence and $F(\lambda; x)$ is the corresponding difference quotient. We denote by $F^+(\lambda; x)$ and $F^-(\lambda; y)$ the Wiener--Hopf factors of $F(\lambda; z)$.
\end{definition}

In particular, by~\eqref{eq:wh}, we have
\formula[eq:wh:fr]{
 F(\lambda; z) & = F^+(\lambda; z) F^-(\lambda; z^{-1})
}
for $\lambda \in \Lambda$ and $z \in D_F \setminus \{0\}$.

By Proposition~\ref{prop:exp}, $F(\lambda; z)$ has the exponential representation
\formula{
 F(\lambda; z) & = c_\lambda \exp\biggl(\frac{1}{\pi} \int_0^1 \frac{1}{s - z^{-1}} \, \ph^+_\lambda(s) ds + \frac{1}{\pi} \int_0^1 \frac{1}{s - z} \, \ph^-_\lambda(s) ds \biggr) ,
}
where $\ph^+_\lambda(s), \ph^-_\lambda(s) \in [0, \pi]$ are given by
\formula[]{
\label{eq:fr:wh:ph:p}
 \ph^+_\lambda(s) & = -\lim_{t \to 0^+} \arg(F(\lambda; s^{-1} + i t)) , \\
\label{eq:fr:wh:ph:m}
 \ph^-_\lambda(s) & = \lim_{t \to 0^+} \arg(F(\lambda; s + i t))
}
for almost every $s \in (0, 1)$. Thus,
\formula[]{
\label{eq:fr:wh:exp:ph:p}
 F^+(\lambda; z) & = c^+_\lambda \exp\biggl(\frac{1}{\pi} \int_0^1 \frac{1}{s - z^{-1}} \, \ph^+_\lambda(s) ds \biggr) , \\
\label{eq:fr:wh:exp:ph:m}
 F^-(\lambda; z) & = c^-_\lambda \exp\biggl(\frac{1}{\pi} \int_0^1 \frac{1}{s - z^{-1}} \, \ph^-_\lambda(s) ds \biggr) .
}
Below we provide a more explicit description of $\ph^+_\lambda$ and $\ph^-_\lambda$.

\begin{proposition}
\label{prop:fr:wh:exp}
Suppose that $F$ is the symbol of a nontrivial summable $\amcm$ sequence. Then for almost every $s \in (0, \infty)$ the boundary limit
\formula[eq:fr:wh:exp:limit]{
 F(s + 0 i) & = \lim_{t \to 0^+} F(s + i t)
}
exists (also nontangentially), and we have $\im F(s + 0 i) \ge 0$ for almost every $s \in (0, 1)$, and $\im F(s + 0 i) \le 0$ for almost every $s \in (1, \infty)$. Furthermore,
\formula[]{
\label{eq:fr:wh:exp:p}
 F^+(\lambda; z) & = c^+_\lambda \exp\biggl(\frac{1}{\pi} \int_0^1 \frac{1}{s - z^{-1}} \, \lv \arg(F(s^{-1} + 0 i) - \lambda) \rv ds \biggr) , \\
\label{eq:fr:wh:exp:m}
 F^-(\lambda; z) & = c^-_\lambda \exp\biggl(\frac{1}{\pi} \int_0^1 \frac{1}{s - z^{-1}} \, \lv \arg(F(s + 0 i) - \lambda) \rv ds \biggr) ,
}
where $\arg$ takes values in $(-\pi, \pi]$.
\end{proposition}

\begin{proof}
By~\eqref{eq:symbol}, $F$ is a Cauchy--Stieltjes integral in the upper complex half-plane, and so the almost everywhere existence of the limit~\eqref{eq:fr:wh:exp:limit} is a consequence of Fatou's theorem (see Theorems~2.2 and~3.5 in~\cite{d70}). Furthermore, since $f(\xi) = F(1 + i \xi)$ is a Rogers function, we have $\re (F(1 + i \xi) / \xi) \ge 0$ when $\re \xi \ge 0$. By setting $\xi = t - (s - 1) i$ and passing to the limit as $t \to 0^+$, we find that $\re (i F(s + 0 i) / (s - 1)) \ge 0$ whenever the limit exists. It follows that $\im F(s + 0 i) \ge 0$ for almost all $s \in (0, 1)$, and $\im F(s + 0 i) \le 0$ for almost all $s \in (1, \infty)$.

Let us write $\alpha \equiv \beta \pmod{2 \pi}$ if $\alpha - \beta$ is a multiple of $2 \pi$. By~\eqref{eq:fr:wh:ph:p} and Definition~\ref{def:fr}, for almost all $s \in (0, 1)$ we have
\formula{
 \ph^+_\lambda(s) & = -\lim_{t \to 0^+} \arg\biggl(\frac{\lv \gamma(\lambda) \rv (s^{-1} + i t - 1)^2}{(\gamma(\lambda) - s^{-1} + i t) (\overline{\gamma}(\lambda) - s^{-1} + i t)} \, (F(s^{-1} + i t) - \lambda)\biggr) \\
 & \equiv -\arg(F(s^{-1} + 0 i) - \lambda) = \lv \arg(F(s^{-1} + 0 i) - \lambda) \rv \pmod{2 \pi}.
}
However, both $\ph^+_\lambda$ and $\lv \arg(F(s^{-1} + 0 i) - \lambda) \rv$ belong to $[0, \pi]$, and so we necessarily have $\ph^+_\lambda(s) = \lv \arg(F(s^{-1} + 0 i) - \lambda)$. Formula~\eqref{eq:fr:wh:exp:p} follows thus from~\eqref{eq:fr:wh:exp:ph:p}. Similarly, for almost all $s \in (0, 1)$, by~\eqref{eq:fr:wh:ph:m},
\formula{
 \ph^-_\lambda(s) & = \lim_{t \to 0^+} \arg\biggl(\frac{\lv \gamma(\lambda) \rv (s + i t - 1)^2}{(\gamma(\lambda) - s + i t) (\overline{\gamma}(\lambda) - s + i t)} \, (F(s + i t) - \lambda)\biggr) \\
 & \equiv \arg(F(s + 0 i) - \lambda) \pmod{2 \pi} ,
}
and since both sides are in $[0, \pi]$, they must be equal. Hence, \eqref{eq:fr:wh:exp:ph:m} implies~\eqref{eq:fr:wh:exp:m}.
\end{proof}

If $\lambda \in \Lambda$ and $r \in Z_f$ are such that $\gamma(\lambda) = 1 + i \zeta_f(r)$, then
\formula{
 F(\lambda; 1 + i \xi) & = \frac{\lv \gamma(\lambda) \rv \xi^2}{f(r; \xi)} = \lv \gamma(\lambda) \rv \, \frac{-i \xi}{f^+(r; -i \xi)} \, \frac{i \xi}{f^-(r; i \xi)} \, .
}
By the uniqueness of the Wiener--Hopf factors, for some $c_\lambda > 0$ we have
\formula{
 F^+(\lambda; 1 + i \xi) & = c_\lambda \lv \gamma(\lambda) \rv \, \frac{-i \xi}{f^+(r; -i \xi)} \, , \\
 F^-(\lambda; (1 + i \xi)^{-1}) & = c_\lambda^{-1} \, \frac{i \xi}{f^-(r; i \xi)} \, .
}
In other words,
\formula[]{
\label{eq:fr:wh:p}
 F^+(\lambda; x) & = c_\lambda \lv \gamma(\lambda) \rv \, \frac{1 - x}{f^+(r; 1 - x)} \, , \\
\label{eq:fr:wh:m}
 F^-(\lambda; y) & = c_\lambda^{-1} \, \frac{y^{-1} - 1}{f^-(r; y^{-1} - 1)} \, .
}
Since $f^+(r; \xi)$ and $\xi / f^+(r; \xi)$ are complete Bernstein functions of $\xi$, also
\formula[eq:fr:cbf:p]{
 & F^+(\lambda; 1 - \xi) \quad \text{and} \quad \frac{\xi}{F^+(\lambda; 1 - \xi)} \quad \text{are complete Bernstein functions of $\xi$.}
}
In a similar way, using the Rogers function $g$ instead of $f$ (or simply applying the above property to the reversed sequence $(\check a_k) = (a_{-k})$ and the dual symbol $\check F(z) = F(z^{-1})$), we find that
\formula[eq:fr:cbf:m]{
 & F^-(\lambda; 1 - \xi) \quad \text{and} \quad \frac{\xi}{F^-(\lambda; 1 - \xi)} \quad \text{are complete Bernstein functions of $\xi$.}
}
In particular, $(1 - z) / F^+(\lambda; z)$ and $(1 - z) / F^-(\lambda; z)$ have finite limits as $z \to 1$ nontangentially in $z \in \C \setminus [1, \infty)$. However,
\formula{
 \frac{1 - z}{F^+(\lambda; z)} \, \frac{1 - z^{-1}}{F^-(\lambda; z^{-1})} & = -\frac{(z - 1)^2}{z F(\lambda; z)}
}
converges to a nonzero limit $\lambda^{-1} \lv \gamma(\lambda) \rv^{-1} \lv 1 - \gamma(\lambda) \rv^2$ as $z \to 1$ nontangentially in $z \in \C \setminus \R$; see Definition~\ref{def:fr}. Therefore, both limits
\formula[eq:fr:one]{
 \lim_{x \to 1^-} \frac{1 - x}{F^+(\lambda; x)} \, , &&
 \lim_{y \to 1^-} \frac{1 - y}{F^-(\lambda; y)}
}
exist (even nontangentially in $\C \setminus [1, \infty)$) and are finite positive numbers.

We will need the following auxiliary result, which is a simple reformulation of Proposition~3.20 from~\cite{k25}.

\begin{proposition}
\label{prop:integral}
Let $F$ be the symbol of a nontrivial summable $\amcm$ sequence and $\lambda \in \Lambda$. Let $G(z)$ denote the branch of the square root of $(\gamma(\lambda) - z) (\overline{\gamma}(\lambda) - z)$ which is a holomorphic function on $D_F^+ \cup D_F^-$, equal to $\lv \gamma(\lambda) - z \rv$ on $D_F^+ \cap (0, 1)$ and on $D_F^- \cap (1, \infty)$. If $x \in D_F^+$ and $y^{-1} \in D_F^-$, then
\formula*[eq:integral]{
 & \exp\biggl(\frac{1}{2 \pi i} \int_{\Gamma^\star} \biggl(\frac{1}{z - x} - \frac{1}{z - y^{-1}}\biggr) \log \lv F(z) - \lambda \rv dz \biggr) \\
 & \qquad = \frac{G(x) G(y^{-1}) F^+(\lambda; x) F^-(\lambda; y)}{\lv \gamma(\lambda) \rv (1 - x) (y^{-1} - 1)} \, .
}
\end{proposition}

\begin{proof}
We apply Propositon~3.20 from~\cite{k25} to the (nondegenerate and nonconstant) Rogers function $f(\xi) = F(1 + i \xi)$, with $\xi_1 = -i (x - 1)$, $\xi_2 = -i (y^{-1} - 1)$ and $r \in Z_f$ such that $\zeta_f(r) = \gamma(\lambda)$ and $\lambda_f(r) = \lambda$. If $\Gamma_f^\star = 1 + i \Gamma^\star$ denotes the symmetrised spine of $f$, then we have
\formula{
 \exp\biggl(\frac{1}{2 \pi i} \int_{\Gamma_f^\star} \biggl(\frac{1}{\xi - \xi_1} - \frac{1}{\xi - \xi_2}\biggr) \log \lv f(\xi) - \lambda \rv d\xi \biggr) & = \frac{g(\xi_1) g(\xi_2)}{f^+(r; -i \xi_1) f^-(r; i \xi_2)} \, ,
}
where $g(\xi) = G(1 + i \xi)$ is the branch of the square root of $-(\xi - \zeta_f(r)) (\xi + \overline{\zeta_f(r)})$ which is holomorphic in $D_f^+ \cup D_f^-$ and equal to $\lv \xi - \zeta_f(r) \rv$ on $D_f^+ \cap (0, i \infty)$ and on $D_f^- \cap (-i \infty, 0)$.

A substitution $z = 1 + i \xi$ transforms the left-hand side to
\formula{
 & \exp\biggl(\frac{1}{2 \pi i} \int_{\Gamma^\star} \biggl(\frac{1}{-i (z - x)} - \frac{1}{-i (z - y^{-1})}\biggr) \log \lv F(z) - \lambda \rv (-i) dz \biggr) ,
}
which is equal to the left-hand side of~\eqref{eq:integral}. Using~\eqref{eq:fr:wh:p}, \eqref{eq:fr:wh:m} and $g(\xi) = G(1 + i \xi)$, we find that the right-hand side can be written as
\formula{
 \frac{G(x) G(y^{-1})}{f^+(r; 1 - x) f^-(r; y^{-1} - 1)} & = \frac{G(x) G(y^{-1})}{\lv \gamma(\lambda) \rv} \, \frac{F^+(\lambda; x)}{1 - x} \, \frac{F^-(\lambda; y)}{y^{-1} - 1} \, ,
}
which reduces to the right-hand side of~\eqref{eq:integral}.
\end{proof}

Note that the definition of $F^+(\lambda; x)$ and $F^-(\lambda; y)$ is symmetric, in the following sense: the Wiener--Hopf factors of the reversed sequence $(\check a_k) = (a_{-k})$ are given by $\check F^+(\lambda; x) = F^-(\lambda; x)$ and $\check F^-(\lambda; y) = F^+(\lambda; x)$.

%
%

\section{Toeplitz matrices, random walks, and Wiener--Hopf factorisation}
\label{sec:wh}

With the extensive toolkit developed in the previous section, we are ready to link the Wiener--Hopf factorisation of symbols $F$ of $\amcm$ sequences $(a_k)$ with Toeplitz matrices $T = (a_{k - l} : k, l \ge 0)$. More precisely, we give an expression for the generating function of the entries of $T^n$.

We take two paths. First, we use a purely analytic approach, exploiting Krein's factorisation theorem for inverse Toeplitz matrices. Next, we expand the discussion of the probabilistic context of our work given in the introduction, and we provide an alternative proof using the fluctuation theory of random walks.


\subsection{Toeplitz matrices}

We consider the infinite Toeplitz matrix $T = (a_{k - l} : k, l \ge 0)$ associated to a summable $\amcm$ sequence $(a_k)$. We denote by $F$ the symbol of $(a_k)$, and we set $F_\sigma(z) = \sigma + F(z)$ for $\sigma \ge 0$. Our goal in this section is to establish the link between the powers of $T$ and the Wiener--Hopf factors $F_\sigma^+$ and $F_\sigma^-$ of $F_\sigma$, described in Definition~\ref{def:wh}.

The following formula is likely well-known, but difficult to locate in this exact form in the literature.

\begin{proposition}
\label{prop:wh}
Let $T = (a_{k - l})$ be the infinite Toeplitz matrix corresponding to a summable $\amcm$ sequence $(a_k)$ with $a_0 = 0$. Denote by $T^n_{k, l}$ the entries of $T^n$, the $n$th matrix power of $T$. Then
\formula{
 \sum_{n = 0}^\infty \sum_{k = 0}^\infty \sum_{l = 0}^\infty \frac{x^k y^l}{(\sigma + \mass)^n} \, T^n_{k, l} & = \frac{\sigma + \mass}{1 - x y} \, \frac{1}{F_\sigma^+(x)} \, \frac{1}{F_\sigma^-(y)}
}
for $x, y \in \disk$ and $\sigma > 0$, where $F_\sigma^+$ and $F_\sigma^-$ are the Wiener--Hopf factors described in Definition~\ref{def:wh}.
\end{proposition}

\begin{proof}
We divide the argument into three steps.

\emph{Step 1.}
Recall that $a_0 = 0$, $\mass = \sum_{k = -\infty}^\infty a_k$ and $F(z) = \mass - \hat a(z)$, where $\hat a$ is the generating function of $(a_k)$. Since $\lv \hat a(z) \rv \le \mass$ for $z \in \torus$, we have $\re F(z) \ge 0$, and so $\re F_\sigma(z) \ge \sigma$ for $z \in \torus$. In particular, the winding number of $0$ with respect to the curve $F_\sigma(e^{i t})$, $t \in [-\pi, \pi]$, is zero.

By Remark~\ref{rem:wh:symbol}, $F_\sigma^+ - \sigma_+$ and $F_\sigma^- - \sigma_-$ are symbols of one-sided summable $\amcm$ sequences for appropriate $\sigma^+, \sigma^- > 0$. Hence, $F_\sigma^+$ and $F_\sigma^-$ are generating functions of one-sided summable sequences. By definition, $F_\sigma^+$ and $F_\sigma^-$ are holomorphic in $\disk$ and have no zeroes in~$\disk$.

Recall the Wiener--Hopf factorisation~\eqref{eq:wh}:
\formula{
 F_\sigma(z) & = F_\sigma^+(z) F_\sigma^-(z^{-1}) ,
}
where $z \in D_F \setminus \{0\}$ (where $D_F$ is defined in~\eqref{eq:dom}). If $z \in \torus \setminus \{1\}$, then it follows that
\formula{
 F_\sigma(z) & = F_\sigma^+(z) F_\sigma^-(\overline{z}) = F_\sigma^+(z) \overline{F_\sigma^-(z)} ,
}
and by continuity this extends to $z = 1$. The above equation is an example of the Wiener--Hopf factorisation discussed in Theorem~1.14 in~\cite{bs99}, and so we may apply Krein's theorem given in Theorem~1.15 in~\cite{bs99}. This result asserts that if we denote by $T(\hat b) = (b_{k - l} : k, l \ge 0)$ the infinite Toeplitz matrix associated to the sequence $(b_k)$ with generating function $\hat b$, then
\formula{
 (T(F_\sigma))^{-1} & = T(1 / F_\sigma^+) T(\overline{1 / F_\sigma^-}) .
}
Here the inverse matrix is understood in the sense of operators on $\ell^2$.

As a part of this statement, $1 / F_\sigma^+(z)$ and $1 / \overline{F_\sigma^-(z)} = 1 / F_\sigma^-(z^{-1})$ (where $z \in \torus$) are generating functions of one-sided sequences $(b_k^+)$ and $(b_k^-)$, with $b_k^+ = 0$ for $k < 0$ and $b_k^- = 0$ for $k > 0$. Thus, if we denote by $b_{k, l}$ the entries of $(T(F_\sigma))^{-1}$, then
\formula{
 b_{k, l} & = \sum_{j = \min\{k, l\}}^\infty b_{j - k}^+ b_{l - j}^- .
}

\emph{Step 2.}
By Fubini's theorem for series, for $x, y \in \disk$ we have
\formula{
 \sum_{k = 0}^\infty \sum_{l = 0}^\infty x^k y^l b_{k, l} & = \sum_{k = 0}^\infty \sum_{l = 0}^\infty \sum_{j = 0}^{\max\{k, l\}} x^k y^l b_{k - j}^+ b_{j - l}^- \\
 & = \sum_{j = 0}^\infty \sum_{k = j}^\infty \sum_{l = j}^\infty x^k y^l b_{k - j}^+ b_{j - l}^- \\
 & = \sum_{j = 0}^\infty \biggl(\sum_{k = j}^\infty x^k b_{k - j}^+\biggr) \biggl(\sum_{l = j}^\infty y^l b_{j - l}^-\biggr) \\
 & = \sum_{j = 0}^\infty x^j y^j \biggl(\sum_{k = 0}^\infty x^k b_k^+\biggr) \biggl(\sum_{l = -\infty}^0 y^{-l} b_l^-\biggr) .
}
Summing up the geometric series and using the definition of the generating function, we obtain
\formula{
 \sum_{k = 0}^\infty \sum_{l = 0}^\infty x^k y^l b_{k, l} & = \frac{1}{1 - x y} \, \frac{1}{F_\sigma^+(x)} \, \frac{1}{F_\sigma^-(y)} \, .
}

\emph{Step 3.}
On the other hand, $T(F_\sigma) = \sigma + T(F) = \sigma + \mass - T$. Hence,
\formula{
 (T(F_\sigma))^{-1} & = (\sigma + \mass - T)^{-1} = \frac{1}{\sigma + \mass} \sum_{n = 0}^\infty \frac{1}{(\sigma + \mass)^n} \, T^n
}
is the resolvent for the operator $T$. Note that the series converges both entry-wise and in the operator norm on every $\ell^p$, $p \in [1, \infty]$, because, by induction, $\sum_{l = 0}^\infty T^n_{k, l} \le \mass^n$ and $\sum_{k = 0}^\infty T^n_{k, l} \le \mass^n$ for every $n, k, l \ge 0$.

Recall that $b_{k, l}$ are the entries of $(T(F_\sigma))^{-1}$. It follows that
\formula{
 b_{k, l} & = \frac{1}{\sigma + \mass} \sum_{n = 0}^\infty \frac{1}{(\sigma + \mass)^n} \, T^n_{k, l} .
}
Thus, for $x, y \in \disk$,
\formula{
 \sum_{n = 0}^\infty \sum_{k = 0}^\infty \sum_{l = 0}^\infty \frac{x^k y^l}{(\sigma + \mass)^n} \, T^n_{k, l} & = (\sigma + \mass) \sum_{k = 0}^\infty \sum_{l = 0}^\infty x^k y^l b_{k, l} .
}
The desired result follows by combining this with the result of the previous step.
\end{proof}


\subsection{Random walks and fluctuation theory}
\label{sec:rw}

Consider a random walk $(X_n)$ on $\Z$ with increments equal to $k$ with probability $a_k$. More precisely, we assume that the increments $X_{n + 1} - X_n$ are i.i.d.\@ random variables with probability mass function $(a_k)$. Typically, we consider $X_0 = 0$ with probability one, but if the initial value $X_0$ is equal to $l$ with probability one, then we write $\pr[\cdot \mid X_0 = l]$ for the corresponding probability.

We assume that $(a_k)$ is an $\amcm$ sequence, possibly with $a_0 > 0$, and in this case we say that $(X_n)$ is an $\amcm$ random walk. If $a_k = 0$ for $k \ne 0$, we say that $(X_n)$ is \emph{trivial}. A nontrivial random walk $(X_n)$ is \emph{one-sided} if $a_k = 0$ for all $k \le -1$ or $a_k = 0$ for all $k \ge 1$. Otherwise, we say that $(X_n)$ is \emph{two-sided}. We are primarily interested in two-sided $\amcm$ random walks.

Observe that the double-infinite matrix $(a_{k - l} : k, l \in \Z)$ is the transpose of the transition matrix of $(X_n)$, and the Toeplitz matrix $T = (a_{k - l} : k, l \ge 0)$ is the transpose of the Markov chain obtained by \emph{killing} $(X_n)$ when it becomes negative. More precisely, if $X_n^+ = X_n$ when $X_0, X_1, \ldots, X_n \ge 0$ and $X_n^+ = \infty$ otherwise, then the transpose of $T$ is the transition matrix of $(X_n^+)$ (with the absorbing \emph{cemetery state} $\infty$ omitted).

The symbol $F$ of the $\amcm$ sequence $(a_k)$ is closely related with the generator matrix of $(X_n)$, and we say that $F$ is the \emph{symbol} of $(X_n)$.

We recall some fundamentals of the fluctuation theory of random walks. We continue to assume that the sequence $(a_k)$ is $\amcm$, although in fact the results discussed below hold for general random walks on $\Z$.

Denote the running minima and the running maxima of $(X_n)$ by
\formula{
 \ol{X}_n & = \max\{X_0, X_1, \ldots, X_n\} , \\
 \ul{X}_n & = \min\{X_0, X_1, \ldots, X_n\} .
}
Consider a random variable $N$ drawn from the geometric distribution with parameter $q \in (0, 1)$: $\pr[N = n] = (1 - q) q^n$ for $n = 0, 1, \ldots$\,, and assume that $N$ and the process $(X_n)$ are independent. It is a fundamental result in the fluctuation theory of random walks that
\formula**{
\notag
 \text{$\ul{X}_N$ and $X_N - \ul{X}_N$ are independent;} \\
\label{eq:wh:max:min}
 \text{$X_N - \ul{X}_N$ is equal in distribution to $\ol{X}_N$;} \\
\notag
 \text{$X_N - \ol{X}_N$ is equal in distribution to $\ul{X}_N$.}
}
Since $N$ and the process $(X_n)$ are independent, the generating function $G_q$ of the probability mass function of $X_N$ is given by
\formula{
 G_q(z) & = \ex[z^{X_N}] = \sum_{n = 0}^\infty \pr[N = n] \ex[z^{X_n}]
}
when $z \in \torus$. The increments of the process $(X_n)$ are i.i.d., and therefore
\formula{
 G_q(z) & = \sum_{n = 0}^\infty \pr[N = n] (\ex[z^{X_1}])^n .
}
The generating function of $(a_k)$ is $\hat a(z) = 1 - F(z)$. Hence,
\formula{
 G_q(z) & = \sum_{n = 0}^\infty (1 - q) q^n (\hat a(z))^n = \frac{1 - q}{1 - q \hat a(z)} = \frac{1 - q}{1 - q + q F(z)} \, .
}
Recall that we denoted $F_\sigma(z) = \sigma + F(z)$. Thus, if $\sigma = q^{-1} - 1$, we obtain
\formula{
 G_q(z) & = \frac{\sigma}{F_\sigma(z)} \, .
}

Let $G_q^+(\cdot)$ and $G_q^-(\cdot)$ denote the generating functions of $\ol{X}_N$ and $-\ul{X}_N$, respectively. By~\eqref{eq:wh:max:min},
\formula{
 G_q(z) & = \ex[z^{X_N}] \\
 & = \ex[z^{X_N - \ul{X}_N}] \ex[z^{\ul{X}_N}] \\
 & = \ex[z^{\ol{X}_N}] \ex[z^{\ul{X}_N}] \\
 & = G_q^+(z) G_q^-(z^{-1}) .
}
Here $z \in \torus$. Note, however, that $G_q^+$ and $G_q^-$ are holomorphic functions on the unit disk $\disk$, continuous up to the boundary. Furthermore, as these are generating functions of infinitely divisible distributions, they are zero-free in $\disk$. It follows that $G_q^+(z)$ and $G_q^-(z^{-1})$ are the \emph{Wiener--Hopf factors} of $G_q(z)$.

The Wiener--Hopf factors are unique, up to multiplication by a constant, and in Definition~\ref{def:wh} we described the Wiener--Hopf factors $F_\sigma^+$ and $F_\sigma^-$ of $F_\sigma$. It follows that $G_q^+$ and $G_q^-$ are, up to multiplication by a constant, equal to $1 / F_\sigma^+$ and $1 / F_\sigma^-$, respectively. Since $G_q^+(1) = G_q^-(1) = 1$, we conclude that for $x, y \in \disk$, $q \in (0, 1)$ and $\sigma = q^{-1} - 1$,
\formula[]{
\label{eq:wh:pr:p}
 \ex[x^{\ol{X}_N}] & = G_q^+(x) = \frac{F_\sigma^+(1)}{F_\sigma^+(x)} \, , \\
\label{eq:wh:pr:m}
 \ex[y^{-\ul{X}_N}] & = G_q^-(y) = \frac{F_\sigma^-(1)}{F_\sigma^-(y)} \, .
}
These identities allow for a detailed analysis of the distributions of $\ol{X}_n$ and $\ul{X}_n$, in a similar way as the running suprema and infima of Lévy processes were studied in~\cite{k25}. This is, however, beyond the scope of the present paper.


\subsection{Transition probabilities}
\label{sec:rw:wh}

We now provide an alternative proof of Proposition~\ref{prop:wh}, using probabilistic methods. Consider the transition probabilities of the Markov chain $(X_n^+)$, obtained by killing $(X_n)$ when it becomes negative:
\formula{
 p_n^+(l, k) & = \pr\bigl[X_n^+ = k \mid X_0^+ = l \bigr] = \pr\bigl[X_n = k, \, \ul{X}_n \ge 0 \mid X_0 = l \bigr] .
}
Note that $p_1^+(l, k) = a_{k - l}$ when $k, l \ge 0$, and hence $T = (p_1^+(l, k) : k, l \ge 0)$. It follows that $T^n = (p_n^+(l, k) : k, l \ge 0)$, that is, $p_n^+(l, k) = T^n_{k, l}$. In other words, the transpose of $T^n$ is the $n$ step transition matrix of $(X_n^+)$.

Let $P_q^+(x, y)$ denote the trivariate generating function of $p_n^+(l, x)$: when $x, y \in \disk$ and $q \in (0, 1)$, we define
\formula{
 P_q^+(x, y) & = \sum_{n = 0}^\infty \sum_{k = 0}^\infty \sum_{l = 0}^\infty q^n x^k y^l p_n^+(l, k) \\
 & = \sum_{n = 0}^\infty \sum_{k = 0}^\infty \sum_{l = 0}^\infty q^n x^k y^l T^n_{k, l} .
}
Suppose that $N$ is a random variable drawn from the geometric distribution with parameter $q$, independent from the process $(X_n)$. Then,
\formula{
 P_q^+(x, y) & = \frac{1}{1 - q} \sum_{n = 0}^\infty \sum_{k = 0}^\infty \sum_{l = 0}^\infty \pr[N = n] x^k y^l \pr\bigl[X_n = k, \, \ul{X}_n \ge 0 \mid X_0 = l\bigr] \\
 & = \frac{1}{1 - q} \sum_{n = 0}^\infty \sum_{k = 0}^\infty \sum_{l = 0}^\infty \pr[N = n] x^k y^l \pr\bigl[l + X_n = x, \, l + \ul{X}_n \ge 0\bigr] .
}
By independence and Fubini's theorem,
\formula{
 P_q(x, y) & = \frac{1}{1 - q} \sum_{n = 0}^\infty \sum_{k = 0}^\infty \sum_{l = 0}^\infty x^k y^l \pr\bigl[N = n, \, l + X_n = k, l + \ul{X}_n \ge 0\bigr] \\
 & = \frac{1}{1 - q} \, \ex \biggl[\sum_{l = 0}^\infty x^{l + X_N} y^l \ind_{\{l + \ul{X}_N \ge 0\}}\biggr] \\
 & = \frac{1}{1 - q} \, \ex \biggl[\sum_{l = -\ul{X}_N}^\infty x^{l + X_N} y^l\biggr] \\
 & = \frac{1}{1 - q} \, \ex \biggl[\frac{x^{X_N - \ul{X}_N} y^{-\ul{X}_N}}{1 - x y}\biggr] .
}
By~\eqref{eq:wh:max:min}, we find that
\formula{
 P_q^+(x, y) & = \frac{1}{1 - q} \, \frac{1}{1 - x y} \, \ex[x^{X_N - \ul{X}_N}] \ex[y^{-\ul{X}_N}] \\
 & = \frac{1}{1 - q} \, \frac{1}{1 - x y} \, \ex[x^{\ol{X}_N}] \ex[y^{-\ul{X}_N}] \\
 & = \frac{1}{1 - q} \, \frac{1}{1 - x y} \, G_q^+(x) G_q^-(y) .
}
Finally, by~\eqref{eq:wh:pr:p} and~\eqref{eq:wh:pr:m},
\formula{
 P_q^+(x, y) & = \frac{1}{1 - q} \, \frac{1}{1 - x y} \, \frac{F_\sigma^+(1)}{F_\sigma^+(x)} \, \frac{F_\sigma^-(1)}{F_\sigma^-(y)} \, .
}
where $\sigma = q^{-1} - 1$. Since $F_\sigma^+(1) F_\sigma^-(1) = F_\sigma(1) = \sigma$ and $q = 1 / (\sigma + 1)$ and $1 / (1 - q) = (\sigma + 1) / \sigma$, we conclude that
\formula{
 P_{1 / (\sigma + 1)}^+(x, y) & = \frac{\sigma + 1}{1 - x y} \, \frac{1}{F_\sigma^+(x)} \, \frac{1}{F_\sigma^-(y)} \, .
}
This proves Proposition~\ref{prop:wh} if $(a_k)$ is a \emph{probabilistic} $\amcm$ sequence (in this case $\mass = 1$).

We remark that by considering the sequence $(\mass^{-1} a_k)$, one can easily reduce Proposition~\ref{prop:wh} to the case of probabilistic $\amcm$ sequences; we omit the details.

%
%

\section{Inversion}
\label{sec:inv}

Recall that our main goal is to study the the powers $T^n$ of the infinite Toeplitz matrix $T = (a_{k - l} : k, l \ge 0)$ corresponding to a summable $\amcm$ sequence $(a_k)$ with $a_0 = 0$. In the previous section, in Proposition~\ref{prop:wh}, we expressed the trivariate generating function of the entries $T^n_{k, l}$ of $T^n$ in terms of the Wiener--Hopf factors of the symbol $F$:
\formula[eq:tri]{
 \sum_{n = 0}^\infty \sum_{k = 0}^\infty \sum_{l = 0}^\infty \frac{x^k y^l}{(\sigma + \mass)^n} \, T^n_{k, l} & = \frac{\sigma + \mass}{1 - x y} \, \frac{1}{F_\sigma^+(x)} \, \frac{1}{F_\sigma^-(y)} \, .
}
Inversion of this trivariate transform is done in two steps: first, we deal with the temporal variable $n$, and only then we handle spatial variables $x, y$.


\subsection{Inversion of temporal transform}
\label{sec:temp}

The following analogue of Proposition~4.4 in~\cite{k25} inverts the transform in the temporal variable $n$. Recall that the regions $D_F^+$ and $D_F^-$ are defined in~\eqref{eq:dom:m} and~\eqref{eq:dom:p}.

\begin{proposition}
\label{prop:bi}
Let $F$ be the symbol of a two-sided summable $\amcm$ sequence $(a_k)$ with $a_0 = 0$, let $T = (a_{k - l} : k, l \ge 0)$ be the corresponding infinite Toeplitz matrix, and let $T^n_{k, l}$ denote the entries of the $n$th power of $T$. For $x \in (0, 1) \cap D_F^+$ and $y^{-1} \in (1, \infty) \cap D_F^-$, we have
\formula{
 \sum_{k = 0}^\infty \sum_{l = 0}^\infty x^k y^l T^n_{k, l} & = \frac{1}{\pi} \int_{\Lambda} (\mass - \lambda)^n \frac{1 - x}{\lv \gamma(\lambda) - x \rv^2 F^+(\lambda; x)} \, \frac{1 - y}{\lv (\gamma(\lambda))^{-1} - y \rv^2 F^-(\lambda; y)} \, \frac{\im \gamma(\lambda)}{\lv \gamma(\lambda) \rv} \, d\lambda .
}
\end{proposition}

\begin{proof}
We argue as in the proof of Proposition~4.4 in~\cite{k25}. We divide the argument into five steps.

\emph{Step 1.}
Our starting point is~\eqref{eq:tri}. We fix $x, y \in \disk$, and we denote
\formula{
 G(\sigma) & = \frac{1}{1 - x y} \, \frac{1}{F_\sigma^+(x)} \, \frac{1}{F_\sigma^-(y)} \, ,
}
so that the right-hand side of~\eqref{eq:tri} is $(\sigma + \mass) G(\sigma)$. Applying Proposition~\ref{prop:wh:int}, we find that
\formula{
 G(\sigma) & = \frac{1}{1 - x y} \, \frac{1}{\sigma} \, \exp\biggl(-\frac{1}{\pi} \int_0^{\limit} \psi(\lambda) \, \frac{d\lambda}{\sigma + \lambda}\biggr) ,
}
where
\formula{
 \psi(\lambda) & = \arg(\gamma(\lambda) - y^{-1}) - \arg(\gamma(\lambda) - x) .
}
Furthermore, $G$ is the reciprocal of a complete Bernstein function, and hence a Stieltjes function of $\sigma$. We study the integral representation~\eqref{eq:stieltjes:int} of $G$:
\formula{
 G(\sigma) & = c_0 + \frac{c_{-1}}{\sigma} + \frac{1}{\pi} \int_{(0, \infty)} \frac{1}{\sigma + \lambda} \, \nu(d\lambda) ,
}
where $c_0, c_{-1} \ge 0$ and $\nu$ is a nonnegative measure on $(0, \infty)$. Before we proceed, let us record three properties of $\psi$:
\begin{itemize}
\item $\psi$ is locally Hölder continuous on $(0, \limit)$;
\item $0 \le \psi(\lambda) \le \pi$ for $\lambda \in (0, \limit)$;
\item $\psi(0^+) = \pi$ and $\psi(\limit^-) = 0$.
\end{itemize}
In order to prove these properties, it is convenient to note that $\psi(\lambda)$ is the interior angle at $\gamma(\lambda)$ of the (possibly degenerate) triangle with vertices $x$, $y^{-1}$ and $\gamma(\lambda)$. Since $x \in D_F^+$ and $y^{-1} \in D_F^-$, we have $\gamma(\lambda) \ne x$ and $\gamma(\lambda) \ne y^{-1}$. Thus, $\gamma(\lambda) - x$ and $\gamma(\lambda) - y^{-1}$ are non-zero and have nonnegative imaginary part. Since the complex argument is differentiable in this region, and $\gamma$ is locally Hölder continuous by Proposition~\ref{prop:spine}\ref{it:spine:f}, the first property follows. The second one is a consequence of $0 < x < y^{-1}$: the numbers $\gamma(\lambda) - x$ and $\gamma(\lambda) - y^{-1}$ have equal nonnegative imaginary part, but the former has a greater real part, and hence a smaller or equal argument. To prove the third one, simply note that $\gamma(\limit^-) \le 0 < x < \gamma(0^+) < y^{-1}$.

\emph{Step 2.}
Recall that $c_0$, $c_{-1}$ and $\nu$ are given by~\eqref{eq:stieltjes:constants} and~\eqref{eq:stieltjes:measure}. Since $\psi \ge 0$, we find that $0 \le G(\lambda) \le (1 - x y)^{-1} \sigma^{-1}$, and so
\formula{
 c_0 & = \lim_{\sigma \to \infty} G(\sigma) = 0 .
}
Furthermore, $\psi(0^+) = \pi$, and so, by the monotone convergence theorem,
\formula{
 c_{-1} & = \lim_{\sigma \to 0^+} (\sigma G(\sigma)) = \frac{1}{1 - x y} \, \exp\biggl( -\frac{1}{\pi} \int_0^{\limit} \psi(\lambda) \, \frac{d\lambda}{\lambda} \biggr) = \frac{1}{1 - x y} \, e^{-\infty} = 0 .
}
Finally, $\nu$ is the measure on $(0, \infty)$ given by
\formula{
 \nu(ds) & = \lim_{t \to 0^+} (-\im G(-s + i t)) ds .
}
By definition, $G$ is holomorphic in $\C \setminus [-\limit, 0]$, and hence $\nu$ is concentrated on $(0, \limit]$. We claim that $\nu(\{\limit\}) = 0$.

By the dominated convergence theorem,
\formula{
 \frac{\nu(\{\limit\})}{\pi} & = \lim_{\sigma \to \limit^+} \int_{(0, \limit]} \frac{\sigma - \limit}{\sigma - \lambda} \, \nu(d\lambda) = \lim_{\sigma \to \limit^+} (-(\sigma - \limit) G(-\sigma)) .
}
Using the definition of $G$, the identity $\int_0^{\limit} (\sigma - \lambda)^{-1} d\lambda = \log \limit - \log(\sigma - \limit)$, and the monotone convergence theorem, we obtain
\formula{
 \frac{\nu(\{\limit\})}{\pi} & = \lim_{\sigma \to \limit^+} \frac{1}{1 - x y} \, \frac{\sigma - \limit}{\sigma} \, \exp\biggl(\frac{1}{\pi} \int_0^{\limit} \psi(\lambda) \, \frac{d\lambda}{\sigma - \lambda}\biggr) \\
 & = \lim_{\sigma \to \limit^+} \frac{1}{1 - x y} \, \exp\biggl(-\frac{1}{\pi} \int_0^{\limit} (\pi - \psi(\lambda)) \, \frac{d\lambda}{\sigma - \lambda}\biggr) \\
 & = \frac{1}{1 - x y} \, \exp\biggl(-\frac{1}{\pi} \int_0^{\limit} (\pi - \psi(\lambda)) \, \frac{d\lambda}{\limit - \lambda}\biggr) .
}
Finally, $\psi(\limit^-) = 0$, and hence
\formula{
 \frac{\nu(\{\limit\})}{\pi} & = \frac{1}{1 - x y} \, e^{-\infty} = 0 ,
}
as claimed.

\emph{Step 3.}
Below we prove that $G$ has a continuous boundary limit on $(-\limit, 0)$, and hence $\nu$ is in fact absolutely continuous. Note that
\formula[]{
\notag
 G(-s + i t) & = \frac{1}{1 - x y} \, \frac{e^{-i \arg(-s + i t)}}{\sqrt{s^2 + t^2}} \, \exp\biggl(-\frac{1}{\pi} \int_0^{\limit} \psi(\lambda) \, \frac{\lambda - s - i t}{(\lambda - s)^2 + t^2} \, d\lambda\biggr) \\
\label{eq:bi:g}
 & = \frac{1}{1 - x y} \, \frac{1}{\sqrt{s^2 + t^2}} \, \exp\biggl(-\frac{1}{\pi} \int_0^{\limit} \psi(\lambda) \, \frac{\lambda - s}{(\lambda - s)^2 + t^2} \, d\lambda\biggr) \\
\notag
 & \qquad \times \exp\biggl(-i \arg(-s + i t) + \frac{i}{\pi} \int_0^{\limit} \psi(\lambda) \, \frac{t}{(\lambda - s)^2 + t^2} \, d\lambda\biggr) .
}
We study the two integrals on the right-hand side.

Since $\psi(s)$ is bounded and locally Hölder continuous on $(0, \limit)$, by the dominated convergence theorem, for every $s \in (0, \limit)$ we have
\formula{
 & \lim_{t \to 0^+} \int_0^{\limit} \psi(\lambda) \, \frac{\lambda - s}{(\lambda - s)^2 + t^2} \, d\lambda \\
 & \qquad = \lim_{t \to 0^+} \biggl(\int_0^{\limit} \psi(\lambda) \, \frac{\lambda - s}{(\lambda - s)^2 + t^2} \, d\lambda - \psi(s) \int_{s - 1}^{s + 1} \frac{\lambda - s}{(\lambda - s)^2 + t^2} \, d\lambda\biggr) \\
 & \qquad = \lim_{t \to 0^+} \int_{-\infty}^\infty \bigl(\psi(\lambda) \ind_{(0, \limit)}(\lambda) - \psi(s) \ind_{(-1, 1)}(\lambda - s)\bigr) \, \frac{\lambda - s}{(\lambda - s)^2 + t^2} \, d\lambda \\
 & \qquad = \int_{-\infty}^\infty \bigl(\psi(\lambda) \ind_{(0, \limit)}(\lambda) - \psi(s) \ind_{(-\delta, \delta)}(\lambda - s)\bigr) \, \frac{1}{\lambda - s} \, d\lambda \\
 & \qquad = \pvint_0^{\limit} \frac{\psi(\lambda)}{\lambda - s} \, d\lambda .
}
Furthermore, the integral under the limit is bounded uniformly with respect to $t \in (0, 1)$ and $s$ in an arbitrary compact subinterval of $(0, \limit)$.

The other integral in~\eqref{eq:bi:g} is a Poisson integral, and so continuity of $\psi$ implies that
\formula{
 \lim_{t \to 0^+} \frac{1}{\pi} \int_0^{\limit} \psi(\lambda) \, \frac{t}{(\lambda - s)^2 + t^2} \, d\lambda & = \psi(s) .
}
It follows that for every $s \in (0, \limit)$, the pointwise limit
\formula{
 \lim_{t \to 0^+} G(-s + i t) & = \frac{1}{1 - x y} \, \frac{1}{s} \, \exp\biggl(-\frac{1}{\pi} \pvint_0^{\limit} \frac{\psi(\lambda)}{\lambda - s} \, d\lambda - i \pi + i \psi(s)\biggr)
}
exists, and furthermore $G(-s + i t)$ is bounded uniformly in $t \in (0, 1)$ and $s$ in an arbitrary compact subinterval of $(0, \limit)$. Since $\nu$ is the vague limit of measures $-\pi^{-1} \im G(-s + i t) ds$ on $(0, \limit)$, it follows that $\nu$ is indeed absolutely continuous, with density function that we denote by the same symbol $\nu$:
\formula{
 \nu(s) & = \lim_{t \to 0^+} (-\im G(-s + i t)) = \frac{1}{1 - x y} \, \frac{\sin \psi(s)}{s} \, \exp\biggl(-\frac{1}{\pi} \pvint_0^{\limit} \frac{\psi(\lambda)}{\lambda - s} \, d\lambda\biggr) .
}
It will be convenient later to have the above formula with the roles of $s$ and $\lambda$ exchanged:
\formula[eq:bi:nu:1]{
 \nu(\lambda) & = \frac{1}{1 - x y} \, \frac{\sin \psi(\lambda)}{\lambda} \, \exp\biggl(-\frac{1}{\pi} \pvint_0^{\limit} \frac{\psi(s)}{s - \lambda} \, ds\biggr) .
}
Observe that if $\lambda \in (0, \limit) \setminus \Lambda$, then $\im \gamma(\lambda) = 0$, so that $\psi(\lambda)$ is either $0$ or $\pi$. It follows that $\sin \psi(\lambda) = 0$, and hence $\nu(\lambda) = 0$. We conclude that
\formula[eq:bi:g:nu]{
 G(\sigma) & = \frac{1}{\pi} \int_{\Lambda} \frac{1}{\sigma + \lambda} \, \nu(\lambda) d\lambda.
}

\emph{Step 4.}
We simplify the expression~\eqref{eq:bi:nu:1} for the density function $\nu$. Using the formula for the sine of a difference, we find that
\formula{
 \sin \psi(\lambda) & = \sin \bigl(\arg(\gamma(\lambda) - y^{-1}) - \arg(\gamma(\lambda) - x)\bigr) \\
 & = \frac{\re \gamma(\lambda) - x}{\lv \gamma(\lambda) - x \rv} \, \frac{\im \gamma(\lambda)}{\lv \gamma(\lambda) - y^{-1} \rv}  - \frac{\im \gamma(\lambda)}{\lv \gamma(\lambda) - x \rv} \, \frac{\re \gamma(\lambda) - y^{-1}}{\lv \gamma(\lambda) - y^{-1} \rv} \\
 & = \frac{(y^{-1} - x) \im \gamma(\lambda)}{\lv \gamma(\lambda) - x \rv \lv \gamma(\lambda) - y^{-1} \rv} \, .
}
Hence,
\formula[eq:bi:nu:2]{
 \nu(\lambda) & = \frac{\im \gamma(\lambda)}{\lambda y \lv \gamma(\lambda) - x \rv \lv \gamma(\lambda) - y^{-1} \rv} \, \exp\biggl(-\frac{1}{\pi} \pvint_0^{\limit} \frac{\psi(s)}{s - \lambda} \, ds\biggr) .
}
We claim that
\formula{
 -\frac{1}{\pi} \pvint_0^{\limit} \frac{\psi(s)}{s - \lambda} \, ds & = \log \lambda + \frac{1}{\pi} \int_{\Lambda} \psi'(s) \log \lv s - \lambda \rv ds .
}
Indeed, fix a sufficiently small $\delta > 0$. By Proposition~\ref{prop:spine}\ref{it:spine:d}, $\lv \gamma' \rv$ is integrable over $[\delta, \limit - \delta]$, and so $\psi$ is absolutely continuous on $[\delta, \limit - \delta]$. It follows that
\formula{
 -\biggl(\int_\delta^{\lambda - \delta} + \int_{\lambda + \delta}^{\limit - \delta}\biggr) \frac{\psi(s)}{s - \lambda} \, ds & = \psi(\delta) \log \lv \lambda - \delta \rv - \psi(\limit - \delta) \log \lv \lambda - \limit + \delta \rv \\
 & \qquad + (\psi(\lambda + \delta) - \psi(\lambda - \delta)) \log \delta \\
 & \qquad\qquad + \biggl(\int_\delta^{\lambda - \delta} + \int_{\lambda + \delta}^{\limit - \delta}\biggr) \psi'(s) \log \lv s - \lambda \rv ds .
}
For almost every $\lambda \in (0, \limit)$ the derivative $\psi'(\lambda)$ exists, and so we may pass to the limit as $\delta \to 0^+$. It follows that
\formula{
 -\pvint_0^{\limit} \frac{\psi(s)}{s - \lambda} \, ds & = \psi(0^+) \log \lambda - \psi(\limit^-) \log (\limit - \lambda) + 0 + \int_0^{\limit} \psi'(s) \log \lv s - \lambda \rv ds \\
 & = \pi \log \lambda + \int_0^{\limit} \psi'(s) \log \lv s - \lambda \rv ds .
}
Recall that if $s \in (0, \limit) \setminus \Lambda$, then $\im \gamma(s) = 0$, and hence $\psi(s) = 0$ or $\psi(s) = \pi$. If $s$ is not an isolated point of $(0, \limit) \setminus \Lambda$ and $\psi'(s)$ exists, then it follows that necessarily $\psi'(s) = 0$. Since there are countably many isolated points of $(0, \limit) \setminus \Lambda$, we have $\psi'(s) = 0$ for almost every $s \in (0, \limit) \setminus \Lambda$. Thus,
\formula{
 -\pvint_0^{\limit} \frac{\psi(s)}{s - \lambda} \, ds & = \pi \log \lambda + \int_{\Lambda} \psi'(s) \log \lv s - \lambda \rv ds ,
}
proving our claim.

Recall that we defined $\gamma(-s) = \overline{\gamma(s)}$, so that $\{\gamma(s) : s \in (-\Lambda) \cup \Lambda\}$ is the parameterisation of the symmetrised spine $\Gamma^\star$ (except for countably many endpoints; see~\eqref{eq:spine:star}). Since the argument is the imaginary part of the complex logarithm, for $s \in \Lambda$ we have
\formula{
 \psi'(s) & = \im \biggl(\frac{\gamma'(s)}{\gamma(s) - y^{-1}} - \frac{\gamma'(s)}{\gamma(s) - x} \biggr) \\
 & = \frac{1}{2 i} \biggl(\frac{\gamma'(s)}{\gamma(s) - y^{-1}} - \frac{\gamma'(s)}{\gamma(s)-x} \biggr) - \frac{1}{2 i} \overline{\biggl(\frac{\gamma'(s)}{\gamma(s) - y^{-1}} - \frac{\gamma'(s)}{\gamma(s) - x}\biggr)} \\
 & = \frac{1}{2 i} \biggl(\frac{\gamma'(s)}{\gamma(s) - y^{-1}} - \frac{\gamma'(s)}{\gamma(s)-x} \biggr) + \frac{1}{2 i} \biggl(\frac{\gamma'(-s)}{\gamma(-s) - y^{-1}} - \frac{\gamma'(-s)}{\gamma(-s) - x}\biggr) .
}
Therefore,
\formula{
 -\frac{1}{\pi} \pvint_0^{\limit} \frac{\psi(s)}{s - \lambda} \, ds & = \log \lambda + \frac{1}{2 \pi i} \int_{(-\Lambda) \cup \Lambda} \biggl(\frac{1}{\gamma(s) - y^{-1}} - \frac{1}{\gamma(s) - x}\biggr) \log \lv s - \lambda \rv \gamma'(s) ds .
}
Recall that if $z = \gamma(s)$, then $s = F(z)$. It follows that
\formula{
 -\frac{1}{\pi} \pvint_0^{\limit} \frac{\psi(s)}{s - \lambda} \, ds & = \log \lambda + \frac{1}{2 \pi i} \int_{\Gamma^\star} \biggl(\frac{1}{z - y^{-1}} - \frac{1}{z - x}\biggr) \log \lv F(z) - \lambda \rv dz .
}
The final step of the simplification is provided by Proposition~\ref{prop:integral}: we have
\formula{
 \exp\biggl(-\frac{1}{\pi} \pvint_0^{\limit} \frac{\psi(s)}{s - \lambda} \, ds\biggr) & = \lambda \, \frac{\lv \gamma(\lambda) \rv (1 - x) (y^{-1} - 1)}{\lv \gamma(\lambda) - x \rv \lv \gamma(\lambda) - y^{-1} \rv F^+(\lambda; x) F^-(\lambda; y)} \, .
}
Inserting this into \eqref{eq:bi:nu:2} yields
\formula[eq:bi:nu:3]{
 \nu(\lambda) & = \frac{\lv \gamma(\lambda) \rv (1 - x) (1 - y) \im \gamma(\lambda)}{y^2 \lv \gamma(\lambda) - x \rv^2 \lv \gamma(\lambda) - y^{-1} \rv^2 F^+(\lambda; x) F^-(\lambda; y)}
}
for $\lambda \in \Lambda$.

\emph{Step 5.}
We are now ready to invert the temporal transform. By~\eqref{eq:tri} and~\eqref{eq:bi:g:nu}, we have
\formula{
 \sum_{n = 0}^\infty \sum_{k = 0}^\infty \sum_{l = 0}^\infty \frac{x^k y^l}{(\sigma + \mass)^n} \, T^n_{k, l} & = (\sigma + \mass) G(\sigma) = \frac{1}{\pi} \int_{\Lambda} \frac{\sigma + \mass}{\sigma + \lambda} \, \nu(\lambda) d\lambda .
}
Recall that, by Proposition~\ref{prop:spine}\ref{it:spine:e}, $\sup \Lambda \le \limit \le 2 \mass$, and so $\lv \mass - \lambda \rv < \mass$ for $\lambda \in \Lambda$. It follows that $\lv \mass - \lambda \rv / \lv \sigma + \mass \rv < 1$ when $\lambda \in \Lambda$ and $\sigma > 0$, and so, by Fubini's theorem,
\formula{
 \sum_{n = 0}^\infty \sum_{k = 0}^\infty \sum_{l = 0}^\infty \frac{x^k y^l}{(\sigma + \mass)^n} \, T^n_{k, l} & = \frac{1}{\pi} \int_{\Lambda} \frac{1}{1 - (\mass - \lambda) / (\sigma + \mass)} \, \nu(\lambda) d\lambda \\
 & = \sum_{n = 0}^\infty \int_{\Lambda} \frac{(\mass - \lambda)^n}{(\sigma + \mass)^n} \, \nu(\lambda) d\lambda .
}
By the uniqueness of the generating function,
\formula{
 \sum_{k = 0}^\infty \sum_{l = 0}^\infty x^k y^l T^n_{k, l} & = \frac{1}{\pi} \int_{\Lambda} (\mass - \lambda)^n \nu(\lambda) d\lambda ,
}
and the desired result follows from~\eqref{eq:bi:nu:3}.
\end{proof}


\subsection{Inversion of spatial transforms}
\label{sec:spat}

After inverting the temporal transform in the trivariate generating function of $T^n_{k, l}$, we now deal with the remaining bivariate spatial transform. The following analogue of Proposition~4.6 in~\cite{k25} essentially proves Theorem~\ref{thm:main}.

\begin{proposition}
\label{prop:main}
Let $F$ be the symbol of a two-sided summable $\amcm$ sequence $(a_k)$ with $a_0 = 0$, let $T = (a_{k - l} : k, l \ge 0)$ be the corresponding infinite Toeplitz matrix, and let $T^n_{k, l}$ denote the entries of the $n$th power of $T$. Suppose that the spine $\Gamma$ of $F$ ends at a point in $(-\infty, 0)$, that is,
\formula{
 \gamma(\limit) & \in (-\infty, 0) .
}
Then
\formula[eq:main]{
 T^n_{k, l} & = \frac{2}{\pi} \int_{\Lambda} (\mass - \lambda)^n \ev^+_{\lambda, k} \ev^-_{\lambda, l} \, \frac{\lv \gamma'(\lambda) \rv}{\lv \gamma(\lambda) \rv} \, d\lambda ,
}
where
\formula{
 \ev^+_{\lambda, k} & = \lv \gamma(\lambda) \rv^{-k - 1} \sin\biggl((k + 1) \arg \gamma(\lambda) - \arg \frac{1 - \gamma(\lambda)}{F^+(\lambda;\gamma(\lambda))}\biggr) - \evr^+_{\lambda, k} , \\
 \ev^-_{\lambda, l} & = \lv \gamma(\lambda) \rv^{l + 1} \sin\biggl((l + 1) \arg \gamma(\lambda) + \arg \frac{1 - (\gamma(\lambda))^{-1}}{F^-(\lambda;(\gamma(\lambda))^{-1})}\biggr) - \evr^-_{\lambda, l} ,
}
and $\evr^+_{\lambda, k}$, $\evr^-_{\lambda, l}$ are completely monotone sequences determined by the generating functions
\formula{
 \sum_{k = 0}^\infty \ev^+_{\lambda, k} x^k & = \frac{\lv F^+(\lambda; \gamma(\lambda)) \rv \im \gamma(\lambda)}{\lv 1 - \gamma(\lambda) \rv} \, \frac{1 - x}{(\gamma(\lambda) - x) (\overline{\gamma}(\lambda) - x) F^+(\lambda; x)} \, , \\
 \sum_{l = 0}^\infty \ev^-_{\lambda, l} y^l & = \frac{\lv F^-(\lambda; (\gamma(\lambda))^{-1}) \rv \im (\overline{\gamma}(\lambda))^{-1}}{\lv 1 - (\gamma(\lambda))^{-1} \rv} \, \frac{1 - y}{((\gamma(\lambda))^{-1} - y) ((\overline{\gamma}(\lambda))^{-1} - y) F^-(\lambda; y)} \, ,
}
where $\lambda \in \Lambda$, $x, y \in \C$, and $\lv x \rv < \min \{\lv \gamma(\lambda) \rv, 1\}$, $\lv y \rv < \min \{\lv \gamma(\lambda) \rv^{-1}, 1\}$.
\end{proposition}

\begin{proof}
In Proposition~\ref{prop:bi}, we proved that for $x \in (0, 1) \cap D_F^+$ and $y^{-1} \in (1, \infty) \cap D_F^-$,
\formula[eq:main:bi]{
 \sum_{k = 0}^\infty \sum_{l = 0}^\infty x^k y^l T^n_{k, l} & = \frac{1}{\pi} \int_{\Lambda} (\mass - \lambda)^n H_+(\lambda; x) H_-(\lambda; y) \, \frac{\im \gamma(\lambda)}{\lv \gamma(\lambda) \rv} \, d\lambda ,
}
where
\formula[]{
\label{eq:main:hpm:p}
 H_+(\lambda; x) & = \frac{1 - x}{\lv \gamma(\lambda) - x \rv^2 F^+(\lambda; x)} \, , \\
\label{eq:main:hpm:m}
 H_-(\lambda; y) & = \frac{1 - y}{\lv (\gamma(\lambda))^{-1} - y \rv^2 F^-(\lambda; y)} \, .
}
Below we fix $\lambda \in \Lambda$ and we identify $H_+(\lambda; x)$ and $H_-(\lambda; x)$, up to multiplication by a constant, with the generating functions of $(\ev^+_{\lambda, k})$ and $(\ev^-_{\lambda, l})$, respectively. The desired result is then obtained using Fubini's theorem.

The argument is divided into five steps.

\emph{Step 1.}
We fix $\lambda \in \Lambda$, and, to simplify the notation, in this step we write $\gamma$ instead of $\gamma(\lambda)$. Recall that by~\eqref{eq:fr:cbf:p},
\formula{
 h(\xi) & = \frac{\xi}{F^+(\lambda; 1 - \xi)}
}
is a complete Bernstein function of $\xi$. Let $\zeta = 1 - \overline{\gamma}$. By Proposition~\ref{prop:cbf:s}, there is a Stieltjes function $g$ such that
\formula{
 \frac{h(\xi)}{(\xi - \zeta) (\xi - \overline{\zeta})} & = \frac{1}{2 i \im \zeta} \biggl(\frac{h(\zeta)}{\xi - \zeta} - \frac{h(\overline{\zeta})}{\xi - \overline{\zeta}}\biggr) - g(\xi) .
}
Setting $\xi = 1 - x$, where $x \in (0, 1)$, we recover the function $H_+(\lambda; x)$:
\formula{
 H_+(\lambda; x) & = \frac{h(1 - x)}{(\overline{\gamma} - x)(\gamma - x)} = \frac{1}{2 i \im \gamma} \biggl(\frac{h(\zeta)}{\overline{\gamma} - x} - \frac{h(\overline{\zeta})}{\gamma - x}\biggr) - g(1 - x) .
}
Since $h(\overline{\zeta}) = \overline{h(\zeta)}$ and $\zeta = 1 - \overline{\gamma}$, we obtain
\formula{
 H_+(\lambda; x) & = \frac{1}{\im \gamma} \, \im \biggl(-\frac{h(\overline{\zeta})}{\gamma - x}\biggr) - g(1 - x) \\
 & = \frac{1}{\im \gamma} \, \im \biggl(-\frac{1 - \gamma}{F^+(\lambda; \gamma)} \, \frac{1}{\gamma - x}\biggr) - g(1 - x) .
}
If $x \in (0, \lv \gamma \rv) \cap (0, 1)$, then $1 / (\gamma - x)$ is the generating function of a geometric series:
\formula{
 H_+(\lambda; x) & = \frac{1}{\im \gamma} \sum_{k = 0}^\infty \im \biggl(-\frac{1 - \gamma}{F^+(\lambda; \gamma)} \, \frac{x^k}{\gamma^{k + 1}}\biggr) - g(1 - x) .
}
Finally, by Proposition~\ref{prop:cm:gf}, $(\lv 1 - \gamma \rv^{-1} \lv F^+(\lambda; \gamma) \rv \im \gamma) g(1 - x)$ is the generating function of a completely monotone sequence $\evr^+_{\lambda, k}$. Hence,
\formula{
 H_+(\lambda; x) & = \frac{1}{\im \gamma} \sum_{k = 0}^\infty \im \biggl(-\frac{1 - \gamma}{F^+(\lambda; \gamma)} \, \frac{x^k}{\gamma^{k + 1}}\biggr) - \frac{\lv 1 - \gamma \rv}{\lv F^+(\lambda; \gamma) \rv \im \gamma} \sum_{k = 0}^\infty \evr^+_{\lambda, k} x^k .
}
Denoting $\thet = \arg \gamma$ and
\formula{
 \thet^+ & = -\arg \frac{1 - \gamma}{F^+(\lambda; \gamma)}
}
(both $\thet$ and $\thet^+$ depend on $\lambda$), we obtain
\formula*[eq:main:hp]{
 H_+(\lambda; x) & = \frac{\lv 1 - \gamma \rv}{\lv F^+(\lambda; \gamma) \rv \im \gamma} \sum_{k = 0}^\infty \Bigl(\lv \gamma \rv^{-k - 1} \sin\bigl((k + 1) \thet + \thet^+\bigr) - \evr^+_{\lambda, k}\Bigr) x^k \\
 & = \frac{\lv 1 - \gamma \rv}{\lv F^+(\lambda; \gamma) \rv \im \gamma} \sum_{k = 0}^\infty \ev^+_{\lambda, k} x^k .
}
This identity is valid for $x \in (0, \lv \gamma \rv) \cap (0, 1)$.

\emph{Step 2.}
Let $\lambda \in \Lambda$ still be fixed, and $\gamma = \gamma(\lambda)$. An argument completely analogous to the one used in Step~1, but applied to $H_-(\lambda; y)$ instead of $H_+(\lambda; x)$, and with $\gamma$ replaced by $\overline{\gamma}{}^{-1}$, leads to the formula
\formula{
 H_-(\lambda; y) & = \frac{\lv 1 - \overline{\gamma}{}^{-1} \rv}{\lv F^-(\lambda; \overline{\gamma}{}^{-1}) \rv \im \overline{\gamma}{}^{-1}} \sum_{l = 0}^\infty \Bigl(\lv\overline{\gamma}\rv^{l + 1} \sin\bigl((l + 1) \thet + \thet^-\bigr) - \evr^-_{\lambda, l}\Bigr) y^l .
}
Here $\evr^-_{\lambda, l}$ is a completely monotone sequence, $\thet = \arg \overline{\gamma}{}^{-1} = \arg \gamma$ is the same number as in Step~1, and
\formula{
 \thet^- & = -\arg \frac{1 - \overline{\gamma}{}^{-1}}{F^-(\lambda; \overline{\gamma}{}^{-1})} = \arg \frac{1 - \gamma^{-1}}{F^-(\lambda; \gamma^{-1})} \, .
}
The above identity holds for $y \in (0, \lv \overline{\gamma}{}^{-1} \rv) \cap (0, 1) = (0, \lv \gamma \rv^{-1}) \cap (0, 1)$.

After elementary simplification, we obtain
\formula*[eq:main:hm]{
 H_-(\lambda; y) & = \frac{\lv \gamma \rv \lv 1 - \gamma \rv}{\lv F^-(\lambda; \gamma^{-1}) \rv \im \gamma} \sum_{l = 0}^\infty \Bigl(\lv \gamma \rv^{l + 1} \sin\bigl((l + 1) \thet + \thet^-\bigr) - \evr^-_{\lambda, l}\Bigr) y^l \\
 & = \frac{\lv \gamma \rv \lv 1 - \gamma \rv}{\lv F^-(\lambda; \gamma^{-1}) \rv \im \gamma} \sum_{l = 0}^\infty \ev^-_{\lambda, l} y^l ,
}
where $y \in (0, \lv \gamma \rv^{-1}) \cap (0, 1)$.

\emph{Step 3.}
Let $\delta \in [0, 1]$ be the largest number such that $\delta \le \lv \gamma(\lambda) \rv \le \delta^{-1}$ for every $\lambda \in \Lambda$. By assumption, $\delta > 0$. Note that if $x \in (0, \delta)$, then $x \in (0, 1)$, $x \in (0, \lv \gamma(\lambda) \rv)$ for every $\lambda \in \Lambda$, and $x \in D_F^+$. Similarly, if $y \in (0, \delta)$, then $y \in (0, 1)$, $y \in (0, \lv \gamma(\lambda) \rv^{-1})$ for every $\lambda \in \Lambda$, and $y^{-1} \in D_F^-$. This allows us to insert the expressions given by~\eqref{eq:main:hp} and~\eqref{eq:main:hm} into~\eqref{eq:main:bi}.

It follows that if $x, y \in (0, \delta)$, then
\formula{
 \sum_{k = 0}^\infty \sum_{l = 0}^\infty x^k y^l T^n_{k, l} & = \frac{1}{\pi} \int_{\Lambda} (\mass - \lambda)^n I(\lambda) \biggl(\sum_{k = 0}^\infty \ev^+_{\lambda, k} x^k\biggr) \biggl(\sum_{l = 0}^\infty \ev^-_{\lambda, l} y^l\biggr) d\lambda ,
}
where
\formula{
 I(\lambda) & = \biggl(\frac{\lv 1 - \gamma(\lambda) \rv}{\lv F^+(\lambda; \gamma(\lambda)) \rv \im \gamma(\lambda)}\biggr) \biggl(\frac{\lv \gamma(\lambda) \rv \lv 1 - \gamma(\lambda) \rv}{\lv F^-(\lambda; (\gamma(\lambda))^{-1}) \rv \im \gamma(\lambda)}\biggr) \frac{\im \gamma(\lambda)}{\lv \gamma(\lambda) \rv} \\
 & = \frac{\lv 1 - \gamma(\lambda) \rv^2}{\lv F^+(\lambda; \gamma(\lambda)) F^-(\lambda; (\gamma(\lambda))^{-1}) \rv \im \gamma(\lambda)} \, .
}
Recall that $F^+(\lambda; x)$ and $F^-(\lambda; y)$ are the Wiener--Hopf factors of the function $F(\lambda; z)$, introduced in Definition~\ref{def:fr}. By~\eqref{eq:wh:fr}, we have
\formula{
 F^+(\lambda; \gamma(\lambda)) F^-(\lambda; (\gamma(\lambda))^{-1}) & = F(\lambda; \gamma(\lambda)) ,
}
and by Definition~\ref{def:fr} and continuity,
\formula{
 F(\lambda; \gamma(\lambda)) & = \frac{\lv \gamma(\lambda) \rv (1 - \gamma(\lambda))^2 F'(\gamma(\lambda))}{2 i \im \gamma(\lambda)} \, .
}
Therefore,
\formula{
 I(\lambda) & = \frac{2}{\lv \gamma(\lambda) \rv \lv F'(\gamma(\lambda)) \rv} = \frac{2 \lv \gamma'(\lambda) \rv}{\lv \gamma(\lambda) \rv} \, ,
}
the last equality being a consequence of
\formula{
 F'(\gamma(\lambda)) \gamma'(\lambda) & = (F(\gamma(\lambda)))' = (\lambda)' = 1 .
}
We conclude that, with the present notation, the assertion of Proposition~\ref{prop:bi} reads
\formula[eq:main:fubini]{
 \sum_{k = 0}^\infty \sum_{l = 0}^\infty x^k y^l T^n_{k, l} & = \frac{2}{\pi} \int_{\Lambda} (\mass - \lambda)^n \biggl(\sum_{k = 0}^\infty \ev^+_{\lambda, k} x^k\biggr) \biggl(\sum_{l = 0}^\infty \ev^-_{\lambda, l} y^l\biggr) \, \frac{\lv \gamma'(\lambda) \rv}{\lv \gamma(\lambda) \rv} \, d\lambda .
}

\emph{Step 4.}
In order to apply Fubini's theorem, we need upper bounds on $\ev^+_{\lambda, k}$ and $\ev^-_{\lambda, l}$. We claim that
\formula{
 \lv \ev^+_{\lambda, k} \rv & \le 1 + \lv \gamma(\lambda) \rv^{-k - 1} , & \lv \ev^-_{\lambda, l} \rv & \le 1 + \lv \gamma(\lambda) \rv^{l + 1} .
}
In order to prove this claim, we pass to the limit $x \to 0^+$ in~\eqref{eq:main:hp} to obtain
\formula{
 H_+(\lambda; 0^+) & = \frac{\lv 1 - \gamma(\lambda) \rv}{\lv F^+(\lambda; \gamma(\lambda)) \rv \im \gamma(\lambda)} \, \bigl(\sin(\thet_\lambda + \thet^+_\lambda) - \evr^+_{\lambda, 0} \bigr) .
}
By definition~\eqref{eq:main:hpm:p}, we have $H_+(\lambda; 0^+) > 0$, and so $\evr^+_{\lambda, 0} < \sin(\thet_\lambda + \thet^+_\lambda)$. Since the sequence $(\evr^+_{\lambda, k})$ is completely monotone, it is nonnegative and nonincreasing, and it follows that
\formula[eq:psi:est:p]{
 0 & \le \evr^+_{\lambda, k} < \sin(\thet_\lambda + \thet^+_\lambda) .
}
Therefore,
\formula{
 \lv \ev^+_{\lambda, k} \rv & \le \lv \gamma(\lambda) \rv^{-k - 1} \bigl| \sin\bigl((k + 1) \thet_\lambda + \thet^+_\lambda\bigr)\bigr| + \evr^+_{\lambda, k} \le \lv \gamma(\lambda) \rv^{-k - 1} + 1 ,
}
as desired. The proof of the other part of our claim is completely analogous, and we additionally have
\formula[eq:psi:est:m]{
 0 & \le \evr^-_{\lambda, l} < \sin(\thet_\lambda + \thet^-_\lambda) .
}

\emph{Step 5.}
By the previous step, if $x \in (0, \delta)$ and $\lambda \in \Lambda$, then
\formula{
 \lv \ev^+_{\lambda, k} x^k \rv & \le x^k + x^k \lv \gamma(\lambda) \rv^{-k - 1} \le x^k + \delta^{-1} (\delta^{-1} x)^k .
}
Similarly, for $y \in (0, \delta)$ and $\lambda \in \Lambda$, we have
\formula{
 \lv \ev^-_{\lambda, l} y^l \rv & \le y^l + y^l \lv \gamma(\lambda) \rv^{l + 1} \le y^k + \delta^{-1} (\delta^{-1} y)^l .
}
Finally, by Proposition~\ref{prop:spine}\ref{it:spine:e}, we have $\lv \mass - \lambda \rv \le \mass$ for $\lambda \in \Lambda$, and by Proposition~\ref{prop:spine}\ref{it:spine:d}, the total length of $\Gamma$ is finite. Thus,
\formula{
 \int_{\Lambda} \lv \mass - \lambda \rv^n \, \frac{\lv \gamma'(\lambda) \rv}{\lv \gamma(\lambda) \rv} \, d\lambda & \le \delta^{-1} \mass^n \int_{\Lambda} \lv \gamma'(\lambda) \rv d\lambda < \infty .
}
It follows that the integral of the double series on the right-hand side of~\eqref{eq:main:fubini} is absolutely integrable if $x, y \in (0, \delta)$:
\formula{
 & \int_{\Lambda} \lv\mass - \lambda\rv^n \biggl(\sum_{k = 0}^\infty \lv \ev^+_{\lambda, k} \rv x^k\biggr) \biggl(\sum_{l = 0}^\infty \lv \ev^-_{\lambda, l} \rv y^l\biggr) \, \frac{\lv \gamma'(\lambda) \rv}{\lv \gamma(\lambda) \rv} \, d\lambda \\
 & \qquad \le \biggl(\sum_{k = 0}^\infty (x^k + \delta^{-1} (\delta^{-1} x)^k)\biggr) \biggl(\sum_{l = 0}^\infty (y^l + \delta^{-1} (\delta^{-1} y)^l)\biggr) \int_{\Lambda} \lv\mass - \lambda\rv^n \, \frac{\lv \gamma'(\lambda) \rv}{\lv \gamma(\lambda) \rv} \, d\lambda < \infty .
}
By Fubini's theorem, \eqref{eq:main:fubini} can be written as
\formula{
 \sum_{k = 0}^\infty \sum_{l = 0}^\infty x^k y^l T^n_{k, l} & = \sum_{k = 0}^\infty \sum_{l = 0}^\infty \biggl(\frac{2}{\pi} \int_{\Lambda} (\mass - \lambda)^n \ev^+_{\lambda, k} \ev^-_{\lambda, l} \, \frac{\lv \gamma'(\lambda) \rv}{\lv \gamma(\lambda) \rv} \, d\lambda\biggr) x^k y^l ,
}
where $x, y \in (0, \delta)$ are arbitrary. The desired result~\eqref{eq:main} follows by the uniqueness of the generating function.

The expressions for the generating functions of the sequences $(\ev^+_{\lambda, k})$ and $(\ev^-_{\lambda, l})$ follow directly from~\eqref{eq:main:hp} and~\eqref{eq:main:hm} combined with~\eqref{eq:main:hpm:p} and~\eqref{eq:main:hpm:m} when $x \in (0, 1) \cap (0, \lv \gamma(\lambda) \rv)$ and $y \in (0, 1) \cap (0, \lv \gamma(\lambda) \rv^{-1})$. By uniqueness of the holomorphic extension, the formula is valid also for complex $x, y$.
\end{proof}


\subsection{Generalised eigenvector expansion}
\label{sec:main}

The proof of Theorem~\ref{thm:main} is now almost immediate.

\begin{proof}[Proof of Theorem~\ref{thm:main}]
Recall that the definition of the symbol $F$ of an $\amcm$ sequence $(a_k)$ does not depend on the value of $a_0$. Therefore, we may rephrase Proposition~\ref{prop:main} as follows: for an arbitrary $\amcm$ sequence $(a_k)$ such that its spine $\Gamma$ winds around $0$, and the corresponding infinite Toeplitz matrix $T = (a_{k - l} : k, l \ge 0)$, we have
\formula{
 (T - a_0 I)^n_{k, l} & = \frac{2}{\pi} \int_{\Lambda} (\mass - \lambda)^n \ev^+_{\lambda, k} \ev^-_{\lambda, l} \, \frac{\lv \gamma'(\lambda) \rv}{\lv \gamma(\lambda) \rv} \, d\lambda ,
}
where $\ev^+_{\lambda, k}$ and $\ev^-_{\lambda, k}$ are defined by the same formulae as in Proposition~\ref{prop:main}. By linearity, for any polynomial $P$ we have
\formula[eq:general]{
 (P(T))_{k, l} & = \frac{2}{\pi} \int_{\Lambda} P(a_0 + \mass - \lambda) \ev^+_{\lambda, k} \ev^-_{\lambda, l} \, \frac{\lv \gamma'(\lambda) \rv}{\lv \gamma(\lambda) \rv} \, d\lambda .
}
In Theorem~\ref{thm:main}, we choose $a_0 = -\sum_{k \in \Z \setminus \{0\}} a_k = -\mass$, in which case~\eqref{eq:general} reduces to~\eqref{eq:main}.

It remains to observe that with the notation introduced in the introduction:
\formula{
 \thet_\lambda & = \arg \gamma(\lambda) , \\
 \thet^+_\lambda & = -\arg \frac{1 - \gamma(\lambda)}{F^+(\lambda; \gamma(\lambda))} , \\
 \thet^-_\lambda & = \arg \frac{1 - (\gamma(\lambda))^{-1}}{F^-(\lambda; (\gamma(\lambda))^{-1})} , 
}
the definitions of $(\ev^+_{\lambda, k})$ and $(\ev^-_{\lambda, l})$ given in~\eqref{eq:ev:p} and~\eqref{eq:ev:m} agree with those in Proposition~\ref{prop:main}.
\end{proof}

Below we verify various properties of $\thet^+_\lambda$, $\thet^-_\lambda$, $(\evr^+_{\lambda, k})$ and $(\evr^-_{\lambda, l})$ listed in the introduction. We fix $\lambda \in \Lambda$, and, for brevity, we write $\gamma$ instead of $\gamma(\lambda)$. The argument is divided into six steps.

\emph{Step 1.}
Recall that by~\eqref{eq:fr:cbf:p}, $\xi (F^+(\lambda; 1 - \xi))^{-1}$ is a complete Bernstein function of $\xi$. If $\lambda \in \Lambda$ and $\xi = 1 - \gamma$, then $\im \xi < 0$, and so, by~\eqref{eq:cbf:arg},
\formula{
 0 & \ge \arg \frac{1 - \gamma}{F^+(\lambda; \gamma)} \ge \arg(1 - \gamma) = \arg(\gamma - 1) - \pi .
}
Therefore,
\formula{
 \thet^+_\lambda & \in [0, \pi - \arg(\gamma - 1)] .
}
Similarly, by~\eqref{eq:fr:cbf:m}, $\xi (F^-(\lambda; 1 - \xi))^{-1}$ is a complete Bernstein function of $\xi$, $\im(1 - \gamma^{-1}) > 0$, and so
\formula{
 0 & \le \arg \frac{1 - \gamma^{-1}}{F^-(\lambda; \gamma^{-1})} \le \arg(1 - \gamma^{-1}) = \arg(\gamma - 1) - \arg \gamma .
}
Thus,
\formula{
 \thet^-_\lambda & \in [0, \arg(\gamma - 1) - \thet_\lambda] .
}
In particular, $\thet_\lambda + \thet^+_\lambda + \thet^-_\lambda \le \pi$.

\emph{Step 2.}
The generating function of the sequence
\formula{
 \ev^+_{\lambda, k} + \evr^+_{\lambda, k} & = \lv \gamma \rv^{-k - 1} \sin((k + 1) \thet_\lambda + \thet^+_\lambda) = \frac{1}{2 i} \biggl(\frac{e^{i \thet^+_\lambda}}{\overline{\gamma}{}^{k + 1}} - \frac{e^{-i \thet^+_\lambda}}{\gamma^{k + 1}}\biggr)
}
extends to a rational function of $x$:
\formula{
 \sum_{k = 0}^\infty (\ev^+_{\lambda, k} + \evr^+_{\lambda, k}) x^k & = \frac{1}{2 i} \sum_{k = 0}^\infty \biggl(\frac{e^{i \thet^+_\lambda}}{\overline{\gamma}} \biggl(\frac{x}{\overline{\gamma}}\biggr)^k - \frac{e^{-i \thet^+_\lambda}}{\gamma} \biggl(\frac{x}{\gamma}\biggr)^k\biggr) \\
 & = \frac{1}{2 i} \biggl(\frac{e^{i \thet^+_\lambda}}{\overline{\gamma} - x} - \frac{e^{-i \thet^+_\lambda}}{\gamma - x}\biggr) .
}
On the other hand, by definition (see Proposition~\ref{prop:main}), the generating function of the sequence $(\ev^+_{\lambda, k})$ extends to a meromorphic function in the unit disk:
\formula{
 \sum_{k = 0}^\infty \ev^+_{\lambda, k} x^k & = \frac{\lv F^+(\lambda; \gamma) \rv \im \gamma}{\lv 1 - \gamma \rv} \, \frac{1 - x}{(\gamma - x) (\overline{\gamma} - x) F^+(\lambda; x)} \, .
}
The series on the left-hand sides of the above formulae converge when $\lv x \rv < \min\{\lv \gamma \rv, 1\}$. However, the generating function of $(\evr^+_{\lambda, k})$ converges in the unit disk, and it follows that for $x \in \disk$ we have
\formula{
 \sum_{k = 0}^\infty \evr^+_{\lambda, k} x^k & = \frac{1}{2 i} \biggl(\frac{e^{i \thet^+_\lambda}}{\overline{\gamma} - x} - \frac{e^{-i \thet^+_\lambda}}{\gamma - x}\biggr) - \frac{\lv F^+(\lambda; \gamma) \rv \im \gamma}{\lv 1 - \gamma \rv} \, \frac{1 - x}{(\gamma - x) (\overline{\gamma} - x) F^+(\lambda; x)} \, .
}

\emph{Step 3.}
We pass to the limit as $x \to 1^-$:
\formula{
 \sum_{k = 0}^\infty \evr^+_{\lambda, k} & = \frac{1}{2 i} \biggl(\frac{e^{i \thet^+_\lambda}}{\overline{\gamma} - 1} - \frac{e^{-i \thet^+_\lambda}}{\gamma - 1}\biggr) - \frac{\lv F^+(\lambda; \gamma) \rv \im \gamma}{\lv 1 - \gamma \rv^3} \, \lim_{x \to 1^-} \frac{1 - x}{F^+(\lambda; x)} \, .
}
By~\eqref{eq:fr:cbf:p}, $\xi / F^+(\lambda; 1 - \xi)$ is a complete Bernstein function of $\xi$, and so the limit on the right-hand side is nonnegative. Thus,
\formula{
 \sum_{k = 0}^\infty \evr^+_{\lambda, k} & \le \frac{1}{2 i} \biggl(\frac{e^{i \thet^+_\lambda}}{\overline{\gamma} - 1} - \frac{e^{-i \thet^+_\lambda}}{\gamma - 1}\biggr) = \im \frac{e^{i \thet^+_\lambda}}{\overline{\gamma} - 1} \, .
}
An analogous argument shows that
\formula{
 \sum_{l = 0}^\infty \evr^-_{\lambda, l} & \le \im \frac{e^{i \thet^-_\lambda}}{\gamma^{-1} - 1} \, .
}
We have already seen in~\eqref{eq:psi:est:p} and~\eqref{eq:psi:est:m} that
\formula{
 0 & \le \evr^+_{\lambda, k} < \sin(\thet_\lambda + \thet^+_\lambda) , \\
 0 & \le \evr^-_{\lambda, l} < \sin(\thet_\lambda + \thet^-_\lambda)
}
for every $k, l \ge 0$.

\emph{Step 4.}
Denote the holomorphic extension of the generating function of $(\evr^+_{\lambda, k})$ by $\hat \evr^+_\lambda$. Since $(\evr^+_{\lambda, k})$ is summable and completely monotone, for some measure $\mu^+_\lambda$ on $[0, 1)$ we have
\formula[eq:psi:cm]{
 \evr^+_{\lambda, k} & = \int_{[0, 1)} s^k \mu^+_\lambda(ds) ,
}
and, by Fubini's theorem,
\formula{
 \hat \evr^+_\lambda(x) & = \sum_{k = 0}^\infty \evr^+_{\lambda, k} x^k = \int_{[0, 1)} \frac{1}{1 - x s} \, \mu^+_\lambda(ds) .
}
It follows that
\formula{
 x^{-1} \hat \evr^+_\lambda(x^{-1}) & = \int_{[0, 1)} \frac{1}{x - s} \, \mu^+_\lambda(ds)
}
is the Stieltjes transform of the measure $\mu^+_\lambda$. Thus, in the sense of vague convergence of measures on $(0, 1)$,
\formula[eq:psi:stieltjes]{
 \mu^+_\lambda(ds) & = \lim_{t \to 0^+} \frac{\im ((s - i t)^{-1} \hat \evr^+_\lambda((s - i t)^{-1}))}{\pi} \, \ind_{(0, 1)}(s) ds .
}
Additionally, by the dominated convergence theorem,
\formula{
 \mu^+_\lambda(\{0\}) & = \lim_{x \to -\infty} \hat \evr^+_\lambda(x) .
}

\emph{Step 5.}
From the formula for $\hat \evr^+_\lambda(x)$ found in Step~2 it follows that for $x \in (-\infty, 1)$,
\formula{
 \lv \hat \evr^+_\lambda(x) \rv & = \biggl| \frac{1}{2 i} \biggl(\frac{e^{i \thet^+_\lambda}}{\overline{\gamma} - x} - \frac{e^{-i \thet^+_\lambda}}{\gamma - x}\biggr) - \frac{\lv F^+(\lambda; \gamma) \rv \im \gamma}{\lv 1 - \gamma \rv} \, \frac{1 - x}{(\gamma - x) (\overline{\gamma} - x) F^+(\lambda; x)} \biggr| \\
 & \le \frac{1}{\lv \gamma - x \rv} + \frac{\lv F^+(\lambda; \gamma) \rv \im \gamma}{\lv 1 - \gamma \rv} \, \frac{\lv 1 - x \rv}{\lv \gamma - x \rv^2 F^+(\lambda; x)} \, .
}
By~\eqref{eq:fr:cbf:p}, $F^+(\lambda; 1 - \xi)$ is a complete Bernstein function of $\xi$, and so $F^+(\lambda; x)$ is positive and decreasing on $(-\infty, 1)$. In particular,
\formula{
 \mu^+_\lambda(\{0\}) & = \lim_{x \to -\infty} \hat \evr^+_\lambda(x) = 0 .
}
We rewrite now rewrite the formula for $\hat \evr^+_\lambda(x)$ using~\eqref{eq:wh:fr}:
\formula{
 \hat \evr^+_\lambda(x) & = \frac{1}{2 i} \biggl(\frac{e^{i \thet^+_\lambda}}{\overline{\gamma} - x} - \frac{e^{-i \thet^+_\lambda}}{\gamma - x}\biggr) - \frac{\lv F^+(\lambda; \gamma) \rv \im \gamma}{\lv 1 - \gamma \rv} \, \frac{(1 - x) F^-(\lambda; x^{-1})}{(\gamma - x) (\overline{\gamma} - x) F(\lambda; x)} \, .
}
If the boundary limit
\formula{
 F(\lambda; s^{-1} + 0 i) & = \lim_{t \to 0^+} F(\lambda; s^{-1} + i t)
}
exists uniformly with respect to $s$ in any compact subinterval of $(0, 1)$, and $F(\lambda; s^{-1} + 0 i) \ne 0$ for all $s \in (0, 1)$, then also the boundary limit $\hat \evr^+_\lambda(s^{-1} + 0 i)$ exists uniformly with respect to $s$ in any compact subinterval of $(0, 1)$, and
\formula{
 \hat \evr^+_\lambda(s^{-1} + i 0) & = \frac{1}{2 i} \biggl(\frac{e^{i \thet^+_\lambda}}{\overline{\gamma} - s^{-1}} - \frac{e^{-i \thet^+_\lambda}}{\gamma - s^{-1}}\biggr) - \frac{\lv F^+(\lambda; \gamma) \rv \im \gamma}{\lv 1 - \gamma \rv} \, \frac{(1 - s^{-1}) F^-(\lambda; s)}{(\gamma - s^{-1}) (\overline{\gamma} - s^{-1}) F(\lambda; s^{-1} + 0 i)} \\
 & = \im \frac{e^{i \thet^+_\lambda}}{\overline{\gamma} - s^{-1}} + \frac{\lv F^+(\lambda; \gamma) \rv \im \gamma}{\lv 1 - \gamma \rv} \, \frac{(s^{-1} - 1) F^-(\lambda; s)}{\lv \gamma - s^{-1} \rv^2 F(\lambda; s^{-1} + 0 i)} \, .
}
In particular, the vague limit in~\eqref{eq:psi:stieltjes} coincides with the pointwise limit of density functions, that is,
\formula{
 \mu^+_\lambda(ds) & = \frac{\im (s^{-1} \hat \evr^+_\lambda(s^{-1} + 0 i))}{\pi} \, \ind_{(0, 1)}(s) ds .
}
Combining this with~\eqref{eq:psi:cm}, we conclude that
\formula{
 \evr^+_{\lambda, k} & = \frac{1}{\pi} \int_0^1 s^k \, \frac{\lv F^+(\lambda; \gamma) \rv \im \gamma}{\lv 1 - \gamma \rv} \, \frac{(1 - s) F^-(\lambda; s)}{\lv \gamma s - 1 \rv^2} \, \im \frac{1}{F(\lambda; s^{-1} + 0 i)} \, ds \\
 & = \frac{1}{\pi} \int_1^\infty \frac{1}{t^k} \, \frac{\lv F^+(\lambda; \gamma) \rv \im \gamma}{\lv 1 - \gamma \rv} \, \frac{(1 - t^{-1}) F^-(\lambda; t^{-1})}{\lv \gamma - t \rv^2} \, \im \frac{1}{F(\lambda; t + 0 i)} \, dt .
}
When the continuous boundary limit $F(\lambda; s^{-1} + 0 i)$ does not exist, the limiting measure $\mu^+_\lambda(ds)$ need not have a density function. Nevertheless, the above argument still works, and the final expression remains valid, as long as we understand the boundary limit in the vague sense, that is,
\formula{
 \im \frac{1}{F(\lambda; s^{-1} + 0 i)} \, ds & = \lim_{t \to 0^+} \im \frac{1}{F(\lambda; s^{-1} + i t)} \, ds .
}

\emph{Step 6.}
A completely analogous argument (or an application of the result of Step~5 to the reflected sequence $(\check a_k) = (a_{-k})$ and the dual symbol $\check F(z) = F(z^{-1})$) leads to the expression
\formula{
 \evr^-_{\lambda, l} & = \frac{1}{\pi}  \int_0^1 s^l \, \frac{\lv F^-(\lambda; \gamma^{-1}) \rv \im \overline{\gamma}{}^{-1}}{\lv 1 - \gamma^{-1} \rv} \, \frac{(1 - s) F^+(\lambda; s)}{\lv \gamma^{-1} s - 1 \rv^2} \, \im \frac{-1}{F(\lambda; s + 0 i)} \, ds \\
 & = \frac{1}{\pi}  \int_1^\infty \frac{1}{t^l} \frac{\lv F^-(\lambda; \gamma^{-1}) \rv \im \overline{\gamma}{}^{-1}}{\lv 1 - \gamma^{-1} \rv} \, \frac{(t - 1) F^+(\lambda; t^{-1})}{t \lv \gamma^{-1} - t \rv^2} \, \im \frac{-1}{F(\lambda; t^{-1} + 0 i)} \, dt ,
}
provided that the boundary limit $F(\lambda; s + 0 i)$ exists uniformly with respect to $s$ in any compact subinterval of $(0, 1)$, and $F(\lambda; s + 0 i) \ne 0$ for all $s \in (0, 1)$. Otherwise, the formula remains valid if the boundary limit is understood in the vague sense, as in the previous step.


\subsection{Eigenvectors}
\label{sec:eig}

Below we prove that under appropriate conditions, $(\ev^+_{\lambda, k})$ and $(\ev^-_{\lambda, l})$ are true eigenvectors and co-eigenvectors of the Toeplitz matrix $T = (a_{k - l} : k, l \ge 0)$. Recall that the generating function of $(a_k)$ is
\formula{
 \sum_{k = -\infty}^\infty a_k z^k & = \hat a(z) = \mass - F(z)
}
for $z \in \torus$, where $\mass = \hat a(1)$.

\begin{proof}[Proof of Theorem~\ref{thm:eig}]
We fix $\lambda \in \Lambda$, and to simplify the notation, in this proof we write $\gamma$ instead of $\gamma(\lambda)$. We also denote by $\hat \ev^+_\lambda$ the holomorphic extension of the generating function of $(\ev^+_{\lambda, k})$. By Proposition~\ref{prop:main},
\formula[eq:eig:phi]{
 \hat \ev^+_\lambda(z) & = \frac{\lv F^+(\lambda; \gamma) \rv \im \gamma}{\lv 1 - \gamma \rv} \, \frac{1 - z}{(\gamma - z) (\overline{\gamma} - z)} \, \frac{1}{F^+(\lambda; z)}
}
for $z \in \C \setminus ([1, \infty) \cup \{\gamma, \overline{\gamma}\}$. We divide the proof into four steps.

\emph{Step 1.}
Suppose first that $\lv \gamma \rv > 1$. In this case $(\ev^+_{\lambda, k})$ is a summable sequence. We extend it to a doubly infinite sequence by setting $\ev^+_{\lambda, k} = 0$ when $k < 0$. The generating function of the convolution of $(\ev^+_{\lambda, k})$ and $(a_k)$ is the product of the corresponding generating functions:
\formula{
 \sum_{k = -\infty}^\infty (T \ev^+_\lambda)_k z^k & = \sum_{k = -\infty}^\infty \biggl(\sum_{l = 0}^\infty a_{k - l} \ev^+_{\lambda, l}\biggr) z^k = (\mass - F(z)) \hat \ev^+_\lambda(z)
}
for $z \in \torus$, and therefore
\formula{
 \sum_{k = -\infty}^\infty \bigl((\mass - \lambda) \ev^+_{\lambda, k} - (T \ev^+_\lambda)_k\bigr) z^k & = (F(z) - \lambda) \hat \ev^+_\lambda(z)
}
By Cauchy's integral formula, for every $k \ge 0$,
\formula[eq:eig:eig]{
 (\mass - \lambda) \ev^+_{\lambda, k} - (T \ev^+_\lambda)_k & = \frac{1}{2 \pi i} \int_\torus (F(z) - \lambda) \hat \ev^+_\lambda(z) \, \frac{dz}{z^{k + 1}} \, ,
}
and we claim that the right-hand side is zero.

By Definition~\ref{def:fr},
\formula{
 F(z) - \lambda & = \frac{(\gamma - z) (\overline{\gamma} - z)}{\lv \gamma \rv (1 - z)^2} \, F(\lambda; z) .
}
Applying this and~\eqref{eq:eig:phi} to~\eqref{eq:eig:eig}, we obtain
\formula{
 (\mass - \lambda) \ev^+_{\lambda, k} - (T \ev^+_\lambda)_k & = \frac{1}{2 \pi i} \int_\torus \frac{\lv F^+(\lambda; \gamma) \rv \im \gamma}{\lv \gamma \rv \lv 1 - \gamma \rv} \, \frac{F(\lambda; z)}{(1 - z) F^+(\lambda; z)} \, \frac{dz}{z^{k + 1}} \, .
}
The Wiener--Hopf factorisation~\eqref{eq:wh:fr} yields
\formula{
 (\mass - \lambda) \ev^+_{\lambda, k} - (T \ev^+_\lambda)_k & = \frac{1}{2 \pi i} \int_\torus \frac{\lv F^+(\lambda; \gamma) \rv \im \gamma}{\lv \gamma \rv \lv 1 - \gamma \rv} \, \frac{F^-(\lambda; z^{-1})}{1 - z} \, \frac{dz}{z^{k + 1}} \, ,
}
which, after substitution $z = \tilde z^{-1}$, $dz = -\tilde z^{-2} d\tilde z$, translates to
\formula{
 (\mass - \lambda) \ev^+_{\lambda, k} - (T \ev^+_\lambda)_k & = -\frac{1}{2 \pi i} \int_\torus \frac{\lv F^+(\lambda; \gamma) \rv \im \gamma}{\lv \gamma \rv \lv 1 - \gamma \rv} \, \frac{F^-(\lambda; \tilde z)}{1 - \tilde z^{-1}} \, \tilde z^{k - 1} d\tilde z \\
 & = \frac{1}{2 \pi i} \int_\torus \frac{\lv F^+(\lambda; \gamma) \rv \im \gamma}{\lv \gamma \rv \lv 1 - \gamma \rv} \, \frac{F^-(\lambda; \tilde z)}{1 - \tilde z} \, \tilde z^k d\tilde z.
}
The integrand is a holomorphic function in $\disk$, and it extends continuously to $\overline \disk$ (see~\eqref{eq:fr:one}). Hence, by Cauchy's theorem, the integral is zero. Our claim is proved, and the desired result follows when $\lv \gamma \rv > 1$.

\emph{Step 2.}
If $\lv \gamma \rv = 1$, then $(\ev^+_{\lambda, k})$ is merely bounded, and the first part of the above argument requires appropriate modifications.

Observe that, by the dominated convergence theorem,
\formula{
 (T \ev^+_\lambda)_k & = \sum_{l = 0}^\infty a_{k - l} \ev^+_{\lambda, l} = \lim_{x \to 1^-} \sum_{l = 0}^\infty a_{k - l} \ev^+_{\lambda, l} x^l .
}
The sum one on the right-hand side is the convolution of two summable sequences, and so we can proceed as in the previous step. Since the generating function of the sequence $(\ev^+_{\lambda, k} x^k)$ is $\hat \ev^+_\lambda(x z)$, we have
\formula{
 (\mass - \lambda) \ev^+_{\lambda, k} x^k - \sum_{l = 0}^\infty a_{k - l} \ev^+_{\lambda, l} x^l & = \frac{1}{2 \pi i} \int_\torus (F(z) - \lambda) \, \hat \ev^+_\lambda(x z) \, \frac{dz}{z^{k + 1}} \, .
}
It follows that
\formula{
 (\mass - \lambda) \ev^+_{\lambda, k} - (T \ev^+_\lambda)_k & = \lim_{x \to 1^-} \frac{1}{2 \pi i} \int_\torus (F(z) - \lambda) \, \hat \ev^+_\lambda(x z) \, \frac{dz}{z^{k + 1}} .
}
In the proof of Proposition~\ref{prop:main} we observed that $\hat \ev^+_\lambda$ is analytic near $\torus \setminus \{1, \gamma, \overline{\gamma}\}$ and it has simple poles at $\gamma$ and $\overline{\gamma}$, while in the proof of Theorem~\ref{thm:main} we also noted that $\hat \ev^+_\lambda$ has a finite nontangential limit at $1$ in $\C \setminus [1, \infty)$. Furthermore, if $z \in \torus$ and $x \in (0, 1)$, then, by~\eqref{eq:disk}, $\lv \gamma - x z \rv = \lv \gamma \overline z - x \rv \ge \tfrac{1}{2} \lv \gamma \overline z - 1 \rv = \tfrac{1}{2} \lv \gamma - z \rv$, and, similarly, $\lv \overline{\gamma} - x z \rv \ge \tfrac{1}{2} \lv \overline{\gamma} - z \rv$. Therefore, there are constants $c_1, c_2 > 0$ such that
\formula{
 \lv (F(z) - \lambda) \, \hat \ev^+_\lambda(x z) \rv & \le c_1 \, \frac{\lv F(z) - \lambda \rv}{\lv \gamma - x z \rv \lv \overline{\gamma} - x z \rv} \le 4 c_1 \, \frac{\lv F(z) - \lambda \rv}{\lv \gamma - z \rv \lv \overline{\gamma} - z \rv} \le c_2
}
for $z \in \torus$ and $x \in (0, 1)$ (the last inequality is a consequence of the fact that $F(z) - \lambda$ has zeroes at $z = \gamma$ and $z = \overline{\gamma}$). By the dominated convergence theorem,
\formula{
 (\mass - \lambda) \ev^+_{\lambda, k} - (T \ev^+_\lambda)_k & = \frac{1}{2 \pi i} \int_\torus (F(z) - \lambda) \, \hat \ev^+_\lambda(z) \, \frac{dz}{z^{k + 1}} .
}
We have thus arrived at formula~\eqref{eq:eig:eig}, and the remaining part of the proof is exactly the same as in Step~1.

\emph{Step 3.}
If $\lv \gamma \rv < 1$ and $\sum_{l = 0}^\infty a_{-l} \lv \gamma \rv^{-l}$ is finite, then $\hat a$ and $F$ extend to holomorphic functions in the region $\lv \gamma \rv < \lv z \rv < 1$, continuous up to the boundary. We have
\formula{
 (T \ev^+_\lambda)_k & = \sum_{l = 0}^\infty a_{k - l} \ev^+_{\lambda, l} = \lim_{x \to 1^-} \lv \gamma \rv^{-k} \sum_{l = 0}^\infty (a_{k - l} \lv \gamma \rv^{k - l}) (\ev^+_{\lambda, l} \lv \gamma \rv^l x^l) ,
}
and again the sum on the right-hand side is the convolution of two summable sequences. Arguing as in the previous step, we find that
\formula{
 (\mass - \lambda) \ev^+_{\lambda, k} - (T \ev^+_\lambda)_k & = \lim_{x \to 1^-} \frac{1}{2 \pi i \lv \gamma \rv^k} \int_\torus (F(\lv \gamma \rv z) - \lambda) \, \hat \ev^+_\lambda(x z / \lv \gamma \rv) \, \frac{dz}{z^{k + 1}} ,
}
and the remaining part of the proof is very similar to Steps~1 and~2; we omit the details.

\emph{Step 4.}
Analogous arguments (or the results proved above applied to the reversed sequence $(\check a_k) = (a_{-k})$ and the dual symbol $\check F(z) = F(z^{-1})$) lead to the desired properties of $(\ev^-_{\lambda, l})$; once again we omit the details.
\end{proof}

%
%

\section{Examples}
\label{sec:ex}

Below we give explicit expressions for the generalised eigenvectors and co-eigenvectors for a variety of examples. The first few of them are $\amcm$ sequences with simple rational generating functions, where the evaluation of the generalised eigenvectors and co-eigenvectors could be done in a more straightforward fashion. These examples are meant to illustrate the validity of Theorems~\ref{thm:main} and~\ref{thm:eig}. Next, we discuss more general $\amcm$ sequences with rational generating functions. Finally, we give a detailed discussion of a class of $\amcm$ sequences which correspond to generating functions with a singularity of Fisher--Hartwig type.


\subsection{Tridiagonal matrices}

Consider a tridiagonal infinite Toeplitz matrix $T$ corresponding to the $\amcm$ sequence
\formula{
 (a_k) & = (\ldots, 0, 0, q, -p - q, p, 0, 0, \ldots) ,
}
where $p, q > 0$. The symbol of $(a_k)$ is
\formula{
 F(z) & = -\hat a(z) = p (1 - z) + q (1 - z^{-1}) .
}
Note that
\formula{
 \im F(z) & = \frac{(p \lv z \rv^2 - q) \im z}{\lv z \rv^2} \, .
}
Hence, the spine $\Gamma$ is the semicircle $\{z \in \C : \lv z \rv = r , \, \im z > 0\}$, where $r = (q / p)^{1 / 2}$. Clearly, $\Gamma$ winds around $0$. By a direct calculation,
\formula{
 F(r e^{i \thet}) & = p + q - 2 \sqrt{p q} \cos \thet ,
}
It follows that $\Lambda = (\lambda_1, \lambda_2)$, where $\lambda_1 = (\sqrt{p} - \sqrt{q})^2$ and $\lambda_2 = (\sqrt{p} + \sqrt{q})^2$, and
\formula{
 \gamma(\lambda) & = r e^{i \thet_\lambda}, && \text{where} & \thet_\lambda & = \arccos \biggl(\frac{p + q - \lambda}{2 \sqrt{p q}}\biggr) .
}
We fix $\lambda \in \Lambda$, and let us write $\gamma$ instead of $\gamma(\lambda)$ for simplicity. We evaluate the difference quotient $F(\lambda; z)$. Observe that
\formula{
 F(z) - \lambda & = F(z) - F(\gamma) \\
 & = p (\gamma - z) + q (\gamma^{-1} - z^{-1}) \\
 & = \frac{p (\gamma - z) (z \gamma - q / p)}{z \gamma} \\
 & = \frac{p (\gamma - z) (z \gamma - \lv \gamma \rv^2)}{z \gamma} \\
 & = \frac{p (\gamma - z) (z - \overline{\gamma})}{z} \, .
}
By~\eqref{eq:i:fr} (or Definition~\ref{def:fr}),
\formula{
 F(\lambda; z) & = \frac{\lv \gamma \rv (z - 1)^2}{(\gamma - z) (\overline{\gamma} - z)} \, (F(z) - \lambda) = \frac{p r (z - 1)^2}{z} = \sqrt{p q} \, (1 - z) (1 - z^{-1}) .
}
Noteworthy, $F(\lambda; z)$ does not depend on $\lambda$. Since $F(\lambda; z)$ is a rational function, its Wiener--Hopf factors are easily identified without the integral formulae~\eqref{eq:i:fr:wh:p} and~\eqref{eq:i:fr:wh:m} (or Proposition~\ref{prop:fr:wh:exp}): up to multiplication by a constant, we have
\formula{
 F^+(\lambda; z) & = \sqrt{p} \, (1 - z), & F^-(\lambda; z^{-1}) & = \sqrt{q} \, (1 - z^{-1}) .
}
In particular, by~\eqref{eq:i:theta},
\formula{
 \thet^+_\lambda & = -\arg \frac{1 - \gamma}{F^+(\lambda; \gamma)} = 0 , \\
 \thet^-_\lambda & = \arg \frac{1 - \gamma^{-1}}{F^-(\lambda; \gamma^{-1})} = 0 ,
}
and we have already given above an expression for $\thet_\lambda$. Since a continuous, real-valued and non-zero boundary limit $F(\lambda; s + 0 i) = \sqrt{p q} \, (1 - s) (1 - s^{-1})$ exists for $s \in (0, 1) \cup (1, \infty)$, by~\eqref{eq:i:psi:p} and~\eqref{eq:i:psi:m} we have
\formula{
 \evr^+_{\lambda, k} & = 0 , &
 \evr^-_{\lambda, l} & = 0
}
for every $k, l \ge 0$. Therefore,
\formula{
 \ev^+_{\lambda, k} & = r^{-k - 1} \sin((k + 1) \thet_\lambda) , \\
 \ev^-_{\lambda, l} & = r^{l + 1} \sin((l + 1) \thet_\lambda) .
}
Finally, $\lv \gamma'(\lambda) / \gamma(\lambda) \rv = \tfrac{d}{d\lambda} \thet_\lambda$. Theorem~\ref{thm:main} implies that for every polynomial $P$,
\formula{
 (P(T))_{k, l} & = \frac{2}{\pi} \int_{\lambda_1}^{\lambda_2} P(-\lambda) \ev^+_{\lambda, k} \ev^-_{\lambda, l} \, \frac{d \thet_\lambda}{d\lambda} \, d\lambda .
}
It is clearly convenient to substitute $\thet = \thet_\lambda$:
\formula[eq:tridiagonal]{
 (P(T))_{k, l} & = \frac{2 r^{l - k}}{\pi} \int_0^\pi P(2 \sqrt{p q} \cos \thet - p - q) \sin((k + 1) \thet) \sin((l + 1) \thet) d\thet .
}
Furthermore, by Theorem~\ref{thm:eig}, we have $T \ev^+_\lambda = -\lambda \ev^+_\lambda$ and $T^* \ev^-_\lambda = -\lambda \ev^-_\lambda$.

We mention the probabilistic interpretation of the above model. If $p + q = 1$, then the transpose of $T$ is the generator of an asymmetric simple random walk on nonnegative integers: $\pr[X_{n + 1} - X_n = 1] = p$ and $\pr[X_{n + 1} - X_n] = q$.

The above spectral properties of $T$ can be obtained in an elementary way, and are well-known. Indeed: identity~\eqref{eq:tridiagonal} follows directly from the application of the discrete sine transform to the symmetric Toeplitz matrix $(r^{k - l} a_{k - l} : k, l \ge 0)$.


\subsection{Geometric sequences}

Consider the infinite Toeplitz matrix $T$ corresponding to the $\amcm$ sequence $(a_k)$ given by two geometric sequences:
\formula{
 a_k & = \begin{cases}
  p (1 - \alpha) \alpha^{k - 1} & \text{if } k \ge 1, \\
  q (1 - \beta) \beta^{-k + 1} & \text{if } k \le -1, \\
  -p - q & \text{if } k = 0,
 \end{cases}
}
where $p, q \ge 0$ and $\alpha, \beta \in [0, 1)$. Note that the tri-diagonal matrix considered in the previous section is included as special case $\alpha = \beta = 0$, and other special cases are discussed later in later sections.

The symbol of $(a_k)$ is given by
\formula{
 F(z) = -\hat a(z) & = p + q - p z \, \frac{1 - \alpha}{1 - \alpha z} - q z^{-1} \, \frac{1 - \beta}{1 - \beta z^{-1}} \\
 & = p \, \frac{1 - z}{1 - \alpha z} + q \, \frac{1 - z}{\beta - z} \, .
}
By a direct calculation,
\formula{
 \im \frac{1 - z}{1 - \alpha z} & = \frac{-(1 - \alpha) \im z}{\lv 1 - \alpha z \rv^2} , &
 \im \frac{1 - z}{\beta - z} & = \frac{(1 - \beta) \im z}{\lv \beta - z \rv^2} ,
}
and so the spine $\Gamma$ is described by the equation
\formula{
 (1 - \beta) q \lv 1 - \alpha z \rv^2 & = (1 - \alpha) p \lv \beta - z \rv^2 .
}
It follows that if $(1 - \alpha) p = \alpha^2 (1 - \beta) q$, then $\Gamma$ is the half-line $\re z = \tfrac{1}{2} (\alpha^{-1} + \beta)$, $\im z > 0$. Otherwise, $\Gamma$ is a semicircle, and another direct calculation shows that its centre and radius are
\formula{
 z_0 & = -\frac{\alpha (1 - \beta) q - \beta (1 - \alpha) p}{(1 - \alpha) p - \alpha^2 (1 - \beta) q} , &
 r_0 & = \frac{\sqrt{(1 - \alpha) (1 - \beta) p q} \, (1 - \alpha \beta)}{\lv (1 - \alpha) p - \alpha^2 (1 - \beta) q \rv} \, .
}
By yet another direct calculation, one can verify that $\Gamma$ winds around zero if and only if $(1 - \alpha) p > \alpha^2 (1 - \beta) q$ and $(1 - \beta) q > \beta^2 (1 - \alpha) p$. Below we assume that these conditions are satisfied.

By a direct calculation, 
\formula[eq:geo:gamma]{
 F(z) - \lambda & = \frac{U(\lambda) - W(\lambda) z + V(\lambda) z^2}{(1 - \alpha z) (\beta - z)} \, ,
}
where
\formula{
 U(\lambda) & = q + \beta p - \beta \lambda , \\
 V(\lambda) & = p + \alpha q - \alpha \lambda , \\
 W(\lambda) & = (1 + \beta) p + (1 + \alpha) q - (1 + \alpha \beta) \lambda .
}
Equation~\eqref{eq:geo:gamma} has a solution $z = \gamma(\lambda)$ with positive imaginary part if and only if $(W(\lambda))^2 < 4 U(\lambda) W(\lambda)$. This condition simplifies to $\lambda \in \Lambda = (\lambda_1, \lambda_2)$, where
\formula[eq:geo:lambda]{
 \lambda_1 & = \frac{\bigl(\sqrt{(1 - \beta) p} - \sqrt{(1 - \alpha) q}\bigr)^2}{1 - \alpha \beta} \, , &
 \lambda_2 & = \frac{\bigl(\sqrt{(1 - \beta) p} + \sqrt{(1 - \alpha) q}\bigr)^2}{1 - \alpha \beta} \, , &
}
The solution is given by
\formula{
 \lv \gamma(\lambda) \rv & = \sqrt{\frac{U(\lambda)}{V(\lambda)}} \, , &
 \thet_\lambda & = \arg \gamma(\lambda) = \arccos \frac{W(\lambda)}{2 \sqrt{U(\lambda) V(\lambda)}} \, , &
 \re \gamma(\lambda) & = \frac{W(\lambda)}{2 V(\lambda)} \, .
}
We also find that
\formula{
 F(z) - \lambda & = \frac{V(\lambda) (\gamma(\lambda) - z) (\overline{\gamma}(\lambda) - z)}{(1 - \alpha z) (\beta - z)} \, .
}
Therefore,
\formula{
 F(\lambda; z) & = \frac{\lv \gamma(\lambda) \rv (z - 1)^2}{(\gamma(\lambda) - z) (\overline{\gamma}(\lambda) - z)} \, (F(z) - \lambda) \\
 & = \frac{\sqrt{U(\lambda) V(\lambda)} \, (z - 1)^2}{(1 - \alpha z) (\beta - z)} \\
 & = \frac{\sqrt{U(\lambda) V(\lambda)} \, (1 - z) (1 - z^{-1})}{(1 - \alpha z) (1 - \beta z^{-1})} \, .
}
Up to multiplication by a constant, the corresponding Wiener--Hopf factors are given by
\formula{
 F^+(\lambda; z) & = \frac{\sqrt{V(\lambda)} \, (1 - z)}{1 - \alpha z} \, , &
 F^-(\lambda; z^{-1}) & = \frac{\sqrt{U(\lambda)} \, (1 - z^{-1})}{1 - \beta z^{-1}} \, .
}
In particular,
\formula{
 \thet^+_\lambda & = -\arg \frac{1 - \gamma(\lambda)}{F^+(\lambda; \gamma(\lambda))} = -\arg (1 - \alpha \gamma(\lambda)) , \\
 \thet^-_\lambda & = \arg \frac{1 - (\gamma(\lambda))^{-1}}{F^+(\lambda; (\gamma(\lambda))^{-1})} = \arg (1 - \beta (\gamma(\lambda))^{-1}) .
}
We have
\formula{
 \lv 1 - \alpha \gamma(\lambda) \rv^2 & = 1 + \alpha^2 \lv \gamma(\lambda) \rv^2 - 2 \alpha \re \gamma(\lambda) = \frac{V(\lambda) + \alpha^2 U(\lambda) - \alpha W(\lambda)}{V(\lambda)} = \frac{(1 - \alpha \beta) p}{V(\lambda)} \, ,
}
and
\formula{
 \lv 1 - \beta (\gamma(\lambda))^{-1} \rv^2 & = \frac{\beta^2 + \lv \gamma(\lambda) \rv^2 - 2 \beta \re \gamma(\lambda)}{\lv \gamma(\lambda) \rv^2} = \frac{\beta^2 V(\lambda) + U(\lambda) - \beta W(\lambda)}{U(\lambda)} = \frac{(1 - \alpha \beta) q}{U(\lambda)} \, ,
}
and therefore
\formula{
 \cos \thet^+_\lambda & = \frac{1 - \alpha \re \gamma(\lambda)}{\lv 1 - \alpha \gamma(\lambda) \rv} = \frac{2 V(\lambda) - \alpha W(\lambda)}{2 \sqrt{(1 - \alpha) (1 - \alpha \beta) p V(\lambda)}} \, , \\
 \cos \thet^-_\lambda & = \frac{\lv \gamma(\lambda) \rv^2 - \beta \re \gamma(\lambda)}{\lv \gamma(\lambda) \rv^2 \lv 1 - \beta (\gamma(\lambda))^{-1} \rv} = \frac{2 U(\lambda) - \beta W(\lambda)}{2 \sqrt{(1 - \beta) (1 - \alpha \beta) q U(\lambda)}} \, .
}
Observe that for $s \in (0, 1) \cup (1, \infty)$, a continuous, real, nonzero boundary limit $1 / F(\lambda; s + 0 i)$ exists, so that $\evr^+_{\lambda, k} = \evr^-_{\lambda, l} = 0$. We conclude that
\formula{
 \ev^+_{\lambda, k} & = \biggl(\frac{V(\lambda)}{U(\lambda)}\biggr)^{(k + 1) / 2} \sin\biggl((k + 1) \arccos \frac{W(\lambda)}{2 \sqrt{U(\lambda) V(\lambda)}} + \arccos \frac{2 V(\lambda) - \alpha W(\lambda)}{2 \sqrt{(1 - \alpha) (1 - \alpha \beta) p V(\lambda)}}\biggr) , \\
 \ev^-_{\lambda, l} & = \biggl(\frac{U(\lambda)}{V(\lambda)}\biggr)^{(l + 1) / 2} \sin\biggl((l + 1) \arccos \frac{W(\lambda)}{2 \sqrt{U(\lambda) V(\lambda)}} + \arccos \frac{2 U(\lambda) - \beta W(\lambda)}{2 \sqrt{(1 - \beta) (1 - \alpha \beta) q U(\lambda)}}\biggr) .
}
Finally,
\formula{
 \biggl|\frac{\gamma'(\lambda)}{\gamma(\lambda)}\biggr|^2 & = \biggl(\frac{U'(\lambda)}{2 U(\lambda)} - \frac{V'(\lambda)}{2 V(\lambda)}\biggr)^2 + \biggl(\frac{2 W'(\lambda) - W(\lambda) (U'(\lambda) / U(\lambda) + V'(\lambda) / V(\lambda))}{2 \sqrt{4 U(\lambda) V(\lambda) - (W(\lambda))^2}}\biggr)^2 \\
 & = \frac{(1 - \alpha) (1 - \beta) (1 - \alpha \beta)^2 p q}{U(\lambda) V(\lambda) (4 U(\lambda) V(\lambda) - (W(\lambda))^2)} \, .
}
Theorem~\ref{thm:main} asserts that if $(1 - \alpha) p > \alpha^2 (1 - \beta) q$ and $(1 - \beta) q > \beta^2 (1 - \alpha) p$, then for every polynomial $P$,
\formula{
 (P(T))_{k, l} & = \frac{2}{\pi} \int_{\lambda_1}^{\lambda_2} P(-\lambda) \ev^+_{\lambda, k} \ev^-_{\lambda, l} \, \biggl|\frac{\gamma'(\lambda)}{\gamma(\lambda)}\biggr| \, d\lambda ,
}
where $\lambda_1, \lambda_2$ are given by~\eqref{eq:geo:lambda}.

By Theorem~\ref{thm:eig}, if $\beta^2 V(\lambda) \le U(\lambda)$, that is,
\formula{
 \lambda & \le \frac{\beta (1 - \beta) p + (1 - \alpha \beta^2) q}{\beta (1 - \alpha \beta)} \, ,
}
then $\ev^+_\lambda$ is indeed and eigenvector of $T$: we have $T \ev^+_\lambda = -\lambda \ev^+_\lambda$. Similarly, if $\alpha^2 U(\lambda) \le V(\lambda)$, which amounts to
\formula{
 \lambda & \le \frac{(1 - \alpha^2 \beta) p + \alpha (1 - \alpha) q}{\alpha (1 - \alpha \beta)} \, ,
}
then $T^* \ev^-_\lambda = -\lambda \ev^-_\lambda$.

In probabilistic terms, if $p + q = 1$, then the transpose of $T$ is the generator of a random walk with increments having two-sided geometric distribution.


\subsection{One-sided geometric sequences}

The complicated expressions appearing above simplify slightly when $\beta = 0$. The corresponding $\amcm$ sequence $(a_k)$ satisfies $a_k = 0$ for $k \le -2$, and so the infinite Toeplitz matrix $T$ is nearly triangular. We have
\formula{
 F(z) & = -\hat a(z) = p \, \frac{1 - z}{1 - \alpha z} + q (1 - z^{-1}) .
}
We assume that $p > \alpha^2 q$. In this case the spine $\Gamma$ is the semicircle with centre $z_0 = -\alpha q / (p - \alpha^2 q)$ and radius $r_0 = \sqrt{p q} / (p - \alpha^2 q)$. We have $U(\lambda) = q$ and
\formula{
 V(\lambda) & = p + \alpha q - \alpha \lambda , &
 W(\lambda) & = p + (1 + \alpha) q - \lambda .
}
Furthermore,
\formula{
 \lambda_1 & = p + (1 - \alpha) q - 2 \sqrt{(1 - \alpha) p q} , &
 \lambda_2 & = p + (1 - \alpha) q + 2 \sqrt{(1 - \alpha) p q} ,
}
The expressions for $\thet_\lambda$ and $\thet^+_\lambda$ do not simplify significantly, but we have $\thet^-_\lambda = 0$. Hence,
\formula{
 \ev^+_{\lambda, k} & = \biggl(\frac{p + \alpha q - \alpha \lambda}{q}\biggr)^{\frac{k + 1}{2}} \sin\biggl((k + 1) \arccos \frac{p + (1 + \alpha) q - \lambda}{2 \sqrt{q (p + \alpha q - \alpha \lambda)}} \\
 & \hspace*{12em} + \arccos \frac{(2 - \alpha) p + \alpha (1 - \alpha) q - \alpha \lambda}{2 \sqrt{(1 - \alpha) p (p + \alpha q - \alpha \lambda)}}\biggr) , \\
 \ev^-_{\lambda, l} & = \biggl(\frac{q}{p + \alpha q - \alpha \lambda}\biggr)^{\frac{l + 1}{2}} \sin\biggl((l + 1) \arccos \frac{p + (1 + \alpha) q - \lambda}{2 \sqrt{q (p + \alpha q - \alpha \lambda)}}\biggr) .
}
Finally,
\formula{
 \biggl|\frac{\gamma'(\lambda)}{\gamma(\lambda)}\biggr|^2 & = \frac{(1 - \alpha) p q}{q (p + \alpha q - \alpha \lambda) (4 q (p + \alpha q - \alpha \lambda) - (p + (1 + \alpha) q - \lambda)^2)} \, .
}
The assertions of Theorems~\ref{thm:main} and~\ref{thm:eig} are the same as in the previous example.

In probabilistic terms, when $p + q = 1$, the transpose of $T$ is the generator of the discrete-time risk process, with geometrically distributed claims.


\subsection{Symmetric geometric sequences}

The formulae simplify greatly in the symmetric case, when $\alpha = \beta$ and $p = q$. We have
\formula{
 F(z) & = -\hat a(z) = p \, \frac{1 - z}{1 - \alpha z} + p \, \frac{1 - z^{-1}}{1 - \alpha z^{-1}} \, ,
}
and $\Gamma$ is the unit semicircle. Furthermore,
\formula{
 U(\lambda) = V(\lambda) & = (1 + \alpha) p - \alpha \lambda , &
 W(\lambda) & = 2 (1 + \alpha) p - (1 + \alpha^2) \lambda .
}
The spectrum is the interval
\formula{
 \Lambda = (\lambda_1, \lambda_2) & = \biggl(0, \frac{4 p}{1 + \alpha}\biggr) ,
}
and the generalised eigenvectors and co-eigenvectors are
\formula{
 \ev^\pm_{\lambda, k} & = \sin\biggl((k + 1) \arccos \frac{2 (1 + \alpha) p - (1 + \alpha^2) \lambda}{2 ((1 + \alpha) p - \alpha \lambda)} + \arccos \frac{(2 p - \alpha \lambda) \sqrt{1 + \alpha}}{2 \sqrt{p ((1 + \alpha) p - \alpha \lambda)}}\biggr) .
}
In Theorem~\ref{thm:main}, it is perhaps convenient to substitute $\thet = \thet_\lambda$. This leads to the formula
\formula{
 (P(T))_{k, l} & = \frac{2}{\pi} \int_0^\pi P(-\lambda(\thet)) \ev^\pm_{\lambda(\thet), k} \ev^\pm_{\lambda(\thet), l} \, d\thet
}
for every polynomial $P$, where
\formula{
 \lambda(\thet) & = \frac{2 (1 + \alpha) p (1 + \cos \thet)}{1 + \alpha^2 + 2 \alpha \cos \thet} \, ,
}
and
\formula{
 \ev^\pm_{\lambda(\thet), k} & = \sin\biggl((k + 1) \thet + \arccos \frac{(1 + \alpha \cos \thet)}{\sqrt{1 + \alpha^2 + 2 \alpha \cos \thet}}\biggr) .
}


\subsection{Fisher--Hartwig symbols}

Toeplitz matrices corresponding to $\amcm$ sequences whose symbols have a finite number of Fisher--Hartwig type singularities attracted much attention. A prime example is the symbol:
\formula{
 F(z) & = (1 - z)^\alpha (1 - z^{-1})^\beta ,
}
which turns out to be the symbol of an $\amcm$ sequence when $\alpha, \beta \in (0, 1)$. This follows easily from Proposition~\ref{prop:characterisation}, but we can also describe the corresponding $\amcm$ sequence $(a_k)$ in a more direct way. For $z \in \torus$, we have
\formula{
 F^+(z) = (1 - z)^\alpha & = \sum_{k = 0}^\infty (-1)^k \binom{\alpha}{k} z^k = \sum_{k = 1}^\infty \frac{\alpha (1 - \alpha) (2 - \alpha) \ldots (k - 1 - \alpha)}{k!} \, (1 - z^k) ,
}
and using $(-1)^{k + 1} \tbinom{\alpha}{k + 1} - (-1)^k \tbinom{\alpha}{k} = (-1)^{k + 1} \tbinom{\alpha + 1}{k + 1}$ one can easily verify that the sequence $((-1)^k \tbinom{\alpha}{k} : k \ge 1)$ is completely monotone. The same argument shows that if $z \in \torus$, then
\formula{
 F^-(z) = (1 - z^{-1})^\alpha & = \sum_{l = -\infty}^0 (-1)^l \binom{\beta}{-l} z^l .
}
Therefore, $-F^+$ and $-F^-$ are generating functions of summable one-sided $\amcm$ sequences $(a^+_k)$ and $(a^-_k)$ (where $a^+_k = 0$ for $k < 0$ and $a^-_k = 0$ for $k > 0$). The product $-F = -F^+ F^-$ is therefore the generating function of the sequence $(a_k)$ given by the convolution:
\formula{
 a_k & = -\sum_{l = -\infty}^\infty a^+_{k + l} a^-_{-l} .
}
We claim that $(a_k)$ is an $\amcm$ sequence. Observe that $a^-_{-l} = (-1)^l \tbinom{\beta - 1}{l} - (-1)^{l - 1} \tbinom{\beta - 1}{l - 1}$ (where $\tbinom{\beta - 1}{l} = 0$ if $l < 0$), and use partial summation:
\formula{
 a_k & = -\sum_{l = -\infty}^\infty a^+_{k + l} \biggl((-1)^l \binom{\beta - 1}{l} - (-1)^{l - 1} \binom{\beta - 1}{l - 1}\biggr) \\
 & = \sum_{l = 0}^\infty (a^+_{k + l + 1} - a^+_{k + l}) (-1)^l \binom{\beta - 1}{l} .
}
Since $(-1)^l \binom{\beta - 1}{l} \ge 0$ and $(a^+_{k + 1} - a^+_k : k \ge 1)$ is completely monotone, also $(a_k : k \ge 1)$ is completely monotone. A similar argument shows that $(a_{-k} : k \ge 1)$ is completely monotone, and our claim follows.

The spine $\Gamma$ is characterised by the condition $\im z > 0$ and
\formula{
 0 & = \arg F(z) = \alpha \arg (1 - z) + \beta \arg (1 - z^{-1}) .
}
If $z = r e^{i \thet}$, the above condition reads
\formula{
 \alpha \arccos \frac{1 - r \cos \thet}{1 + r^2 - 2 r \cos \thet} & = \beta \arccos \frac{r (r - \cos \thet)}{1 + r^2 - 2 r \cos \thet} \, .
}
The two sides of the above condition have different monotonicity, and this easily implies that for every $\thet \in (0, \pi)$ there is a unique solution $r$, and so $\Gamma$ winds around $0$. Furthermore, $r > 1$ if $\alpha < \beta$, and $r < 1$ if $\alpha > \beta$. No simple description of the solution $r$ seems to be available, except some special cases. For example, if $\alpha = \beta$, then clearly $r = 1$; this case is discussed in more detail in the next section. On the other hand, numerical approximation of the solution $r$ using, for example, Newton's method, presents no difficulties.

The parameterisation $\gamma(\lambda)$ of the spine is therefore inexplicit, unless, for example, $\alpha = \beta$. However, it is easy to see that the endpoints of $\Gamma$ are $1$ and the solution of $F'(z) = 0$, that is, $z = -\beta / \alpha$. It follows that
\formula{
 \Lambda & = \biggl(0, \frac{(\alpha + \beta)^{\alpha + \beta}}{\alpha^\alpha \beta^\beta}\biggr)
}
Fix $\lambda \in \Lambda$, and, for brevity, write $\gamma = \lv \gamma \rv e^{i \thet}$ instead of $\gamma(\lambda) = \lv \gamma(\lambda) \rv e^{i \thet_\lambda}$. Clearly, for $s \in (0, 1)$,
\formula{
 F(s^{-1} + 0 i) & = e^{-i \alpha \pi} (s^{-1} - 1)^\alpha (1 - s)^\beta = e^{-i \alpha \pi} s^{-\alpha} (1 - s)^{\alpha + \beta} ,
}
and hence
\formula{
 \lv \arg(F(s^{-1} + 0 i) - \lambda) \rv & = \arccot \frac{\cos(\alpha \pi) - \lambda s^\alpha (1 - s)^{-\alpha - \beta}}{\sin(\alpha \pi)} \, .
}
Similarly,
\formula{
 \lv \arg(F(s + 0 i) - \lambda) \rv & = \arccot \frac{\cos(\beta \pi) - \lambda s^\beta (1 - s)^{-\alpha - \beta}}{\sin(\beta \pi)} \, .
}
In particular,
\formula{
 \thet^+_\lambda & = \frac{1}{\pi} \int_0^1 \frac{\lv \gamma \rv \sin \thet}{\lv s \gamma - 1 \rv^2} \, \arccot \frac{\lambda s^\alpha (1 - s)^{-\alpha - \beta} - \cos(\alpha \pi)}{\sin(\alpha \pi)} \, ds , \\
 \thet^-_\lambda & = \frac{1}{\pi} \int_0^1 \frac{\lv \gamma \rv \sin \thet}{\lv s - \gamma \rv^2} \, \arccot \frac{\lambda s^\beta (1 - s)^{-\alpha - \beta} - \cos(\beta \pi)}{\sin(\beta \pi)} \, ds ,
}
and
\formula{
 F^+(\lambda; z) & = c^+_\lambda \exp\biggl(\frac{1}{\pi} \int_0^1 \frac{1}{s - z^{-1}} \, \arccot \frac{\cos(\alpha \pi) - \lambda s^\alpha (1 - s)^{-\alpha - \beta}}{\sin(\alpha \pi)} \, ds \biggr) , \\
 F^-(\lambda; z) & = c^-_\lambda \exp\biggl(\frac{1}{\pi} \int_0^1 \frac{1}{s - z^{-1}} \, \arccot \frac{\cos(\beta \pi) - \lambda s^\beta (1 - s)^{-\alpha - \beta}}{\sin(\beta \pi)} \, ds \biggr) .
}
Here $c^+_\lambda$ and $c^-_\lambda$ are constants such that $F^+(\lambda; -1) F^-(\lambda; -1) = F(\lambda; -1)$.

We have
\formula{
 F(\lambda; z) & = \frac{\lv \gamma \rv (z - 1)^2}{\lv \gamma \rv^2 + z^2 - 2 \lv \gamma \rv z \cos \thet} \, ((1 - z)^\alpha (1 - z^{-1})^\beta - \lambda) .
}
It follows that for $s \in (0, 1)$,
\formula{
 \biggl|\im \frac{1}{F(\lambda; s^{-1} + 0 i)}\biggr| & = \frac{\lv \gamma \rv^2 s^2 + 1 - 2 \lv \gamma \rv s \cos \thet}{\lv \gamma \rv (1 - s)^2} \, \frac{s^\alpha (1 - s)^{-\alpha - \beta} \sin(\pi \alpha)}{\lv e^{i \pi \alpha} - s^\alpha (1 - s)^{-\alpha - \beta} \lambda \rv^2} \, , \\
 \biggl|\im \frac{1}{F(\lambda; s + 0 i)}\biggr| & = \frac{\lv \gamma \rv^2 + s^2 - 2 \lv \gamma \rv s \cos \thet}{\lv \gamma \rv (1 - s)^2} \, \frac{s^\beta (1 - s)^{-\alpha - \beta} \sin(\pi \beta)}{\lv e^{i \pi \beta} - s^\beta (1 - s)^{-\alpha - \beta} \lambda \rv^2} \, .
}
With the above definitions, $\ev^+_{\lambda, k}$ and $\ev^-_{\lambda, l}$ are defined as in the introduction; formulae~\eqref{eq:i:psi:p} and~\eqref{eq:i:psi:m} for $\evr^+_{\lambda, k}$ and $\evr^-_{\lambda, l}$ do not seem to simplify in any way, but they are numerically tractable.

By Theorem~\ref{thm:eig}, $\ev^+_\lambda$ is an eigenvector of $T$ if $\alpha \le \beta$, and $\ev^-_\lambda$ is a co-eigenvector if $\alpha \ge \beta$. Theorem~\ref{thm:main} provides a generalised eigenvector expansion of $P(T)$ for an arbitrary polynomial~$P$.


\subsection{Symmetric Fisher--Hartwig symbols}

The expressions in the previous example simplify significantly when $\alpha = \beta$. We have already noted that in this case $\lv \gamma(\lambda) \rv = 1$, that is, $\gamma(\lambda) = e^{i \thet_\lambda}$. We simply have
\formula{
 \Lambda & = (0, 2^{2 \alpha}) ,
}
and since $F(e^{i \thet}) = (2 - 2 \cos \thet)^\alpha = (2 \sin(\tfrac{1}{2} \thet))^{2 \alpha}$, we find that $\gamma(\lambda) = e^{i \thet_\lambda}$ with
\formula{
 \thet_\lambda & = \arccos(1 - \tfrac{1}{2} \lambda^{1 / \alpha}) .
}
Furthermore,
\formula{
 \thet^\pm_\lambda & = \frac{1}{2 \pi} \int_0^1 \frac{\lambda^{1 / (2 \alpha)} \sqrt{4 - \lambda^{1 / \alpha}}}{(1 - s)^2 + \lambda^{1 / \alpha} s} \, \arccot \frac{\lambda s^\alpha (1 - s)^{-2 \alpha} - \cos(\alpha \pi)}{\sin(\alpha \pi)} \, ds .
}
We also have
\formula{
 F^\pm(\lambda; z) & = c^\pm_\lambda \exp\biggl(\frac{1}{\pi} \int_0^1 \frac{1}{s - z^{-1}} \, \arccot \frac{\cos(\alpha \pi) - \lambda s^\alpha (1 - s)^{-2 \alpha}}{\sin(\alpha \pi)} \, ds \biggr) ,
}
where $c^\pm_\lambda$ is such that $(F^\pm(\lambda; -1))^2 = F(\lambda; -1) = 4 (4^\alpha - \lambda) / (4 - \lambda^{1 / \alpha})$. For $s \in (0, 1)$,
\formula{
 \biggl|\im \frac{1}{F(\lambda; s^{-1} + 0 i)}\biggr| = \biggl|\im \frac{1}{F(\lambda; s + 0 i)}\biggr| & = \frac{(1 - s)^2 + \lambda^{1 / \alpha} s}{(1 - s)^2} \, \frac{s^\alpha (1 - s)^{-2 \alpha} \sin(\pi \alpha)}{\lv e^{i \pi \alpha} - s^\alpha (1 - s)^{-2 \alpha} \lambda \rv^2} \, .
}
Therefore,
\formula{
 \evr^\pm_{\lambda, k} & = \frac{1}{\pi} \, \frac{\lv F^+(\lambda; \gamma) \rv \im \gamma}{\lv 1 - \gamma \rv} \int_0^1 s^k \, \frac{(1 - s) F^-(\lambda; s)}{\lv \gamma s - 1 \rv^2} \, \frac{(1 - s)^2 + \lambda^{1 / \alpha} s}{(1 - s)^2} \, \frac{s^\alpha (1 - s)^{-2 \alpha} \sin(\pi \alpha)}{\lv e^{i \pi \alpha} - s^\alpha (1 - s)^{-2 \alpha} \lambda \rv^2} \, ds \\
 & = \frac{1}{\pi} \, \frac{\lv F^+(\lambda; \gamma) \rv \im \gamma}{\lv 1 - \gamma \rv} \int_0^1 s^k \, \frac{F^-(\lambda; s)}{1 - s} \, \frac{s^\alpha (1 - s)^{-2 \alpha} \sin(\pi \alpha)}{\lv e^{i \pi \alpha} - s^\alpha (1 - s)^{-2 \alpha} \lambda \rv^2} \, ds ,
}
and
\formula{
 \ev^\pm_{\lambda, k} & = \sin\bigl((k + 1) \thet_\lambda + \thet^+_\lambda\bigr) - \evr^+_{\lambda, k} .
}
Finally,
\formula{
 \lv \gamma'(\lambda) \rv & = \frac{1}{\alpha \lambda^{1 - 1 / 2 \alpha} \sqrt{4 - \lambda^{1 / \alpha}}} \, .
}
Theorem~\ref{thm:main} states that for every polynomial $P$,
\formula{
 (P(T))_{k, l} & = \frac{2}{\pi} \int_0^{2^{2 \alpha}} P(-\lambda) \ev^\pm_{\lambda, k} \ev^\pm_{\lambda, l} \, \frac{1}{\alpha \lambda^{1 - 1 / 2 \alpha} \sqrt{4 - \lambda^{1 / \alpha}}} \, d\lambda .
}
It may be convenient to substitute $\thet = \thet_\lambda$ in the above integral:
\formula{
 (P(T))_{k, l} & = \frac{2}{\pi} \int_0^\pi P(-(2 - 2 \cos \thet)^\alpha) \ev^\pm_{(2 - 2 \cos \thet)^\alpha, k} \ev^\pm_{(2 - 2 \cos \thet)^\alpha, l} d\thet .
}
Furthermore, by Theorem~\ref{thm:eig}, $\ev^\pm_\lambda$ is an eigenvector of $T$: we have $T \ev^\pm_\lambda = -\lambda \ev^\pm_\lambda$.

%
%

\newcommand{\nist}[2]{\href{https://dlmf.nist.gov/#1\#E#2}{Eq.~#1.#2}}
\newcommand{\doi}[1]{\href{https://doi.org/#1}{\textsf{\scriptsize DOI:#1}}}
\newcommand{\arxiv}[1]{\href{https://arxiv.org/abs/#1}{\textsf{\scriptsize arXiv:#1}}}
\newcommand{\isbn}[1]{\textsf{ISBN:#1}}

%
%

\end{document}